\documentclass[a4j,10pt,showkeys]{article}

\usepackage{amsmath,amsthm,amssymb}
\usepackage{geometry}
\usepackage{hyperref}
\usepackage{cancel}
\usepackage{mathtools}
\usepackage{bm}
\usepackage{color}
\usepackage{mathrsfs}

\newtheorem{theo}{Theorem}[section]
\newtheorem{lemm}[theo]{Lemma}
\newtheorem{prop}[theo]{Proposition}

\theoremstyle{definition}
\newtheorem{defi}[theo]{Definition}
\newtheorem{rem}[theo]{Remark}
\newtheorem{exam}[theo]{Example}

\newtheorem{assum}{Assumption}

\newcommand{\bP}{\mathbb{P}}
\newcommand{\bE}{\mathbb{E}}
\newcommand{\bF}{\mathbb{F}}
\newcommand{\bR}{\mathbb{R}}
\newcommand{\bN}{\mathbb{N}}
\newcommand{\bX}{\mathbb{X}}
\newcommand{\bY}{\mathbb{Y}}
\newcommand{\bZ}{\mathbb{Z}}
\newcommand{\cF}{\mathcal{F}}

\newcommand{\cB}{\mathcal{B}}
\newcommand{\cM}{\mathcal{M}}
\newcommand{\cU}{\mathcal{U}}
\newcommand{\cY}{\mathcal{Y}}
\newcommand{\cZ}{\mathcal{Z}}
\newcommand{\cR}{\mathcal{R}}
\newcommand{\sL}{\mathscr{L}}
\newcommand{\rd}{\,\mathrm{d}}

\newcommand{\relmiddle}[1]{\mathrel{}\middle#1\mathrel{}}
\newcommand{\1}{\mbox{\rm{1}}\hspace{-0.25em}\mbox{\rm{l}}}
\newcommand{\comp}{\mathrm{c}}
\newcommand{\CD}{^\mathrm{C}\!D^\alpha_{0+}}

\providecommand{\keywords}[1]{\textbf{Keywords:} #1}

\allowdisplaybreaks[0]

\makeatletter

\@addtoreset{equation}{section}
\makeatother
\def\widebar{\accentset{{\cc@style\underline{\mskip10mu}}}}
\numberwithin{equation}{section}
\def\rnum#1{\expandafter{\romannumeral #1}} 
\def\Rnum#1{\uppercase\expandafter{\romannumeral #1}}

\allowdisplaybreaks

\title{Infinite horizon backward stochastic Volterra integral equations and discounted control problems}
\author{
	Yushi Hamaguchi\thanks{
	Graduate School of Engineering Science, Department of Systems Innovation,
	Osaka University.
	1-3, Machikaneyama, Toyonaka, Osaka, Japan.
	\href{mailto:hmgch2950@gmail.com}{hmgch2950@gmail.com}}
}

\begin{document}
\maketitle

\begin{abstract}
Infinite horizon backward stochastic Volterra integral equations (BSVIEs for short) are investigated. We prove the existence and uniqueness of the adapted M-solution in a weighted $L^2$-space. Furthermore, we extend some important known results for finite horizon BSVIEs to the infinite horizon setting. We provide a variation of constant formula for a class of infinite horizon linear BSVIEs and prove a duality principle between a linear (forward) stochastic Volterra integral equation (SVIE for short) and an infinite horizon linear BSVIE in a weighted $L^2$-space. As an application, we investigate infinite horizon stochastic control problems for SVIEs with discounted cost functional. We establish both necessary and sufficient conditions for optimality by means of Pontryagin's maximum principle, where the adjoint equation is described as an infinite horizon BSVIE. These results are applied to discounted control problems for fractional stochastic differential equations and stochastic integro-differential equations.
\end{abstract}
\keywords
Infinite horizon backward stochastic Volterra integral equation;
stochastic Volterra integral equation;
duality principle;
discounted stochastic control.






\section{Introduction}


Backward stochastic Volterra integral equations (BSVIEs for short) are Volterra-type extensions of backward stochastic differential equations (BSDEs for short). A class of BSVIEs introduced by Lin~\cite{Li02} and Yong~\cite{Yo06,Yo08} is the following:
\begin{equation}\label{finite horizon Type-I}
	Y(t)=\psi(t)+\int^T_tg(t,s,Y(s),Z(t,s))\rd s-\int^T_tZ(t,s)\rd W(s),\ t\in[0,T],
\end{equation}
where $T\in(0,\infty)$ is a given terminal time, $W(\cdot)$ is a Brownian motion, $\psi(\cdot)$ is a given (not necessarily adapted) stochastic process called the \emph{free term}, and $g$ is a given map called the \emph{driver}. The unknown we are looking for is the pair $(Y(\cdot),Z(\cdot,\cdot))$, where $Y(\cdot)=(Y(s))_{s\in[0,T]}$ and $Z(t,\cdot)=(Z(t,s))_{s\in[t,T]}$ are adapted for each $t\in[0,T]$. Sometimes a BSVIE of the form \eqref{finite horizon Type-I} is called a \emph{Type-I} BSVIE. If the free term $\psi(\cdot)$ and the driver $g$ of Type-I BSVIE~\eqref{finite horizon Type-I} does not depend on $t$, then the second term $Z(t,s)$ of the solution becomes independent of $t$, and \eqref{finite horizon Type-I} becomes the standard BSDE:
\begin{equation*}
	Y(t)=\psi+\int^T_tg(s,Y(s),Z(s))\rd s-\int^T_tZ(s)\rd W(s),\ t\in[0,T],
\end{equation*}
which can be written in the differential form:
\begin{equation*}
	\begin{cases}
	\mathrm{d}Y(t)=-g(t,Y(t),Z(t))\rd t+Z(t)\rd W(t),\ t\in[0,T],\\
	Y(T)=\psi.
	\end{cases}
\end{equation*}
BSDEs have very interesting applications in stochastic control, mathematical finance, PDE theory, and many other fields. Relevant theory of BSDEs can be found in the textbook of Zhang \cite{Zh17} and references cited therein. Type-I BSVIEs are important tools to study stochastic control and mathematical finance taking into account the time-inconsistency. For example, time-inconsistent dynamic risk measures and time-inconsistent recursive utilities can be modelled by the solutions to Type-I BSVIEs (see \cite{Yo07,KrOv17,Ag19,WaSuYo19}). Time-inconsistent stochastic control problems related to Type-I BSVIEs were studied by Wang and Yong~\cite{WaYo21}, Hern\'{a}ndez and Possama\"{i}~\cite{HePo20}, and the author~\cite{Ha21}. Also, Beissner and Rosazza Gianin \cite{BeRG21} applied Type-I BSVIEs to arbitrage-free asset pricing via a path of EMMs called an EMM-string.

As a further extension of Type-I BSVIEs, Yong~\cite{Yo06,Yo08} considered the following equation:
\begin{equation}\label{finite horizon Type-II}
	Y(t)=\psi(t)+\int^T_tg(t,s,Y(s),Z(t,s),Z(s,t))\rd s-\int^T_tZ(t,s)\rd W(s),\ t\in[0,T],
\end{equation}
which is called a \emph{Type-II} BSVIE. Unlike Type-I BSVIE~\eqref{finite horizon Type-I}, the driver $g$ of Type-II BSVIE~\eqref{finite horizon Type-II} depends on the term $Z(s,t)$ for $0\leq t\leq s\leq T$, which is a part of the solution. Therefore, to obtain a solution of Type-II BSVIE~\eqref{finite horizon Type-II}, one has to determine $Z(t,s)$ for $(t,s)\in[0,T]^2$. Since the equation \eqref{finite horizon Type-II} alone is not a closed form for $Y(\cdot)=(Y(t))_{t\in[0,T]}$ and $Z(\cdot,\cdot)=(Z(t,s))_{(t,s)\in[0,T]^2}$, we have to impose an additional constraint on the term $Z(t,s)$, $0\leq s\leq t\leq T$, for the well-posedness of the equation. The so-called \emph{adapted M-solution} is a solution concept of Type-II BEVIE~\eqref{finite horizon Type-II} which, by the definition, determines the values of $Z(t,s)$ for $0\leq s\leq t\leq T$ according to the martingale representation theorem $Y(t)=\bE[Y(t)]+\int^t_0Z(t,s)\rd W(s)$. This kind of solution concept was introduced by Yong~\cite{Yo08} and applied to the duality principle appearing in stochastic control problems for (forward) stochastic Volterra integral equations (SVIEs for short). Since then, Type-II BSVIEs have been found important to study stochastic control problems of SVIEs (see \cite{ChYo07,ShWaYo13,ShWaYo15,WaTZh17,WaT20}).

We emphasize that, in all the previous works mentioned above, BSVIEs and the related stochastic control problems are considered on a finite horizon $[0,T]$ with a given (deterministic) terminal time $T\in(0,\infty)$. In contrast, in this paper, we are interested in BSVIEs and stochastic control problems with $T=\infty$. More precisely, we consider an \emph{infinite horizon BSVIE} of the following form:
\begin{equation}\label{Type-II}
	Y(t)=\psi(t)+\int^\infty_te^{-\lambda(s-t)}g(t,s,Y(s),Z(t,s),Z(s,t))\rd s-\int^\infty_tZ(t,s)\rd W(s),\ t\geq0,
\end{equation}
where $\lambda\in\bR$ represents a given (not necessarily positive) ``\emph{discount rate}'' which determines the long-time behavior of the solution. This form of BSVIE appears for the first time in the literature. As in the finite horizon case, we often call equation~\eqref{Type-II} an infinite horizon Type-II BSVIE and the equation~\eqref{Type-II} with the driver $g(t,s,y,z_1,z_2)$ being independent of $z_2$ an infinite horizon Type-I BSVIE. Infinite horizon BSVIE~\eqref{Type-II} can be seen as an extension of an infinite horizon BSDE of the following form:
\begin{equation}\label{infinite horizon BSDE}
	\mathrm{d}Y(t)=-\bigl\{g(t,Y(t),Z(t))-\lambda Y(t)\bigr\}\rd t+Z(t)\rd W(t),\ t\geq0,
\end{equation}
which was studied in, for example, \cite{Pa99,PeSh00,FuTe04,FuHu06,Yi08,WuZh12}. Compared with the finite horizon case, in the infinite horizon BSDE~\eqref{infinite horizon BSDE}, the terminal condition for $Y(\cdot)$ is replaced by a (given) growth condition at infinity. The main difficulty appearing in the infinite horizon case is to identify the condition for the discount rate $\lambda$ which ensures the existence and uniqueness of the solution in a given weighted space. A similar difficulty also arises in the study of infinite horizon BSVIE~\eqref{Type-II}. Specifically, consider the situation where we are given a free term $\psi(\cdot)$ satisfying
\begin{equation*}
	\bE\Bigl[\int^\infty_0e^{2\eta t}|\psi(t)|^2\rd t\Bigr]<\infty
\end{equation*}
for some fixed weight $\eta\in\bR$. Then the adapted M-solution $(Y(\cdot),Z(\cdot,\cdot))$ to the infinite horizon BSVIE~\eqref{Type-II} should satisfy a similar growth condition at infinity:
\begin{equation*}
	\bE\Bigl[\int^\infty_0e^{2\eta t}|Y(t)|^2\rd t+\int^\infty_0e^{2\eta t}\int^\infty_0|Z(t,s)|^2\rd s\rd t\Bigr]<\infty.
\end{equation*}
A delicate question is the condition which ensures the existence and uniqueness of the adapted M-solution for \emph{any} free term $\psi(\cdot)$ in a \emph{given} weighted $L^2$-space. In this paper, we mainly focus on this problem. Intuitively speaking, in order to ensure the well-posedness in a large space (i.e., the weight $\eta\in\bR$ is a large negative number), the driver $g$ has to be discounted sufficiently at infinity (i.e., the discount rate $\lambda\in\bR$ should be a sufficiently large positive number). A criterion for this trade-off relationship between the weight $\eta$ and the discount rate $\lambda$ is determined by the driver $g$.

With the above intuitive discussions in mind, we prove the well-posedness of infinite horizon BSVIE~\eqref{Type-II} (see Theorem~\ref{theo: well-posedness BSVIE}). We also show that the adapted M-solutions of finite horizon BSVIEs converge to the adapted M-solution of the original infinite horizon BSVIE (see Theorem~\ref{theo: finite horizon}). Furthermore, we extend some important known results for finite horizon BSVIEs to the infinite horizon setting. We provide a variation of constant formula for an infinite horizon linear Type-I BSVIE (see Theorem~\ref{theo: explicit}), which is an extension of the results of Hu and {\O}ksendal~\cite{HuOk19} and  Wang, Yong, and Zhang~\cite{WaYoZh20}. We also show a duality principle between a linear SVIE and an infinite horizon linear Type-II BSVIE in a weighted $L^2$-space  (see Theorem~\ref{theo: duality}), which is an extension of that of Yong~\cite{Yo08}. We remark that our results mentioned above are not only extensions of the corresponding known results to infinite horizon cases, but also generalizations to unbounded coefficients cases. Indeed, in \cite{HuOk19,WaYoZh20} and \cite{Yo08}, they considered linear (finite horizon) BSVIEs with bounded coefficients, while in this paper we treat linear infinite horizon BSVIEs with unbounded coefficients. These generalizations are important to apply the results to stochastic control problems for a forward SVIEs with singular coefficients.

As an application of infinite horizon BSVIEs, we investigate an infinite horizon optimal control problem for a forward SVIE. There are many applications of SVIEs in different research areas concerned with the effects of memory and delay, such as stochastic production management models (see \cite{WaTZh17}), epidemic models (see \cite{Cr85,PaPa20}), and rough volatility models in mathematical finance (see \cite{GaJaRo18}). In this paper, we define the state process $X^u(\cdot)$ corresponding to a given control process $u(\cdot)$ by the solution to a controlled SVIE of the form
\begin{equation*}
	X^u(t)=\varphi(t)+\int^t_0b(t,s,X^u(s),u(s))\rd s+\int^t_0\sigma(t,s,X^u(s),u(s))\rd W(s),\ t\geq0.
\end{equation*}
The objective is to minimize a discounted cost functional
\begin{equation*}
	J_\lambda(u(\cdot)):=\bE\Bigl[\int^\infty_0e^{-\lambda t}h(t,X^u(t),u(t))\rd t\Bigr]
\end{equation*}
over all admissible control processes $u(\cdot)$. The coefficients $b$ and $\sigma$ are allowed to be singular at the diagonal $t=s$. The state process $X^u(\cdot)$ is in general neither Markovian nor semimartingales and includes the fractional Brownian motion with Hurst parameter smaller than $1/2$ as a special case. In this framework, by using the duality principle mentioned above, we provide a necessary condition for optimality by means of Pontryagin's maximum principle (see Theorem~\ref{theo: maximum principle}). The corresponding adjoint equation can be expressed as an infinite horizon linear Type-II BSVIE with unbounded coefficients. We also provide sufficient conditions for optimality in terms of the related Hamiltonian (see Theorem~\ref{theo: sufficient maximum principle}). These results can be applied to infinite horizon stochastic control problems of linear-quadratic types (see Example~\ref{exam: LQ}), fractional stochastic differential equations (see Example~\ref{exam: FSDE}), and stochastic integro-differential equations (see Example~\ref{exam: integro-differential}).

Let us further discuss some novelties of this paper compared with the related works. As we mentioned above, this paper is the first to investigate BSVIEs on the infinite horizon. Furthermore, our method to show the well-posedness is completely different from the ones in the literature of (finite horizon) BSVIEs. A main technique used in the literature is a fixed point argument based on some estimates for auxiliary BSDE systems (see, for example, \cite{Yo08}). In contrast, our method to show the well-posedness of infinite horizon BSVIEs does not rely on the theory of BSDEs or It\^o's calculus. Instead, we first show a priori estimates for the adapted M-solutions by some very simple calculations and then prove the well-posedness by using the so-called \emph{method of continuation} which is known as a fundamental technique in the theory of forward-backward SDEs (see the textbook of Ma and Yong~\cite{MaYo99}).

Infinite horizon stochastic control problems for stochastic differential equations (SDEs for short) with discounted cost functional have been studied in, for example, \cite{FuTe04,MaVe14,OrVe17}. In these papers, infinite horizon BSDEs play central roles in different ways. Fuhrman and Tessitore~\cite{FuTe04} investigated a relationship between an infinite horizon BSDE and an elliptic Hamilton--Jacobi--Bellman equation to obtain the optimal control. Maslowski and Veverka~\cite{MaVe14} obtained a sufficient maximum principle, and Orrieri and Veverka~\cite{OrVe17} established a necessary maximum principle for an SDE with dissipative coefficients. Our results are extensions of \cite{MaVe14,OrVe17} to the Volterra setting. Indeed, in the SDEs setting, the adjoint equation which we obtain in the general Volterra setting is reduced to the infinite horizon BSDE considered in the aforementioned papers (see Example~\ref{exam: SDE}).

Concerned with stochastic control problems for SVIEs, Yong~\cite{Yo08} and Wang~\cite{WaT20} obtained a necessary condition for optimality by means of the stochastic maximum principle, where the (first order) adjoint equation can be described as a finite horizon BSVIE with bounded coefficients. Shi, Wang, and Yong~\cite{ShWaYo13} considered a mean-field type stochastic control problem and obtained a necessary condition by using a mean-field BSVIE. Agram and {\O}ksendal~\cite{AgOk15} proposed a Malliavin calculus technique to obtain the optimal control. Chen and Yong~\cite{ChYo07} and Wang~\cite{WaT18} studied linear-quadratic control problems for SVIEs. The aforementioned papers are restricted to finite horizon cases. Also, since some kinds of regularity are required for the coefficients, they exclude the case of singular coefficients such as the fractional Brownian motion with Hurst parameter $H<1/2$. Concerned with the later problem, Abi Jaber, Miller, and Pham~\cite{AbMiPh19} considered a finite horizon linear-quadratic control problem for an SVIE of convolution type, which includes a singular coefficients case. They used a technique of the so-called Markovian lift of a stochastic Volterra process and obtained the optimal feedback control, depending on non-standard Banach space-valued Riccati equations. Their results heavily rely on the special structure of the state process and the cost functional. Compared with the above papers, we consider an infinite horizon stochastic control problem with a discounted cost functional of a general form including the quadratic cost, where the controlled dynamics of the state process is described as an SVIE with nonlinear, non-convolution type, and singular coefficients. Therefore, apart from the extension to the infinite horizon case, our results give a new insight at this general setting.

This paper is organized as follows. In Section~\ref{section: SVIE}, we briefly discuss the well-posedness of (forward) SVIEs with singular coefficients. In Section~\ref{section: BSVIE}, we prove the well-posedness of infinite horizon BSVIEs. We also provide a convergence result for finite horizon BSVIEs, a variation of constant formula for an infinite horizon linear Type-I BSVIE, and a duality principle between a linear SVIE and an infinite horizon linear Type-II BSVIE in a weighted $L^2$-space. In Section~\ref{section: control}, as an application of infinite horizon BSVIEs, we study an infinite horizon stochastic control problem for an SVIE with singular coefficients. Some examples including controlled fractional stochastic differential equations and controlled stochastic integro-differential equations are presented in Section~\ref{subsection: example}.


\subsection*{Notation}


For each $d_1,d_2\in\bN$, we denote the space of $(d_1\times d_2)$-matrices by $\bR^{d_1\times d_2}$, which is endowed with the Frobenius norm denoted by $|\cdot|$. We define $\bR^{d_1}:=\bR^{d_1 \times 1}$, that is, each element of $\bR^{d_1}$ is understood as a column vector. We denote by $\langle\cdot,\cdot\rangle$ the usual inner product in a Euclidean space. For each  matrix $A$, $A^\top$ denotes the transpose of $A$.

$(\Omega,\cF,\bP)$ is a complete probability space and $W(\cdot)=(W_1(\cdot),\dots,W_d(\cdot))^\top$ is a $d$-dimensional Brownian motion on $(\Omega,\cF,\bP)$ with $d\in\bN$. $\bF=(\cF_t)_{t\geq0}$ denotes the filtration generated by $W(\cdot)$ and augmented by $\bP$. We define $\cF_\infty:=\bigvee_{t\geq0}\cF_t$. For each $t\geq0$, $\bE_t[\cdot]:=\bE[\cdot|\cF_t]$ is the conditional expectation given by $\cF_t$.

We define
\begin{equation*}
	\Delta[0,\infty):=\{(t,s)\in[0,\infty)^2\,|\,0\leq t\leq s<\infty\}\ \text{and}\ \Delta^\comp[0,\infty):=\{(t,s)\in[0,\infty)^2\,|\,0\leq s\leq t<\infty\}.
\end{equation*}

Let $U$ be a nonempty Borel subset of a Euclidean space, and let $p\in[1,\infty)$ and $\beta\in\bR$ be fixed. $L^p_{\cF_\infty}(\Omega;U)$ denotes the set of $U$-valued and $\cF_\infty$-measurable $L^p$-random variables. Define
\begin{equation*}
	L^{p,\beta}(0,\infty;U):=\Bigl\{f:[0,\infty)\to U\,|\,\text{$f$ is measurable},\ \int^\infty_0e^{p\beta t}|f(t)|^p\rd t<\infty\Bigr\}.
\end{equation*}
Note that $L^{p',\beta'}(0,\infty;U)\subset L^{p,\beta}(0,\infty;U)$ if $p\leq p'$ and $\beta<\beta'$. However, $p<p'$ does not imply $L^{p',\beta}(0,\infty;U)\subset L^{p,\beta}(0,\infty;U)$. We define
\begin{equation*}
	L^p(0,\infty;U):=L^{p,0}(0,\infty;U)\ \text{and}\ L^{p,*}(0,\infty;U):=\bigcup_{\beta\in\bR}L^{p,\beta}(0,\infty;U).
\end{equation*}
Also, we define
\begin{equation*}
	L^{p,\beta}_{\cF_\infty}(0,\infty;U):=\Bigl\{f:\Omega\times[0,\infty)\to U\,|\,\text{$f$ is $\cF_\infty\otimes\cB([0,\infty))$-measurable},\ \bE\Bigl[\int^\infty_0e^{p\beta t}|f(t)|^p\rd t\Bigr]<\infty\Bigr\},
\end{equation*}
\begin{equation*}
	L^{p,\beta}_\bF(0,\infty;U):=\{f(\cdot)\in L^{p,\beta}_{\cF_\infty}(0,\infty;U)\,|\,f(\cdot)\ \text{is adapted}\},
\end{equation*}
and
\begin{equation*}
	\sL^{p,\beta}_\bF(0,\infty;U):=\left\{f:\Omega\times[0,\infty)^2\to U\relmiddle|\begin{aligned}&\text{$f$ is $\cF\otimes\cB([0,\infty)^2)$-measurable},\\&\text{$f(t,\cdot)$ is adapted for a.e.\ $t\geq0$},\\&\bE\Bigl[\int^\infty_0e^{p\beta t}\int^\infty_0|f(t,s)|^p\rd s\rd t\Bigr]<\infty\end{aligned}\right\}.
\end{equation*}
If $U$ is a Euclidean space, these spaces are Banach spaces with the norms, for example,
\begin{equation*}
	L^{p,\beta}_{\cF_\infty}(0,\infty;U)\ni f(\cdot)\mapsto \bE\Bigl[\int^\infty_0e^{p\beta t}|f(t)|^p\rd t\Bigr]^{1/p}
\end{equation*}
and
\begin{equation*}
	\sL^{p,\beta}_\bF(0,\infty;U)\ni f(\cdot,\cdot)\mapsto\bE\Bigl[\int^\infty_0e^{p\beta t}\int^\infty_0|f(t,s)|^p\rd s\rd t\Bigr]^{1/p}.
\end{equation*}
$L^p_{\cF_\infty}(0,\infty;U)$, $L^p_\bF(0,\infty;U)$, $\sL^p_\bF(0,\infty;U)$, $L^{p,*}_{\cF_\infty}(0,\infty;U)$, $L^{p,*}_\bF(0,\infty;U)$, and $\sL^{p,*}_\bF(0,\infty;U)$ are defined by similar manners as before.

For each $z(\cdot)\in L^2_\bF(0,\infty;\bR^{d_1\times d})$ with $d_1\in\bN$, the stochastic integral of $z(\cdot)$ with respect to the $d$-dimensional Brownian motion $W(\cdot)$ is defined by
\begin{equation*}
	\int^t_0z(s)\rd W(s):=\sum^d_{k=1}\int^t_0z^k(s)\rd W_k(s),\ t\geq0,
\end{equation*}
which is a $d_1$-dimensional square-integrable martingale. Here and elsewhere, for each $z\in\bR^{d_1\times d}$, $z^k\in\bR^{d_1}$ denotes the $k$-th column vector for $k=1,\dots,d$.

We denote by $\cM^{2,\beta}_\bF(0,\infty;\bR^{d_1}\times\bR^{d_1\times d})$ with $d_1\in\bN$ and $\beta\in\bR$ the set of pairs $(y(\cdot),z(\cdot,\cdot))\in L^{2,\beta}_\bF(0,\infty;\bR^{d_1})\times\sL^{2,\beta}_\bF(0,\infty;\bR^{d_1\times d})$ such that
\begin{equation*}
	y(t)=\bE[y(t)]+\int^t_0z(t,s)\rd W(s)
\end{equation*}
for a.e.\ $t\geq0$, a.s. Clearly, the space $\cM^{2,\beta}_\bF(0,\infty;\bR^{d_1}\times\bR^{d_1\times d})$ is a closed subspace of the product space $L^{2,\beta}_\bF(0,\infty;\bR^{d_1})\times\sL^{2,\beta}_\bF(0,\infty;\bR^{d_1\times d})$, and thus it is a Hilbert space. As before, we define $\cM^2_\bF(0,\infty;\bR^{d_1}\times\bR^{d_1\times d}):=\cM^{2,0}_\bF(0,\infty;\bR^{d_1}\times\bR^{d_1\times d})$ and $\cM^{2,*}_\bF(0,\infty;\bR^{d_1}\times\bR^{d_1\times d}):=\bigcup_{\beta\in\bR}\cM^{2,\beta}_\bF(0,\infty;\bR^{d_1}\times\bR^{d_1\times d})$.

Lastly, for each $f\in L^{p,*}(0,\infty;\bR_+)$, define $[f]_p:\bR\to[0,\infty]$ by
\begin{equation*}
	[f]_p(\rho):=\Bigl(\int^\infty_0e^{-p\rho t}f(t)^p\rd t\Bigr)^{1/p},\ \rho\in\bR.
\end{equation*}
Note that the function $[f]_p$ is non-increasing and satisfies $\lim_{\rho\to\infty}[f]_p(\rho)=0$.


\section{Forward SVIEs with singular coefficients}\label{section: SVIE}


Before studying infinite horizon BSVIEs, we investigate the following (forward) SVIE:
\begin{equation}\label{SVIE}
	X(t)=\varphi(t)+\int^t_0b(t,s,X(s))\rd s+\int^t_0\sigma(t,s,X(s))\rd W(s),\ t\geq0,
\end{equation}
where $\varphi(\cdot)\in L^{2,*}_\bF(0,\infty;\bR^n)$ is a given $n$-dimensional adapted process with $n\in\bN$, and $b,\sigma$ are given maps. We impose the following assumptions on the coefficients $b$ and $\sigma$:


\begin{assum}\label{assum: SVIE}
\begin{itemize}
\item[(i)]
$b:\Omega\times\Delta^\comp[0,\infty)\times\bR^n\to\bR^n$ and $\sigma:\Omega\times\Delta^\comp[0,\infty)\times\bR^n\to\bR^{n\times d}$ are measurable.
\item[(ii)]
$(b(t,s,x))_{s\in[0,t]}$ and $(\sigma(t,s,x))_{s\in[0,t]}$ are adapted for each $t\in[0,\infty)$ and $x\in\bR^n$.
\item[(iii)]
There exist $K_b\in L^{1,*}(0,\infty;\bR_+)$ and $K_\sigma\in L^{2,*}(0,\infty;\bR_+)$ such that
\begin{equation*}
	|b(t,s,x)-b(t,s,x')|\leq K_b(t-s)|x-x'|\ \text{and}\ |\sigma(t,s,x)-\sigma(t,s,x')|\leq K_\sigma(t-s)|x-x'|
\end{equation*}
for any $x,x'\in\bR^n$, for a.e.\ $(t,s)\in\Delta^\comp[0,\infty)$, a.s.
\item[(iv)]
$b(t,s,0)=0$ and $\sigma(t,s,0)=0$ for a.e.\ $(t,s)\in\Delta^\comp[0,\infty)$, a.s.
\end{itemize}
\end{assum}


\begin{rem}
We note that $b$ and $\sigma$ can be singular at the diagonal $t=s$ in the sense that the Lipschitz coefficients $K_b(\tau)$ and $K_\sigma(\tau)$ allows to diverge as $\tau\downarrow0$. Also, they are allowed to diverge as $\tau\uparrow\infty$. For example, we can choose $b(t,s,x)=K(t-s)b(x)$ and $\sigma(t,s,x)=K(t-s)\sigma(x)$ for some Lipschitz continuous functions $b(x)$ and $\sigma(x)$ and the fractional kernel $K(\tau)=\tau^{\alpha-1}$ with $\alpha\in(1/2,1)$.
\end{rem}

We say that a process $X(\cdot)$ is a solution to SVIE~\eqref{SVIE} if $X(\cdot)$ is in $L^{2,*}_\bF(0,\infty;\bR^n)$ and satisfies the equality for a.e.\ $t\geq0$, a.s. A delicate problem is the identification of the weight of the space where the solution uniquely exists. The following lemma is useful for the analysis of SVIEs.


\begin{lemm}\label{lemm: coefficients SVIE}
Suppose that Assumption~\ref{assum: SVIE} holds. Fix $\mu\in\bR$ with $[K_b]_1(\mu)+[K_\sigma]_2(\mu)<\infty$. Then for any $x(\cdot),x'(\cdot)\in L^{2,-\mu}_\bF(0,\infty;\bR^n)$, it holds that
\begin{equation}\label{lemm: coefficients SVIE: 1}
	\bE\Bigl[\int^\infty_0e^{-2\mu t}\Bigl(\int^t_0\bigl|b(t,s,x(s))-b(t,s,x'(s))\bigr|\rd s\Bigr)^2\rd t\Bigr]^{1/2}\leq[K_b]_1(\mu)\bE\Bigl[\int^\infty_0e^{-2\mu t}|x(t)-x'(t)|^2\rd t\Bigr]^{1/2}
\end{equation}
and
\begin{equation}\label{lemm: coefficients SVIE: 2}
	\bE\Bigl[\int^\infty_0e^{-2\mu t}\int^t_0\bigl|\sigma(t,s,x(s))-\sigma(t,s,x'(s))\bigr|^2\rd s\rd t\Bigr]^{1/2}\leq [K_\sigma]_2(\mu)\bE\Bigl[\int^\infty_0e^{-2\mu t}|x(t)-x'(t)|^2\rd t\Bigr]^{1/2}.
\end{equation}
\end{lemm}


\begin{proof}
By using the assumptions, together with Young's convolution inequality, we have
\begin{align*}
	&\bE\Bigl[\int^\infty_0e^{-2\mu t}\Bigl(\int^t_0\bigl|b(t,s,x(s))-b(t,s,x'(s))\bigr|\rd s\Bigr)^2\rd t\Bigr]^{1/2}\\
	&\leq\bE\Bigl[\int^\infty_0e^{-2\mu t}\Bigl(\int^t_0K_b(t-s)|x(s)-x'(s)|\rd s\Bigr)^2\rd t\Bigr]^{1/2}\\
	&=\bE\Bigl[\int^\infty_0\Bigl(\int^t_0e^{-\mu(t-s)}K_b(t-s)e^{-\mu s}|x(s)-x'(s)|\rd s\Bigr)^2\rd t\Bigr]^{1/2}\\
	&\leq\bE\Bigl[\Bigl(\int^\infty_0e^{-\mu t}K_b(t)\rd t\Bigr)^2\int^\infty_0e^{-2\mu t}|x(t)-x'(t)|^2\rd t\Bigr]^{1/2}\\
	&=[K_b]_1(\mu)\bE\Bigl[\int^\infty_0e^{-2\mu t}|x(t)-x'(t)|^2\rd t\Bigr]^{1/2}.
\end{align*}
Thus the estimate~\eqref{lemm: coefficients SVIE: 1} holds. Similarly, we have
\begin{align*}
	&\bE\Bigl[\int^\infty_0e^{-2\mu t}\int^t_0|\sigma(t,s,x(s))-\sigma(t,s,x'(s))|^2\rd s\rd t\Bigr]^{1/2}\\
	&\leq \bE\Bigl[\int^\infty_0e^{-2\mu t}\int^t_0K_\sigma(t-s)^2|x(s)-x'(s)|^2\rd s\rd t\Bigr]^{1/2}\\
	&=\bE\Bigl[\int^\infty_0\int^t_0e^{-2\mu(t-s)}K_\sigma(t-s)^2e^{-2\mu s}|x(s)-x'(s)|^2\rd s\rd t\Bigr]^{1/2}\\
	&\leq\bE\Bigl[\int^\infty_0e^{-2\mu t}K_\sigma(t)^2\rd t\,\int^\infty_0e^{-2\mu t}|x(t)-x'(t)|^2\rd t\Bigr]^{1/2}\\
	&=[K_\sigma]_2(\mu)\bE\Bigl[\int^\infty_0e^{-2\mu t}|x(t)-x'(t)|^2\rd t\Bigr]^{1/2}.
\end{align*}
Thus the estimate~\eqref{lemm: coefficients SVIE: 2} holds.
\end{proof}

Observe that, if $K_b\in L^{1,*}(0,\infty;\bR_+)\setminus L^1(0,\infty;\bR_+)$ or $K_\sigma\in L^{2,*}(0,\infty;\bR_+)\setminus L^2(0,\infty;\bR_+)$, then the number $\mu$ in the above lemma has to be strictly positive, and this is indeed a typical case. Motivated by the above lemma, we define
\begin{equation}\label{domain: b sigma}
	\rho_{b,\sigma}:=\inf\{\rho\in\bR\,|\,[K_b]_1(\rho)+[K_\sigma]_2(\rho)\leq1\}\in\bR.
\end{equation}
Note that $[K_b]_1(\rho)+[K_\sigma]_2(\rho)<1$ for any $\rho\in(\rho_{b,\sigma},\infty)$. The number $\rho_{b,\sigma}$ gives a criterion for the weight of the space in which the solution to SVIE~\eqref{SVIE} uniquely exists for any $\varphi(\cdot)$. Now we prove the well-posedness of SVIE~\eqref{SVIE}.


\begin{prop}\label{prop: SVIE}
Suppose that Assumption~\ref{assum: SVIE} holds, and fix $\mu\in(\rho_{b,\sigma},\infty)$. Then for any $\varphi(\cdot)\in L^{2,-\mu}_\bF(0,\infty;\bR^n)$, SVIE~\eqref{SVIE} admits a unique solution $X(\cdot)\in L^{2,-\mu}_\bF(0,\infty;\bR^n)$. Furthermore, the following estimate holds:
\begin{equation}\label{prop: SVIE: 1}
	\bE\Bigl[\int^\infty_0e^{-2\mu t}|X(t)|^2\rd t\Bigr]^{1/2}\leq C_\mu\bE\Bigl[\int^\infty_0e^{-2\mu t}\bigl|\varphi(t)\bigr|^2\rd t\Bigr]^{1/2},
\end{equation}
where
\begin{equation*}
	C_\mu:=\frac{1}{1-[K_b]_1(\mu)-[K_\sigma]_2(\mu)}\in(0,\infty).
\end{equation*}
Let $b'$ and $\sigma'$ satisfy Assumption~\ref{assum: SVIE}, let $\varphi'(\cdot)\in L^{2,-\mu}_\bF(0,\infty;\bR^n)$ be given, and denote by $X'(\cdot)\in L^{2,-\mu}_\bF(0,\infty;\bR^n)$ the solution of SVIE~\eqref{SVIE} corresponding to $(\varphi',b',\sigma')$. Then it holds that
\begin{equation}\label{prop: SVIE: 2}
\begin{split}
	&\bE\Bigl[\int^\infty_0e^{-2\mu t}|X(t)-X'(t)|^2\rd t\Bigr]^{1/2}\\
	&\leq C_\mu\bE\Bigl[\int^\infty_0e^{-2\mu t}\Bigl|\varphi(t)-\varphi'(t)+\int^t_0\bigl\{b(t,s,X'(s))-b'(t,s,X'(s))\bigr\}\rd s\\
	&\hspace{4cm}+\int^t_0\bigl\{\sigma(t,s,X'(s))-\sigma'(t,s,X'(s))\bigr\}\rd W(s)\Bigr|^2\rd t\Bigr]^{1/2}.
\end{split}
\end{equation}
\end{prop}


\begin{proof}
Define a map $\Theta:L^{2,-\mu}_\bF(0,\infty;\bR^n)\to L^{2,-\mu}_\bF(0,\infty;\bR^n)$ by
\begin{equation*}
	\Theta[x(\cdot)](t):=\varphi(t)+\int^t_0b(t,s,x(s))\rd s+\int^t_0\sigma(t,s,x(s))\rd W(s),\ t\geq0,
\end{equation*}
for each $x(\cdot)\in L^{2,-\mu}_\bF(0,\infty;\bR^n)$. By Lemma~\ref{lemm: coefficients SVIE}, the map $\Theta$ is well-defined. Let $x(\cdot),x'(\cdot)\in L^{2,-\mu}_\bF(0,\infty;\bR^n)$ be fixed. By Lemma~\ref{lemm: coefficients SVIE},
\begin{align*}
	&\bE\Bigl[\int^\infty_0e^{-2\mu t}\bigl|\Theta[x(\cdot)](t)-\Theta[x'(\cdot)](t)\bigr|^2\rd t\Bigr]^{1/2}\\
	&\leq\bE\Bigl[\int^\infty_0e^{-2\mu t}\Bigl(\int^t_0\bigl|b(t,s,x(s))-b(t,s,x'(s))\bigr|\rd s\Bigr)^2\rd t\Bigr]^{1/2}\\
	&\hspace{1cm}+\bE\Bigl[\int^\infty_0e^{-2\mu t}\int^t_0\bigl|\sigma(t,s,x(s))-\sigma(t,s,x'(s))\bigr|^2\rd s\rd t\Bigr]^{1/2}\\
	&\leq\bigl\{[K_b]_1(\mu)+[K_\sigma]_2(\mu)\bigr\}\bE\Bigl[\int^\infty_0e^{-2\mu t}|x(t)-x'(t)|^2\rd t\Bigr]^{1/2}.
\end{align*}
Since $\mu\in(\rho_{b,\sigma},\infty)$ with $\rho_{b,\sigma}\in\bR$ defined by \eqref{domain: b sigma}, it holds that $[K_b]_1(\mu)+[K_\sigma]_2(\mu)<1$. This implies that $\Theta$ is a contraction map on the Banach space $L^{2,-\mu}_\bF(0,\infty;\bR^n)$. Thus, there exists a unique fixed point $X(\cdot)\in L^{2,-\mu}_\bF(0,\infty;\bR^n)$ of $\Theta$, which is the solution of SVIE~\eqref{SVIE}. Furthermore, by the same calculation yields that
\begin{equation*}
	\bE\Bigl[\int^\infty_0e^{-2\mu t}|X(t)|^2\rd t\Bigr]^{1/2}\leq\bE\Bigl[\int^\infty_0e^{-2\mu t}\bigl|\varphi(t)\bigr|^2\rd t\Bigr]^{1/2}+\bigl\{[K_b]_1(\mu)+[K_\sigma]_2(\mu)\bigr\}\bE\Bigl[\int^\infty_0e^{-2\mu t}|X(t)|^2\rd t\Bigr]^{1/2},
\end{equation*}
which implies the estimate~\eqref{prop: SVIE: 1}.

We prove the estimate \eqref{prop: SVIE: 2}. Observe that $\overline{X}(\cdot):=X(\cdot)-X'(\cdot)\in L^{2,-\mu}_\bF(0,\infty;\bR^n)$ is the solution of the SVIE
\begin{equation*}
	\overline{X}(t)=\overline{\varphi}(t)+\int^t_0\overline{b}(t,s,\overline{X}(s))\rd s+\int^t_0\overline{\sigma}(t,s,\overline{X}(s))\rd W(s),\ t\geq0,
\end{equation*}
where
\begin{align*}
	&\overline{\varphi}(t):=\varphi(t)-\varphi'(t)+\int^t_0\bigl\{b(t,s,X'(s))-b'(t,s,X'(s))\bigr\}\rd s+\int^t_0\bigl\{\sigma(t,s,X'(s))-\sigma'(t,s,X'(s))\bigr\}\rd W(s),\\
	&\overline{b}(t,s,x):=b(t,s,x+X'(s))-b(t,s,X'(s)),\ \overline{\sigma}(t,s,x):=\sigma(t,s,x+X'(s))-\sigma(t,s,X'(s)).
\end{align*}
Note that $\overline{\varphi}(\cdot)\in L^{2,-\mu}_\bF(0,\infty;\bR^n)$, and $\overline{b}$ and $\overline{\sigma}$ satisfy Assumption~\ref{assum: SVIE} with the functions $K_b$ and $K_\sigma$. Therefore, by the estimate~\eqref{prop: SVIE: 1}, we have
\begin{equation*}
	\bE\Bigl[\int^\infty_0e^{-2\mu t}|\overline{X}(t)|^2\rd t\Bigr]^{1/2}\leq C_\mu\bE\Bigl[\int^\infty_0e^{-2\mu t}|\overline{\varphi}(t)|^2\rd t\Bigr]^{1/2}.
\end{equation*}
Thus the estimate~\eqref{prop: SVIE: 2} holds, and we finish the proof.
\end{proof}


\begin{rem}\label{rem: SVIE nonzero}
The above results can be easily generalized to the case where $b(t,s,0)$ and $\sigma(t,s,0)$ are nonzero. Indeed, suppose that $b$ and $\sigma$ satisfy Assumption~\ref{assum: SVIE} (i), (ii), (iii), and
\begin{equation*}
	\bE\Bigl[\int^\infty_0e^{-2\mu t}\Bigl(\int^t_0|b(t,s,0)|\rd s\Bigr)^2\rd t+\int^\infty_0e^{-2\mu t}\int^t_0|\sigma(t,s,0)|^2\rd s\rd t\Bigr]<\infty
\end{equation*}
with $\mu\in(\rho_{b,\sigma},\infty)$. For any given $\varphi(\cdot)\in L^{2,-\mu}_\bF(0,\infty;\bR^n)$, define
\begin{align*}
	&\tilde{\varphi}(t):=\varphi(t)+\int^t_0b(t,s,0)\rd s+\int^t_0\sigma(t,s,0)\rd W(s),\\
	&\tilde{b}(t,s,x):=b(t,s,x)-b(t,s,0),\ \tilde{\sigma}(t,s,x):=\sigma(t,s,x)-\sigma(t,s,0),
\end{align*}
for $(t,s,x)\in\Delta^\comp[0,\infty)\times\bR^n$. Then $\tilde{\varphi}(\cdot)\in L^{2,-\mu}_\bF(0,\infty;\bR^n)$, and $\tilde{b}$ and $\tilde{\sigma}$ satisfy all the conditions in Assumption~\ref{assum: SVIE} with the functions $K_b$ and $K_\sigma$. Thus, there exists a unique solution $X(\cdot)\in L^{2,-\mu}_\bF(0,\infty;\bR^n)$ to the SVIE
\begin{equation*}
	X(t)=\tilde{\varphi}(t)+\int^t_0\tilde{b}(t,s,X(s))\rd s+\int^t_0\tilde{\sigma}(t,s,X(s))\rd W(s),\ t\geq0.
\end{equation*}
Clearly, $X(\cdot)$ is the unique solution of the original SVIE~\eqref{SVIE}.
\end{rem}


\section{Infinite horizon BSVIEs}\label{section: BSVIE}


In this section, we investigate infinite horizon BSVIE~\eqref{Type-II}, which we rewrite for readers' convenience:
\begin{equation}\label{BSVIE}
	Y(t)=\psi(t)+\int^\infty_te^{-\lambda(s-t)}g(t,s,Y(s),Z(t,s),Z(s,t))\rd s-\int^\infty_tZ(t,s)\rd W(s),\ t\geq0,
\end{equation}
where the free term $\psi(\cdot)\in L^{2,*}_{\cF_\infty}(0,\infty;\bR^m)$ with $m\in\bN$, the driver $g$, and the discount rate $\lambda\in\bR$ are given. We impose the following assumptions on the driver $g$:


\begin{assum}\label{assum: BSVIE}
\begin{itemize}
\item[(i)]
$g:\Omega\times\Delta[0,\infty)\times\bR^m\times\bR^{m\times d}\times\bR^{m\times d}\to\bR^m$ is measurable.
\item[(ii)]
$(g(t,s,y,z_1,z_2))_{s\in[t,\infty)}$ is adapted for each $t\in[0,\infty)$ and $(y,z_1,z_2)\in\bR^m\times\bR^{m\times d}\times\bR^{m\times d}$.
\item[(iii)]
There exist $K_{g,y}\in L^{1,*}(0,\infty;\bR_+)$ and $K_{g,z_1},K_{g,z_2}\in L^{2,*}(0,\infty;\bR_+)$ such that
\begin{equation*}
	|g(t,s,y,z_1,z_2)-g(t,s,y',z'_1,z'_2)|\leq K_{g,y}(s-t)|y-y'|+K_{g,z_1}(s-t)|z_1-z'_1|+K_{g,z_2}(s-t)|z_2-z'_2|
\end{equation*}
for any $(y,z_1,z_2),(y',z'_1,z'_2)\in\bR^m\times\bR^{m\times d}\times\bR^{m\times d}$, for a.e.\ $(t,s)\in\Delta[0,\infty)$, a.s.
\item[(iv)]
$g(t,s,0,0,0)=0$ for a.e.\ $(t,s)\in\Delta[0,\infty)$, a.s.
\end{itemize}
\end{assum}

Note that, as in the SVIEs case, the Lipschitz coefficients are allowed to be singular at the diagonal $t=s$ and the infinity. We give a definition of the solution.


\begin{defi}
We call a pair $(Y(\cdot),Z(\cdot,\cdot))$ the \emph{adapted M-solution} of infinite horizon BSVIE~\eqref{BSVIE} if $(Y(\cdot),Z(\cdot,\cdot))$ is in $\cM^{2,*}_\bF(0,\infty;\bR^m\times\bR^{m\times d})$ and satisfies the equality for a.e.\ $t\geq0$, a.s.
\end{defi}

The above definition is a natural extension of that of \cite{Yo08} to the infinite horizon setting. The ``M''-solution is named after the martingale representation theorem $Y(t)=\bE[Y(t)]+\int^t_0Z(t,s)\rd W(s)$ which determines the values of $Z(t,s)$ for $(t,s)\in\Delta^\comp[0,\infty)$. Recall that this relation is a part of the definition of the space $\cM^{2,*}_\bF(0,\infty;\bR^m\times\bR^{m\times d})$.


\subsection{Well-posedness of infinite horizon BSVIEs}


In this subsection, we prove the well-posedness of infinite horizon BSVIE~\eqref{BSVIE} under Assumption~\ref{assum: BSVIE}. A delicate problem is the condition which ensures the existence and uniqueness of the adapted M-solution for \emph{any} free term $\psi(\cdot)$ in a \emph{given} weighted $L^2$-space. As we mentioned in the introductory section, in order to ensure the well-posedness in a large space, the driver $g$ has to be discounted sufficiently at infinity. The following lemma provides a useful criterion for this trade-off relationship.


\begin{lemm}\label{lemm: coefficients BSVIE}
Suppose that Assumption~\ref{assum: BSVIE} holds. Fix $\eta,\lambda\in\bR$ satisfying $[K_{g,y}]_1(\eta+\lambda)+[K_{g,z_1}]_2(\lambda)+[K_{g,z_2}]_2(\eta+\lambda)<\infty$. Then for any $(y(\cdot),z(\cdot,\cdot)),(y'(\cdot),z'(\cdot,\cdot))\in \cM^{2,\eta}_\bF(0,\infty;\bR^m\times\bR^{m\times d})$, it holds that
\begin{align*}
	&\bE\Bigl[\int^\infty_0e^{2\eta t}\Bigl(\int^\infty_te^{-\lambda(s-t)}|g(t,s,y(s),z(t,s),z(s,t))-g(t,s,y'(s),z'(t,s),z'(s,t))|\rd s\Bigr)^2\rd t\Bigr]^{1/2}\\
	&\leq \bigl\{[K_{g,y}]_1(\eta+\lambda)+[K_{g,z_2}]_2(\eta+\lambda)\bigr\}\bE\Bigl[\int^\infty_0e^{2\eta t}|y(t)-y'(t)|^2\rd t\Bigr]^{1/2}\\
	&\hspace{1cm}+[K_{g,z_1}]_2(\lambda)\bE\Bigl[\int^\infty_0e^{2\eta t}\int^\infty_t|z(t,s)-z'(t,s)|^2\rd s\rd t\Bigr]^{1/2}.
\end{align*}
\end{lemm}


\begin{proof}
We observe that
\begin{align*}
	&\bE\Bigl[\int^\infty_0e^{2\eta t}\Bigl(\int^\infty_te^{-\lambda(s-t)}|g(t,s,y(s),z(t,s),z(s,t))-g(t,s,y'(s),z'(t,s),z'(s,t))|\rd s\Bigr)^2\rd t\Bigr]^{1/2}\\
	&\leq I_y+I_{z_1}+I_{z_2},
\end{align*}
where
\begin{align*}
	&I_y:=\bE\Bigl[\int^\infty_0e^{2\eta t}\Bigl(\int^\infty_te^{-\lambda(s-t)}K_{g,y}(s-t)|y(s)-y'(s)|\rd s\Bigr)^2\rd t\Bigr]^{1/2},\\
	&I_{z_1}:=\bE\Bigl[\int^\infty_0e^{2\eta t}\Bigl(\int^\infty_te^{-\lambda(s-t)}K_{g,z_1}(s-t)|z(t,s)-z'(t,s)|\rd s\Bigr)^2\rd t\Bigr]^{1/2},\\
	&I_{z_2}:=\bE\Bigl[\int^\infty_0e^{2\eta t}\Bigl(\int^\infty_te^{-\lambda(s-t)}K_{g,z_2}(s-t)|z(s,t)-z'(s,t)|\rd s\Bigr)^2\rd t\Bigr]^{1/2}.
\end{align*}
By using Young's convolution inequality, we get
\begin{align*}
	I_y&=\bE\Bigl[\int^\infty_0\Bigl(\int^\infty_te^{-(\eta+\lambda)(s-t)}K_{g,y}(s-t)e^{\eta s}|y(s)-y'(s)|\rd s\Bigr)^2\rd t\Bigr]^{1/2}\\
	&\leq\bE\Bigl[\Bigl(\int^\infty_0e^{-(\eta+\lambda)t}K_{g,y}(t)\rd t\Bigr)^2\int^\infty_0e^{2\eta t}|y(t)-y'(t)|^2\rd t\Bigr]^{1/2}\\
	&=[K_{g,y}]_1(\eta+\lambda)\bE\Bigl[\int^\infty_0e^{2\eta t}|y(t)-y'(t)|^2\rd t\Bigr]^{1/2}.
\end{align*}
Next, by H\"{o}lder's inequality,
\begin{align*}
	I_{z_1}&\leq\bE\Bigl[\int^\infty_0e^{2\eta t}\int^\infty_te^{-2\lambda(s-t)}K_{g,z_1}(s-t)^2\rd s\int^\infty_t|z(t,s)-z'(t,s)|^2\rd s\rd t\Bigr]^{1/2}\\
	&=\Bigl(\int^\infty_0e^{-2\lambda t}K_{g,z_1}(t)^2\rd t\Bigr)^{1/2}\bE\Bigl[\int^\infty_0e^{2\eta t}\int^\infty_t|z(t,s)-z'(t,s)|^2\rd s\rd t\Bigr]^{1/2}\\
	&=[K_{g,z_1}]_2(\lambda)\bE\Bigl[\int^\infty_0e^{2\eta t}\int^\infty_t|z(t,s)-z'(t,s)|^2\rd s\rd t\Bigr]^{1/2}.
\end{align*}
Lastly, by using H\"{o}lder's inequality and Fubini's theorem, we have
\begin{align*}
	I_{z_2}&=\bE\Bigl[\int^\infty_0\Bigl(\int^\infty_te^{-(\eta+\lambda)(s-t)}K_{g,z_2}(s-t)e^{\eta s}|z(s,t)-z'(s,t)|\rd s\Bigr)^2\rd t\Bigr]^{1/2}\\
	&\leq\bE\Bigl[\int^\infty_0\int^\infty_te^{-2(\eta+\lambda)(s-t)}K_{g,z_2}(s-t)^2\rd s\int^\infty_te^{2\eta s}|z(s,t)-z'(s,t)|^2\rd s\rd t\Bigr]^{1/2}\\
	&=\Bigl(\int^\infty_0e^{-2(\eta+\lambda)t}K_{g,z_2}(t)^2\rd t\Bigr)^{1/2}\bE\Bigl[\int^\infty_0\int^\infty_te^{2\eta s}|z(s,t)-z'(s,t)|^2\rd s\rd t\Bigr]^{1/2}\\
	&=[K_{g,z_2}]_2(\eta+\lambda)\Bigl(\int^\infty_0e^{2\eta t}\bE\Bigl[\int^t_0|z(t,s)-z'(t,s)|^2\rd s\Bigr]\rd t\Bigr)^{1/2}.
\end{align*}
We note that, since $(y(\cdot),z(\cdot,\cdot)),(y'(\cdot),z'(\cdot,\cdot))\in \cM^{2,\eta}_\bF(0,\infty;\bR^m\times\bR^{m\times d})$, it holds that
\begin{align*}
	\bE\Bigl[\int^t_0|z(t,s)-z'(t,s)|^2\rd s\Bigr]&=\bE\Bigl[\Bigl|\int^t_0\bigl\{z(t,s)-z'(t,s)\bigr\}\rd W(s)\Bigr|^2\Bigr]\\
	&=\bE\bigl[\bigl|y(t)-y'(t)-\bE[y(t)-y'(t)]\bigr|^2\bigr]\\
	&\leq\bE\bigl[|y(t)-y'(t)|^2\bigr]
\end{align*}
for a.e.\ $t\geq0$. Thus, we get
\begin{equation*}
	I_{z_2}\leq[K_{g,z_2}]_2(\eta+\lambda)\bE\Bigl[\int^\infty_0e^{2\eta t}|y(t)-y'(t)|^2\rd t\Bigr]^{1/2}.
\end{equation*}
From the above observations, we get the assertion of the lemma.
\end{proof}

With the above lemma in mind, we introduce the following domain:
\begin{equation}\label{domain: g}
	\cR_g:=\{(\eta,\lambda)\in\bR^2\,|\,[K_{g,y}]_1(\eta+\lambda)+[K_{g,z_1}]_2(\lambda)+[K_{g,z_2}]_2(\eta+\lambda)<1\},
\end{equation}
which gives a criterion for the trade-off relationship between the weight of the space of the solution and the discount rate. Note that, since $K_{g,y}\in L^{1,*}(0,\infty;\bR_+)$ and $K_{g,z_1},K_{g,z_2}\in L^{2,*}(0,\infty;\bR_+)$, the set $\cR_g$ is nonempty.


\begin{rem}
\begin{itemize}
\item[(i)]
If $(\eta,\lambda)\in\cR_g$, then for any $(\eta',\lambda')\in\bR^2$ such that $\eta\leq \eta'$ and $\lambda\leq \lambda'$, we have $(\eta',\lambda')\in\cR_g$.
\item[(ii)]
If the driver $g(t,s,y,z_1,z_2)$ does not depend on $z_1$, then we can take $K_{g,z_1}=0$. In this case, we have $\cR_g=\{(\eta,\lambda)\in\bR^2\,|\,\eta+\lambda>\rho_{g;y,z_2}\}$ with $\rho_{g;y,z_2}:=\inf\{\rho\in\bR\,|\,[K_{g,y}]_1(\rho)+[K_{g,z_2}]_2(\rho)\leq1\}$.
\item[(iii)]
If the driver $g(t,s,y,z_1,z_2)$ does not depend on $(y,z_2)$, then we can take $K_{g,y}=K_{g,z_2}=0$. In this case, we have $\cR_g=\{(\eta,\lambda)\in\bR^2\,|\,\lambda>\rho_{g;z_1}\}$ with $\rho_{g;z_1}:=\inf\{\rho\in\bR\,|\,[K_{g,z_1}]_2(\rho)\leq1\}$.
\end{itemize}
\end{rem}

We prove the well-posedness of infinite horizon BSVIE~\eqref{BSVIE} by using the so-called \emph{method of continuation}. For this purpose, we first provide a priori estimates for infinite horizon BSVIEs, which is important on its own right.


\begin{prop}\label{prop: a priori estimate BSVIE}
Suppose that Assumption~\ref{assum: BSVIE} holds, and fix $(\eta,\lambda)\in\cR_g$. Let $\psi(\cdot)\in L^{2,\eta}_{\cF_\infty}(0,\infty;\bR^m)$ be given, and let $(Y(\cdot),Z(\cdot,\cdot))\in\cM^{2,\eta}_\bF(0,\infty;\bR^m\times\bR^{m\times d})$ be an adapted M-solution of infinite horizon BSVIE~\eqref{BSVIE}. Then it holds that
\begin{equation}\label{prop: a priori estimate BSVIE: 1}
	\bE\Bigl[\int^\infty_0e^{2\eta t}|Y(t)|^2\rd t+\int^\infty_0e^{2\eta t}\int^\infty_0|Z(t,s)|^2\rd s\rd t\Bigr]^{1/2}\leq C_{\eta,\lambda}\bE\Bigl[\int^\infty_0e^{2\eta t}\bigl|\psi(t)\bigr|^2\rd t\Bigr]^{1/2},
\end{equation}
where
\begin{equation*}
	C_{\eta,\lambda}:=\frac{\sqrt{2}}{1-[K_{g,y}]_1(\eta+\lambda)-[K_{g,z_1}]_2(\lambda)-[K_{g,z_2}]_2(\eta+\lambda)}\in(0,\infty).
\end{equation*}
Suppose that $g'$ satisfies Assumption~\ref{assum: BSVIE}, let $\psi'(\cdot)\in L^{2,\eta}_{\cF_\infty}(0,\infty;\bR^m)$, and let $(Y'(\cdot),Z'(\cdot,\cdot))\in\cM^{2,\eta}_\bF(0,\infty;\bR^m\times\bR^{m\times d})$ be an adapted M-solution of \eqref{BSVIE} corresponding to $(\psi',g',\lambda)$. Then it holds that
\begin{equation}\label{prop: a priori estimate BSVIE: 2}
\begin{split}
	&\bE\Bigl[\int^\infty_0e^{2\eta t}|Y(t)-Y'(t)|^2\rd t+\int^\infty_0e^{2\eta t}\int^\infty_0|Z(t,s)-Z'(t,s)|^2\rd s\rd t\Bigr]^{1/2}\\
	&\leq C_{\eta,\lambda}\bE\Bigl[\int^\infty_0e^{2\eta t}\Bigl|\psi(t)-\psi'(t)+\int^\infty_te^{-\lambda(s-t)}\bigl\{g(t,s,Y'(s),Z'(t,s),Z'(s,t))\\
	&\hspace{7cm}-g'(t,s,Y'(s),Z'(t,s),Z'(s,t))\bigr\}\rd s\Bigr|^2\rd t\Bigr]^{1/2}.
\end{split}
\end{equation}
\end{prop}


\begin{proof}
By equation~\eqref{BSVIE}, it holds that
\begin{align*}
	\bE\Bigl[|Y(t)|^2+\int^\infty_t|Z(t,s)|^2\rd s\Bigr]&=\bE\Bigl[\Bigl|Y(t)+\int^\infty_tZ(t,s)\rd W(s)\Bigr|^2\Bigr]\\
	&=\bE\Bigl[\Bigl|\psi(t)+\int^\infty_te^{-\lambda(s-t)}g(t,s,Y(s),Z(t,s),Z(s,t))\rd s\Bigr|^2\Bigr]
\end{align*}
for a.e.\ $t\geq0$. Thus, by using Minkowski's inequality, we get
\begin{equation}\label{proof: a priori estimate BSVIE: 1}
\begin{split}
	&\bE\Bigl[\int^\infty_0e^{2\eta t}|Y(t)|^2\rd t+\int^\infty_0e^{2\eta t}\int^\infty_t|Z(t,s)|^2\rd s\rd t\Bigr]^{1/2}\\
	&\leq\bE\Bigl[\int^\infty_0e^{2\eta t}\bigl|\psi(t)\bigr|^2\rd t\Bigr]^{1/2}+\bE\Bigl[\int^\infty_0e^{2\eta t}\Bigl(\int^\infty_te^{-\lambda(s-t)}|g(t,s,Y(s),Z(t,s),Z(s,t))|\rd s\Bigr)^2\rd t\Bigr]^{1/2}.
\end{split}
\end{equation}
By Lemma~\ref{lemm: coefficients BSVIE}, the second term in the right-hand side is estimated as
\begin{equation}\label{proof: a priori estimate BSVIE: 2}
\begin{split}
	&\bE\Bigl[\int^\infty_0e^{2\eta t}\Bigl(\int^\infty_te^{-\lambda(s-t)}|g(t,s,Y(s),Z(t,s),Z(s,t))|\rd s\Bigr)^2\rd t\Bigr]^{1/2}\\
	&\leq\bigl\{[K_{g,y}]_1(\eta+\lambda)+[K_{g,z_2}]_2(\eta+\lambda)\bigr\}\bE\Bigl[\int^\infty_0e^{2\eta t}|Y(t)|^2\rd t\Bigr]^{1/2}\\
	&\hspace{1cm}+[K_{g,z_1}]_2(\lambda)\bE\Bigl[\int^\infty_0e^{2\eta t}\int^\infty_t|Z(t,s)|^2\rd s\rd t\Bigr]^{1/2}\\
	&\leq\bigl\{[K_{g,y}]_1(\eta+\lambda)+[K_{g,z_1}]_2(\lambda)+[K_{g,z_2}]_2(\eta+\lambda)\bigr\}\bE\Bigl[\int^\infty_0e^{2\eta t}|Y(t)|^2\rd t+\int^\infty_0e^{2\eta t}\int^\infty_t|Z(t,s)|^2\rd s\rd t\Bigr]^{1/2}.
\end{split}
\end{equation}
Noting that $(\eta,\lambda)\in\cR_g$ with $\cR_g$ defined by \eqref{domain: g}, by the estimates~\eqref{proof: a priori estimate BSVIE: 1} and \eqref{proof: a priori estimate BSVIE: 2}, we get
\begin{align*}
	&\bE\Bigl[\int^\infty_0e^{2\eta t}|Y(t)|^2\rd t+\int^\infty_0e^{2\eta t}\int^\infty_t|Z(t,s)|^2\rd s\rd t\Bigr]^{1/2}\\
	&\leq\frac{1}{1-[K_{g,y}]_1(\eta+\lambda)-[K_{g,z_1}]_2(\lambda)-[K_{g,z_2}]_2(\eta+\lambda)}\bE\Bigl[\int^\infty_0e^{2\eta t}\bigl|\psi(t)\bigr|^2\rd t\Bigr]^{1/2}.
\end{align*}
On the other hand, since $(Y(\cdot),Z(\cdot,\cdot))\in\cM^{2,\eta}_\bF(0,\infty;\bR^m\times\bR^{m\times d})$, it holds that
\begin{equation*}
	\bE\Bigl[\int^t_0|Z(t,s)|^2\rd s\Bigr]=\bE\Bigl[\Bigl|\int^t_0Z(t,s)\rd W(s)\Bigr|^2\Bigr]=\bE\bigl[\bigl|Y(t)-\bE[Y(t)]\bigr|^2\bigr]\leq\bE\bigl[|Y(t)|^2\bigr]
\end{equation*}
for a.e.\ $t\geq0$. Therefore, we obtain
\begin{align*}
	&\bE\Bigl[\int^\infty_0e^{2\eta t}|Y(t)|^2\rd t+\int^\infty_0e^{2\eta t}\int^\infty_0|Z(t,s)|^2\rd s\rd t\Bigr]^{1/2}\\
	&\leq\sqrt{2}\bE\Bigl[\int^\infty_0e^{2\eta t}|Y(t)|^2\rd t+\int^\infty_0e^{2\eta t}\int^\infty_t|Z(t,s)|^2\rd s\rd t\Bigr]^{1/2}\\
	&\leq\frac{\sqrt{2}}{1-[K_{g,y}]_1(\eta+\lambda)\mathalpha{-}[K_{g,z_1}]_2(\lambda)-[K_{g,z_2}]_2(\eta+\lambda)}\bE\Bigl[\int^\infty_0e^{2\eta t}\bigl|\psi(t)\bigr|^2\rd t\Bigr]^{1/2},
\end{align*}
and thus the estimate~\eqref{prop: a priori estimate BSVIE: 1} holds.

Next, we prove the estimate~\eqref{prop: a priori estimate BSVIE: 2}. Define $\overline{Y}(\cdot):=Y(\cdot)-Y'(\cdot)$ and $\overline{Z}(\cdot,\cdot):=Z(\cdot,\cdot)-Z'(\cdot,\cdot)$. Then $(\overline{Y}(\cdot),\overline{Z}(\cdot,\cdot))\in\cM^{2,\eta}_\bF(0,\infty;\bR^m\times\bR^{m\times d})$ is an adapted M-solution of the infinite horizon BSVIE
\begin{equation*}
	\overline{Y}(t)=\overline{\psi}(t)+\int^\infty_te^{-\lambda(s-t)}\overline{g}(t,s,\overline{Y}(s),\overline{Z}(t,s),\overline{Z}(s,t))\rd s-\int^\infty_t\overline{Z}(t,s)\rd W(s),\ t\geq0,
\end{equation*}
where
\begin{align*}
	&\overline{\psi}(t):=\psi(t)-\psi'(t)+\int^\infty_te^{-\lambda(s-t)}\bigl\{g(t,s,Y'(s),Z'(t,s),Z'(s,t))-g'(t,s,Y'(s),Z'(t,s),Z'(s,t))\bigr\}\rd s,\\
	&\overline{g}(t,s,y,z_1,z_2):=g(t,s,y+Y'(s),z_1+Z'(t,s),z_2+Z'(s,t))-g(t,s,Y'(s),Z'(t,s),Z'(s,t)).
\end{align*}
Note that $\overline{\psi}(\cdot)\in L^{2,\eta}_{\cF_\infty}(0,\infty;\bR^m)$, and the driver $\overline{g}$ satisfies Assumption~\ref{assum: BSVIE} with the functions $K_{g,y}$, $K_{g,z_1}$, and $K_{g,z_2}$. By estimate~\eqref{prop: a priori estimate BSVIE: 1}, it holds that
\begin{equation*}
	\bE\Bigl[\int^\infty_0e^{2\eta t}|\overline{Y}(t)|^2\rd t+\int^\infty_0e^{2\eta t}\int^\infty_0|\overline{Z}(t,s)|^2\rd s\rd t\Bigr]^{1/2}\leq C_{\eta,\lambda}\bE\Bigl[\int^\infty_0e^{2\eta t}|\overline{\psi}(t)|^2\rd t\Bigr]^{1/2}.
\end{equation*}
Thus the estimate~\eqref{prop: a priori estimate BSVIE: 2} holds.
\end{proof}

Next, we consider a trivial infinite horizon BSVIE whose generator $g$ vanishes.


\begin{lemm}\label{lemm: trivial BSVIE}
For any $\psi(\cdot)\in L^{2,\eta}_{\cF_\infty}(0,\infty;\bR^m)$ with $\eta\in\bR$, the infinite horizon BSVIE
\begin{equation}\label{lemm: trivial BSVIE: 1}
	Y(t)=\psi(t)-\int^\infty_tZ(t,s)\rd W(s),\ t\geq0,
\end{equation}
has a unique adapted M-solution $(Y(\cdot),Z(\cdot,\cdot))\in\cM^{2,\eta}_\bF(0,\infty;\bR^m\times\bR^{m\times d})$.
\end{lemm}


\begin{proof}
The uniqueness is trivial. We prove the existence. We should consider the issues of (joint) measurability and the integrability carefully. We split the proof into three steps.

\emph{Step 1.} Assume that $t\mapsto\psi(t)\in L^2_{\cF_\infty}(\Omega;\bR^m)$ is continuous. By the martingale representation theorem, for any $t\geq0$, there exists a unique adapted process $Z(t,\cdot)\in L^2_\bF(0,\infty;\bR^{m\times d})$ such that $\psi(t)=\bE[\psi(t)]+\int^\infty_0Z(t,s)\rd W(s)$ a.s. Define $Y(t):=\bE_t[\psi(t)]$. Then it holds that
\begin{equation*}
	Y(t)=\psi(t)-\int^\infty_tZ(t,s)\rd W(s)=\bE[Y(t)]+\int^t_0Z(t,s)\rd W(s)\ \text{a.s.}
\end{equation*}
for any $t\geq0$. By the continuity of the map $t\mapsto\psi(t)\in L^2_{\cF_\infty}(\Omega;\bR^m)$, it is easy to see that the maps $t\mapsto Y(t)\in L^2_{\cF_\infty}(\Omega;\bR^m)$ and $t\mapsto Z(t,\cdot)\in L^2_\bF(0,\infty;\bR^{m\times d})$ are continuous. Thus, there exist (jointly) measurable versions of $Y(\cdot)$ and $Z(\cdot,\cdot)$. Furthermore, we have
\begin{align*}
	&\bE\Bigl[\int^\infty_0e^{2\eta t}|Y(t)|^2\rd t]\leq\bE\Bigl[\int^\infty_0e^{2\eta t}|\psi(t)|^2\rd t\Bigr]<\infty,\\
	&\bE\Bigl[\int^\infty_0e^{2\eta t}\int^\infty_0|Z(t,s)|^2\rd s\rd t]\leq\bE\Bigl[\int^\infty_0e^{2\eta t}|\psi(t)|^2\rd t\Bigr]<\infty.
\end{align*}
Thus, the measurable version $(Y(\cdot),Z(\cdot,\cdot))$ is in $\cM^{2,\eta}_\bF(0,\infty;\bR^m\times\bR^{m\times d})$, and it is the adapted M-solution of \eqref{lemm: trivial BSVIE: 1}.

\emph{Step 2.} Assume that there exists a constant $M\in(0,\infty)$ such that $|\psi(t)|\leq Me^{-(|\eta|+1)t}$ for a.e.\ $t\geq0$, a.s. For each $N\in\bN$, define $\psi_N(t):=N\int^{t+1/N}_t\psi(s)\rd s$. Then the map $t\mapsto\psi_N(t)\in L^2_{\cF_\infty}(\Omega;\bR^m)$ is continuous. Furthermore, it holds that $|\psi_N(t)|\leq Me^{-(|\eta|+1)t}$, and hence $\psi_N(\cdot)\in L^{2,\eta}_{\cF_\infty}(0,\infty;\bR^m)$. By Step 1, there exists a unique adapted M-solution $(Y_N(\cdot),Z_N(\cdot,\cdot))\in\cM^{2,\eta}_\bF(0,\infty;\bR^m\times\bR^{m\times d})$ of the infinite horizon BSVIE
\begin{equation*}
	Y_N(t)=\psi_N(t)-\int^\infty_tZ_N(t,s)\rd W(s),\ t\geq0.
\end{equation*}
By Lebesgue's differentiation theorem, we see that $\lim_{N\to\infty}\psi_N(t)=\psi(t)$ for a.e.\ $t\geq0$, a.s. Thus, by the dominated convergence theorem, it holds that
\begin{equation*}
	\lim_{N\to\infty}\bE\Bigl[\int^\infty_0e^{2\eta t}|\psi_N(t)-\psi(t)|^2\rd t\Bigr]=0.
\end{equation*}
From this, it is easy to see that $\{(Y_N(\cdot),Z_N(\cdot,\cdot))\}_{N\in\bN}$ is a Cauchy sequence in $\cM^{2,\eta}_\bF(0,\infty;\bR^m\times\bR^{m\times d})$, and the limit $(Y(\cdot),Z(\cdot,\cdot))\in\cM^{2,\eta}_\bF(0,\infty;\bR^m\times\bR^{m\times d})$ is the adapted M-solution of \eqref{lemm: trivial BSVIE: 1}.

\emph{Step 3.} We consider the general $\psi(\cdot)\in L^{2,\eta}_{\cF_\infty}(0,\infty;\bR^m)$. For each $M\in\bN$, define
\begin{equation*}
	\psi_M(t):=\psi(t)\1_{\{|\psi(t)|\leq Me^{-(|\eta|+1)t}\}}.
\end{equation*}
By Step 2, there exists a unique adapted M-solution $(Y_M(\cdot),Z_M(\cdot,\cdot))\in\cM^{2,\eta}_\bF(0,\infty;\bR^m\times\bR^{m\times d})$ of the infinite horizon BSVIE
\begin{equation*}
	Y_M(t)=\psi_M(t)-\int^\infty_tZ_M(t,s)\rd W(s),\ t\geq0.
\end{equation*}
By the dominated convergence theorem, we see that
\begin{equation*}
	\lim_{M\to\infty}\bE\Bigl[\int^\infty_0e^{2\eta t}|\psi_M(t)-\psi(t)|^2\rd t\Bigr]=0.
\end{equation*}
Thus the sequence $\{(Y_M(\cdot),Z_M(\cdot,\cdot))\}_{M\in\bN}$ converges in $\cM^{2,\eta}_\bF(0,\infty;\bR^m\times\bR^{m\times d})$, and the limit is the adapted M-solution of \eqref{lemm: trivial BSVIE: 1}. This completes the proof.
\end{proof}

For each $\psi(\cdot)\in L^{2,\eta}_{\cF_\infty}(0,\infty;\bR^m)$ and $\gamma\in[0,1]$, consider the following infinite horizon BSVIE:
\begin{equation}\label{Type-II alpha}
	Y_\gamma(t)=\psi(t)+\int^\infty_te^{-\lambda(s-t)}\gamma g(t,s,Y_\gamma(s),Z_\gamma(t,s),Z_\gamma(s,t))\rd s-\int^\infty_tZ_\gamma(t,s)\rd W(s),\ t\geq0.
\end{equation}
We note that when $\gamma=0$, the equation \eqref{Type-II alpha} becomes the trivial infinite horizon BSVIE~\eqref{lemm: trivial BSVIE: 1}, while when $\gamma=1$, it coincides with the original infinite horizon BSVIE~\eqref{BSVIE}. For each $\gamma\in[0,1]$, we consider the following property:
\begin{itemize}
\item[(P$_\gamma$)]
For any $\psi(\cdot)\in L^{2,\eta}_{\cF_\infty}(0,\infty;\bR^m)$, the infinite horizon BSVIE~\eqref{Type-II alpha} admits a unique adapted M-solution $(Y_\gamma(\cdot),Z_\gamma(\cdot,\cdot))\in\cM^{2,\eta}_\bF(0,\infty;\bR^m\times\bR^{m\times d})$.
\end{itemize}
The following lemma is a key step of the method of continuation.


\begin{lemm}\label{lemm: continuation}
Suppose that Assumption~\ref{assum: BSVIE} holds, and fix $(\eta,\lambda)\in\cR_g$. Assume that the property (P$_{\gamma_0}$) holds for some $\gamma_0\in[0,1)$. Then (P$_\gamma$) holds for any $\gamma\in[0,1]$ with $\gamma_0\leq\gamma\leq\gamma_0+C^{-1}_{\eta,\lambda}$, where $C_{\eta,\lambda}\in(0,\infty)$ is the constant defined in Proposition~\ref{prop: a priori estimate BSVIE}.
\end{lemm}


\begin{proof}
Let $\gamma\in[0,1]$ with $\gamma_0\leq\gamma\leq\gamma_0+C^{-1}_{\eta,\lambda}$ be fixed. Take an arbitrary $\psi(\cdot)\in L^{2,\eta}_{\cF_\infty}(0,\infty;\bR^m)$. Let $(Y^{(0)}(\cdot),Z^{(0)}(\cdot,\cdot)):=(0,0)$, and define $(Y^{(i)}(\cdot),Z^{(i)}(\cdot,\cdot))\in\cM^{2,\eta}_\bF(0,\infty;\bR^m\times\bR^{m\times d})$ for $i\in\bN$ as the unique adapted M-solution of the following infinite horizon BSVIE:
\begin{equation*}
	Y^{(i)}(t)=\psi^{(i)}(t)+\int^\infty_te^{-\lambda(s-t)}\gamma_0g(t,s,Y^{(i)}(s),Z^{(i)}(t,s),Z^{(i)}(s,t))\rd s-\int^\infty_tZ^{(i)}(t,s)\rd W(s),\ t\geq0,
\end{equation*}
where
\begin{equation*}
	\psi^{(i)}(t):=\psi(t)+(\gamma-\gamma_0)\int^\infty_te^{-\lambda(s-t)}g(t,s,Y^{(i-1)}(s),Z^{(i-1)}(t,s),Z^{(i-1)}(s,t))\rd s,\ t\geq0.
\end{equation*}
By the assumption, the sequence $\{(Y^{(i)}(\cdot),Z^{(i)}(\cdot,\cdot))\}_{i\in\bN}$ can be defined inductively. Noting that $\gamma_0\in[0,1)$, by the stability estimate \eqref{prop: a priori estimate BSVIE: 2}, together with Lemma~\ref{lemm: coefficients BSVIE}, for each $i\in\bN$, we have
\begin{align*}
	&\bE\Bigl[\int^\infty_0e^{2\eta t}|Y^{(i+1)}(t)-Y^{(i)}(t)|^2\rd t+\int^\infty_0e^{2\eta t}\int^\infty_0|Z^{(i+1)}(t,s)-Z^{(i)}(t,s)|^2\rd s\rd t\Bigr]^{1/2}\\
	&\leq C_{\eta,\lambda}\bE\Bigl[\int^\infty_0e^{2\eta t}|\psi^{(i+1)}(t)-\psi^{(i)}(t)|^2\rd t\Bigr]^{1/2}\\
	&\leq C_{\eta,\lambda}(\gamma-\gamma_0)\bE\Bigl[\int^\infty_0e^{2\eta t}\Bigl(\int^\infty_te^{-\lambda(s-t)}|g(t,s,Y^{(i)}(s),Z^{(i)}(t,s),Z^{(i)}(s,t))\\
	&\hspace{5cm}-g(t,s,Y^{(i-1)}(s),Z^{(i-1)}(t,s),Z^{(i-1)}(s,t))|\rd s\Bigr)^2\rd t\Bigr]^{1/2}\\
	&\leq C_{\eta,\lambda}(\gamma-\gamma_0)\Bigl\{\bigl([K_{g,y}]_1(\eta+\lambda)+[K_{g,z_2}]_2(\eta+\lambda)\bigr)\bE\Bigl[\int^\infty_0e^{2\eta t}|Y^{(i)}(t)-Y^{(i-1)}(t)|^2\rd t\Bigr]^{1/2}\\
	&\hspace{3cm}+[K_{g,z_1}]_2(\lambda)\bE\Bigl[\int^\infty_0e^{2\eta t}\int^\infty_t|Z^{(i)}(t,s)-Z^{(i-1)}(t,s)|^2\rd s\rd t\Bigr]^{1/2}\Bigr\}\\
	&\leq C_{\eta,\lambda}(\gamma-\gamma_0)\bigl\{[K_{g,y}]_1(\eta+\lambda)+[K_{g,z_1}]_2(\lambda)+[K_{g,z_2}]_2(\eta+\lambda)\bigr\}\\
	&\hspace{1cm}\times\bE\Bigl[\int^\infty_0e^{2\eta t}|Y^{(i)}(t)-Y^{(i-1)}(t)|^2\rd t+\int^\infty_0e^{2\eta t}\int^\infty_0|Z^{(i)}(t,s)-Z^{(i-1)}(t,s)|^2\rd s\rd t\Bigr]^{1/2}.
\end{align*}
Since $(\eta,\lambda)\in\cR_g$ with $\cR_g$ defined by \eqref{domain: g}, and $\gamma_0\leq\gamma\leq\gamma_0+C^{-1}_{\eta,\lambda}$, we have $C_{\eta,\lambda}(\gamma-\gamma_0)\bigl\{[K_{g,y}]_1(\eta+\lambda)+[K_{g,z_1}]_2(\lambda)+[K_{g,z_2}]_2(\eta+\lambda)\bigr\}<1$. Thus the sequence $\{(Y^{(i)}(\cdot),Z^{(i)}(\cdot,\cdot))\}_{i\in\bN}$ is Cauchy in $\cM^{2,\eta}_\bF(0,\infty;\bR^m\times\bR^{m\times d})$. Denote the limit by $(Y^{(\infty)}(\cdot),Z^{(\infty)}(\cdot,\cdot))\in\cM^{2,\eta}_\bF(0,\infty;\bR^m\times\bR^{m\times d})$. Again by Lemma~\ref{lemm: coefficients BSVIE}, we see that
\begin{align*}
	&\Bigl(\int^\infty_te^{-\lambda(s-t)}g(t,s,Y^{(i)}(s),Z^{(i)}(t,s),Z^{(i)}(s,t))\rd s\Bigr)_{t\geq0}\\
	&\to\Bigl(\int^\infty_te^{-\lambda(s-t)}g(t,s,Y^{(\infty)}(s),Z^{(\infty)}(t,s),Z^{(\infty)}(s,t))\rd s\Bigr)_{t\geq0}
\end{align*}
in $L^{2,\eta}_{\cF_\infty}(0,\infty;\bR^m)$ as $i\to\infty$. Therefore, it holds that
\begin{equation*}
	Y^{(\infty)}(t)=\psi(t)+\int^\infty_te^{-\lambda(s-t)}\gamma g(t,s,Y^{(\infty)}(s),Z^{(\infty)}(t,s),Z^{(\infty)}(s,t))\rd s-\int^\infty_tZ^{(\infty)}(t,s)\rd W(s)
\end{equation*}
for a.e.\ $t\geq0$, a.s., and thus the pair $(Y^{(\infty)}(\cdot),Z^{(\infty)}(\cdot,\cdot))\in\cM^{2,\eta}_\bF(0,\infty;\bR^m\times\bR^{m\times d})$ is an adapted M-solution of the infinite horizon BSVIE~\eqref{Type-II alpha} with the parameter $\gamma$. Furthermore, by Proposition~\ref{prop: a priori estimate BSVIE}, the adapted M-solution is unique. Since $\psi(\cdot)\in L^{2,\eta}_{\cF_\infty}(0,\infty;\bR^m)$ is arbitrary, we see that the property (P$_\gamma$) holds. This completes the proof.
\end{proof}

Now we are ready to prove the well-posedness of infinite horizon BSVIE~\eqref{BSVIE}.


\begin{theo}\label{theo: well-posedness BSVIE}
Suppose that Assumption~\ref{assum: BSVIE} holds, and fix $(\eta,\lambda)\in\cR_g$. Then for any $\psi(\cdot)\in L^{2,\eta}_{\cF_\infty}(0,\infty;\bR^m)$, there exists a unique adapted M-solution $(Y(\cdot),Z(\cdot,\cdot))\in\cM^{2,\eta}_\bF(0,\infty;\bR^m\times\bR^{m\times d})$ of infinite horizon BSVIE~\eqref{BSVIE}.
\end{theo}


\begin{proof}
By Lemma~\ref{lemm: trivial BSVIE}, the property (P$_0$) holds. By Lemma~\ref{lemm: continuation}, we see that the property (P$_{\gamma_1}$) holds for any $\gamma_1\in[0,1]$ with $0\leq\gamma_1\leq C^{-1}_{\eta,\lambda}$. Again by Lemma~\ref{lemm: continuation}, (P$_{\gamma_2}$) holds for any $\gamma_2\in[0,1]$ with $\gamma_1\leq\gamma_2\leq\gamma_1+C^{-1}_{\eta,\lambda}$. By repeating this procedure, we see that the property (P$_1$) holds. This completes the proof.
\end{proof}


\begin{rem}\label{rem: BSVIE nonzero}
As in the SVIEs case, the above results can be easily generalized to the case where $g(t,s,0,0,0)$ is nonzero. Indeed, suppose that $g$ satisfies Assumption~\ref{assum: BSVIE} (i), (ii), (iii), and
\begin{equation*}
	\bE\Bigl[\int^\infty_0e^{2\eta t}\Bigl(\int^\infty_te^{-\lambda(s-t)}|g(t,s,0,0,0)|\rd s\Bigr)^2\rd t\Bigr]<\infty
\end{equation*}
with $(\eta,\lambda)\in\cR_g$. For any given $\psi(\cdot)\in L^{2,\eta}_{\cF_\infty}(0,\infty;\bR^m)$, define
\begin{equation*}
	\tilde{\psi}(t):=\psi(t)+\int^\infty_te^{-\lambda(s-t)}g(t,s,0,0,0)\rd s\ \text{and}\ \tilde{g}(t,s,y,z_1,z_2):=g(t,s,y,z_1,z_2)-g(t,s,0,0,0)
\end{equation*}
for $(t,s,y,z_1,z_2)\in\Delta[0,\infty)\times\bR^m\times\bR^{m\times d}\times\bR^{m\times d}$. Then $\tilde{\psi}(\cdot)\in L^{2,\eta}_{\cF_\infty}(0,\infty;\bR^m)$, and $\tilde{g}$ satisfies all the conditions in Assumption~\ref{assum: BSVIE} with the functions $K_{g,y}$, $K_{g,z_1}$ and $K_{g,z_2}$. Thus, there exists a unique adapted M-solution $(Y(\cdot),Z(\cdot,\cdot))\in\cM^{2,\eta}_\bF(0,\infty;\bR^m\times\bR^{m\times d})$ to the infinite horizon BSVIE
\begin{equation*}
	Y(t)=\tilde{\psi}(t)+\int^\infty_te^{-\lambda(s-t)}\tilde{g}(t,s,Y(s),Z(t,s),Z(s,t))\rd s-\int^\infty_tZ(t,s)\rd W(s),\ t\geq0.
\end{equation*}
Clearly, $(Y(\cdot),Z(\cdot,\cdot))$ is the unique adapted M-solution of the original infinite horizon BSVIE~\eqref{BSVIE}.
\end{rem}


\subsection{Convergence of finite horizon BSVIEs}


In this subsection, we show that the adapted M-solutions to finite horizon BSVIEs converge to the one of the original infinite horizon BSVIE~\eqref{BSVIE}. For each $\psi(\cdot)\in L^{2,\eta}_{\cF_\infty}(0,\infty;\bR^m)$ and $T>0$, define
\begin{equation*}
	\psi_T(t):=\bE_T\bigl[\psi(t)\bigr]\1_{[0,T]}(t),\ t\geq0.
\end{equation*}
We consider the following BSVIE on $[0,T]$:
\begin{equation*}
	\tilde{Y}_T(t)=\psi_T(t)+\int^T_te^{-\lambda(s-t)}g(t,s,\tilde{Y}_T(s),\tilde{Z}_T(t,s),\tilde{Z}_T(s,t))\rd s-\int^T_t\tilde{Z}_T(t,s)\rd W(s),\ t\in[0,T].
\end{equation*}
The above BSVIE has a unique adapted M-solution $(\tilde{Y}_T(\cdot),\tilde{Z}_T(\cdot,\cdot))$ on $[0,T]$ (see Yong~\cite{Yo08}). We extend it to infinite horizon processes:
\begin{equation*}
\begin{cases}
	Y_T(t):=\tilde{Y}_T(t)\1_{[0,T]}(t),\ t\in[0,\infty),\\
	Z_T(t,s):=\tilde{Z}_T(t,s)\1_{[0,T]^2}(t,s),\ (t,s)\in[0,\infty)^2.
\end{cases}
\end{equation*}


\begin{theo}\label{theo: finite horizon}
Suppose that Assumption~\ref{assum: BSVIE} holds, and fix $(\eta,\lambda)\in\cR_g$ with $\cR_g$ defined by \eqref{domain: g}. Then for any $\psi(\cdot)\in L^{2,\eta}_{\cF_\infty}(0,\infty;\bR^m)$, it holds that
\begin{equation*}
	\lim_{T\to\infty}\bE\Bigl[\int^\infty_0e^{2\eta t}|Y(t)-Y_T(t)|^2\rd t+\int^\infty_0e^{2\eta t}\int^\infty_0|Z(t,s)-Z_T(t,s)|^2\rd s\rd t\Bigr]=0.
\end{equation*}
\end{theo}


\begin{proof}
Let $T>0$ be fixed. Clearly, $Y_T(\cdot)\in L^{2,\eta}_\bF(0,\infty;\bR^m)$, $Z_T(\cdot,\cdot)\in \sL^{2,\eta}_\bF(0,\infty;\bR^{m\times d})$, and the following relation holds:
\begin{equation*}
	Y_T(t)=\bE[Y_T(t)]+\int^t_0Z_T(t,s)\rd W(s)
\end{equation*}
for a.e.\ $t\geq0$, a.s. Thus the pair $(Y_T(\cdot),Z_T(\cdot,\cdot))$ is in $\cM^{2,\eta}_\bF(0,\infty;\bR^m\times\bR^{m\times d})$. Furthermore, noting that $\psi_T(t)=0$ for $t>T$ and $g(t,s,0,0,0)=0$ for a.e.\ $(t,s)\in\Delta[0,\infty)$, a.s., we see that
\begin{equation*}
	Y_T(t)=\psi_T(t)+\int^\infty_te^{-\lambda(s-t)}g(t,s,Y_T(s),Z_T(t,s),Z_T(s,t))\rd s-\int^\infty_tZ_T(t,s)\rd W(s)
\end{equation*}
for a.e.\ $t\geq0$, a.s. This implies that $(Y_T(\cdot),Z_T(\cdot,\cdot))\in\cM^{2,\eta}_\bF(0,\infty;\bR^m\times\bR^{m\times d})$ is the (unique) adapted M-solution of the infinite horizon BSVIE~\eqref{BSVIE} with the free term $\psi_T(\cdot)$ and the discount rate $\lambda$. By Proposition~\ref{prop: a priori estimate BSVIE}, we see that
\begin{align*}
	&\bE\Bigl[\int^\infty_0e^{2\eta t}|Y(t)-Y_T(t)|^2\rd t+\int^\infty_0e^{2\eta t}\int^\infty_0|Z(t,s)-Z_T(t,s)|^2\rd s\rd t\Bigr]^{1/2}\\
	&\leq C_{\eta,\lambda}\bE\Bigl[\int^\infty_0e^{2\eta t}\bigl|\psi(t)-\psi_T(t)\bigr|^2\rd t\Bigr]^{1/2}=C_{\eta,\lambda}\bE\Bigl[\int^\infty_0e^{2\eta t}\bigl|\psi(t)-\bE_T[\psi(t)]\1_{[0,T]}(t)\bigr|^2\rd t\Bigr]^{1/2},
\end{align*}
where $C_{\eta,\lambda}\in(0,\infty)$ is the constant defined in Proposition~\ref{prop: a priori estimate BSVIE}. Since $\psi(\cdot)\in L^{2,\eta}_{\cF_\infty}(0,\infty;\bR^m)$, the last expectation tends to zero as $T\to\infty$. This completes the proof.
\end{proof}


\begin{rem}
From the above result, we can extend some important properties of finite horizon BSVIEs to the infinite horizon case. For example, under suitable assumptions, comparison theorems for infinite horizon BSVIEs can be proved by considering the corresponding convergent sequence of finite horizon BSVIEs. We do not come into this topic in this paper. For detailed discussions on comparison theorems for finite horizon BSVIEs, see Wang and Yong~\cite{WaTYo15}.
\end{rem}


\subsection{A variation of constant formula}


Consider the following infinite horizon linear Type-I BSVIE:
\begin{equation}\label{linear}
	Y(t)=\psi(t)+\int^\infty_te^{-\lambda(s-t)}\Bigl\{A(t,s)Y(s)+\sum^d_{k=1}B_k(t,s)Z^k(t,s)\Bigr\}\rd s-\int^\infty_tZ(t,s)\rd W(s),\ t\geq0.
\end{equation}
We impose the following assumptions on the coefficients $A,B_k$, $k=1,\dots,d$:


\begin{assum}\label{assum: linear}
$A,B_k:\Omega\times\Delta[0,\infty)\to\bR^{m\times m}$, $k=1,\dots,d$, are measurable; $A(t,\cdot)$, $B_k(t,\cdot)$, $k=1,\dots,d$, are adapted for each $t\geq0$; there exist $K_A\in L^{1,*}(0,\infty;\bR_+)$ and $K_B\in L^{2,*}(0,\infty;\bR_+)$ such that
\begin{equation*}
	|A(t,s)|\leq K_A(s-t)\ \text{and}\ \Bigl(\sum^d_{k=1}|B_k(t,s)|^2\Bigr)^{1/2}\leq K_B(s-t)
\end{equation*}
for a.e.\ $(t,s)\in\Delta[0,\infty)$, a.s.
\end{assum}

As before, we define
\begin{equation}\label{domain: linear}
	\cR_{A,B}:=\{(\eta,\lambda)\in\bR^2\,|\,[K_A]_1(\eta+\lambda)+[K_B]_2(\lambda)<1\}.
\end{equation}
By Theorem~\ref{theo: well-posedness BSVIE}, under Assumption~\ref{assum: linear}, for any $\psi(\cdot)\in L^{2,\eta}_{\cF_\infty}(0,\infty;\bR^m)$ with $(\eta,\lambda)\in\cR_{A,B}$, there exists a unique adapted M-solution $(Y(\cdot),Z(\cdot,\cdot))\in\cM^{2,\eta}_\bF(0,\infty;\bR^m\times\bR^{m\times d})$ of infinite horizon linear Type-I BSVIE~\eqref{linear}.

The objective of this subsection is to establish a \emph{variation of constant formula} for \eqref{linear}, which provides an explicit form of $Y(\cdot)$. To do so, we introduce the following notations:
\begin{itemize}
\item
$\mathbb{B}$ denotes the Banach space consisting of (jointly) measurable maps $\zeta:\Omega\times\Delta[0,\infty)\to\bR^{m\times m}$ such that $\zeta(t,\cdot)$ is adapted for a.e.\ $t\geq0$, $[t,\infty)\ni s\mapsto \zeta(t,s)\in L^2_{\cF_\infty}(\Omega;\bR^{m\times m})$ is continuous for a.e.\ $t\geq0$, and the norm $\|\zeta(\cdot,\cdot)\|_\mathbb{B}:=\mathrm{ess\,sup}_{t\in[0,\infty)}\sup_{s\in[t,\infty)}\bE\bigl[|\zeta(t,s)|^2\bigr]^{1/2}$ is finite.
\item
For a.e.\ $t\in[0,\infty)$, define $\Phi(t,\cdot)$ as the solution of the SDE
\begin{equation}\label{linear SDE}
	\Phi(t,s)=I_m+\sum^d_{k=1}\int^s_te^{-\lambda(r-t)}\Phi(t,r)B_k(t,r)\rd W_k(r),\ s\geq t,
\end{equation}
where $I_m\in\bR^{m\times m}$ is the identity matrix.
\item
$\Xi_1(t,s):=\Phi(t,s)A(t,s)$ and
\begin{equation*}
	\Xi_{i+1}(t,s):=\int^s_t\Xi_i(t,r)\Xi_1(r,s)\rd r,\ i\in\bN,
\end{equation*}
for $(t,s)\in\Delta[0,\infty)$.
\item
$R(t,s):=\sum^\infty_{i=1}\Xi_i(t,s)$ for $(t,s)\in\Delta[0,\infty)$.
\end{itemize}
Since the coefficients $A,B_k$, $k=1,\dots,d$, are possibly unbounded, we have to treat some estimates carefully. The following lemma justifies the above notations.


\begin{lemm}\label{lemm: linear}
Suppose that measurable maps $A,B_k:\Omega\times\Delta[0,\infty)\to\bR^{m\times m}$, $k=1,\dots,d$, satisfy Assumption~\ref{assum: linear}. Let $(\eta,\lambda)\in\cR_{A,B}$, where $\cR_{A,B}$ is defined by \eqref{domain: linear}.
\begin{itemize}
\item[(i)]
There exists a unique process $\Phi(\cdot,\cdot)\in\mathbb{B}$ satisfying \eqref{linear SDE} for any $s\geq0$, for a.e.\ $t\geq0$, a.s. Moreover, for a.e.\ $t\in[0,\infty)$, $\Phi(t,\cdot)$ is a square integrable martingale, and thus the limit $\Phi(t,\infty):=\lim_{s\to\infty}\Phi(t,s)$ exists in $L^2_{\cF_\infty}(\Omega;\bR^{m\times m})$. Furthermore, the following estimate holds:
\begin{equation}\label{lemm: linear: 1}
	\bE_t\bigl[|\Phi(t,s)|^2_\mathrm{op}\bigr]^{1/2}\leq\frac{1}{1-[K_B]_2(\lambda)},\ \forall\,s\in[t,\infty),\ \text{for a.e.}\ t\in[0,\infty),\ \text{a.s.},
\end{equation}
where, for each $\mathcal{A}\in\bR^{m\times m}$, $|\mathcal{A}|_\mathrm{op}$ denotes the operator norm of $\mathcal{A}$ as a linear operator on $\bR^m$.
\item[(ii)]
For any $\zeta(\cdot)\in L^{2,\eta}_\bF(0,\infty;\bR^m)$, it holds that
\begin{equation}\label{lemm: linear: 2}
\begin{split}
	&\sum^\infty_{i=1}\bE\Bigl[\int^\infty_0e^{2\eta t}\bE_t\Bigl[\int^\infty_te^{-\lambda(s-t)}|\Xi_i(t,s)|\,|\zeta(s)|\rd s\Bigr]^2\rd t\Bigr]^{1/2}\\
	&\leq\sum^\infty_{i=1}\Bigl(\frac{[K_A]_1(\eta+\lambda)}{1-[K_B]_2(\lambda)}\Bigr)^i\bE\Bigl[\int^\infty_0e^{2\eta t}|\zeta(t)|^2\rd t\Bigr]<\infty.
\end{split}
\end{equation}
In particular, $R(t,s)$ is well-defined for a.e.\ $(t,s)\in\Delta[0,\infty)$, a.s.
\end{itemize}
\end{lemm}


\begin{proof}
First, we prove the assertion (i). Observe that, for each $\Phi(\cdot,\cdot)\in\mathbb{B}$,
\begin{align*}
	\sum^d_{k=1}\bE\Bigl[\int^\infty_t|e^{-\lambda(s-t)}\Phi(t,s)B_k(t,s)|^2\rd s\Bigr]&\leq\|\Phi(\cdot,\cdot)\|^2_\mathbb{B}\int^\infty_te^{-2\lambda(s-t)}K_B(s-t)^2\rd s\\
	&=[K_B]_2(\lambda)^2\|\Phi(\cdot,\cdot)\|^2_\mathbb{B}<\infty
\end{align*}
for a.e.\ $t\geq0$. Thus the stochastic integral $\sum^d_{k=1}\int^\cdot_te^{-\lambda(r-t)}\Phi(t,r)B_k(t,r)\rd W_k(r)$ is well-defined for a.e.\ $t\geq0$. By Lemma~2.A in \cite{BeMi80}, there exists a jointly measurable version of the map
\begin{equation*}
	\Omega\times\Delta[0,\infty)\ni(\omega,t,s)\mapsto\sum^d_{k=1}\int^s_te^{-\lambda(r-t)}\Phi(t,r)B_k(t,r)\rd W_k(r)(\omega).
\end{equation*}
We define $\tilde{\Phi}(\cdot,\cdot)\in\mathbb{B}$ by
\begin{equation*}
	\tilde{\Phi}(t,s):=I_m+\sum^d_{k=1}\int^s_te^{-\lambda(r-t)}\Phi(t,r)B_k(t,r)\rd W_k(r),\ (t,s)\in\Delta[0,\infty).
\end{equation*}
For another $\Phi'(\cdot,\cdot)\in\mathbb{B}$, define $\tilde{\Phi}'(\cdot,\cdot)\in\mathbb{B}$ by the same manner. Then, for a.e.\ $t\geq0$,
\begin{align*}
	\sup_{s\in[t,\infty)}\bE\bigl[|\tilde{\Phi}(t,s)-\tilde{\Phi}'(t,s)|^2\bigr]&=\sum^d_{k=1}\bE\Bigl[\int^\infty_t|e^{-\lambda(s-t)}\bigl\{\Phi(t,s)-\Phi'(t,s)\bigr\}B_k(t,s)|^2\rd s\Bigr]\\
	&\leq [K_B]_2(\lambda)^2\|\Phi(\cdot,\cdot)-\Phi'(\cdot,\cdot)\|^2_\mathbb{B}.
\end{align*}
Thus, we obtain
\begin{equation*}
	\|\tilde{\Phi}(\cdot,\cdot)-\tilde{\Phi}'(\cdot,\cdot)\|_\mathbb{B}\leq[K_B]_2(\lambda)\|\Phi(\cdot,\cdot)-\Phi'(\cdot,\cdot)\|_\mathbb{B}.
\end{equation*}
Since $[K_B]_2(\lambda)<1$, the map $\Phi(\cdot,\cdot)\mapsto\tilde{\Phi}(\cdot,\cdot)$ is a contraction map on the Banach space $\mathbb{B}$. Thus, the map admits a unique fixed point in $\mathbb{B}$, which is the solution of \eqref{linear SDE}. Denote the unique solution by $\Phi(\cdot,\cdot)$. Clearly, for a.e.\ $t\in[0,\infty)$, $\Phi(t,\cdot)$ is a square integrable martingale, and thus the limit $\Phi(t,\infty):=\lim_{s\to\infty}\Phi(t,s)$ exists in $L^2_{\cF_\infty}(\Omega;\bR^{m\times m})$. Now we prove the estimate \eqref{lemm: linear: 1}. We note that, for each $\mathcal{A},\mathcal{B}\in\bR^{m\times m}$, $|\mathcal{A}|_\mathrm{op}\leq|\mathcal{A}|$ and $|\mathcal{A}\mathcal{B}|\leq|\mathcal{A}|_\mathrm{op}|\mathcal{B}|$ (recall that $|\cdot|$ denotes the Frobenius norm, and $|\cdot|_\mathrm{op}$ denotes the operator norm). By using these inequalities, we have, for a.e.\ $t\geq0$, a.s.,
\begin{align*}
	\sup_{s\in[t,\infty)}\bE_t\bigl[|\Phi(t,s)|^2_\mathrm{op}\bigr]^{1/2}&=\sup_{s\in[t,\infty)}\bE_t\Bigl[\Bigl|I_m+\sum^d_{k=1}\int^s_te^{-\lambda(r-t)}\Phi(t,r)B_k(t,r)\rd W_k(r)\Bigr|^2_\mathrm{op}\Bigr]^{1/2}\\
	&\leq1+\sup_{s\in[t,\infty)}\bE_t\Bigl[\Bigl|\sum^d_{k=1}\int^s_te^{-\lambda(r-t)}\Phi(t,r)B_k(t,r)\rd W_k(r)\Bigr|^2\Bigr]^{1/2}\\
	&=1+\sup_{s\in[t,\infty)}\bE_t\Bigl[\sum^d_{k=1}\int^s_te^{-2\lambda(r-t)}|\Phi(t,r)B_k(t,r)\bigr|^2\rd r\Bigr]^{1/2}\\
	&\leq1+\sup_{s\in[t,\infty)}\bE_t\Bigl[\int^s_te^{-2\lambda(r-t)}|\Phi(t,r)|^2_\mathrm{op}\sum^d_{k=1}|B_k(t,r)\bigr|^2\rd r\Bigr]^{1/2}\\
	&\leq1+\sup_{s\in[t,\infty)}\bE_t\bigl[|\Phi(t,s)|^2_\mathrm{op}\bigr]^{1/2}\Bigl(\int^\infty_te^{-2\lambda(r-t)}K_B(r-t)^2\rd r\Bigr)^{1/2}\\
	&=1+[K_B]_2(\lambda)\sup_{s\in[t,\infty)}\bE_t\bigl[|\Phi(t,s)|^2_\mathrm{op}\bigr]^{1/2}.
\end{align*}
Thus, the estimate \eqref{lemm: linear: 1} holds.

Next, we prove the assertion (ii). We show that, for any $i\in\bN$, the following estimate holds:
\begin{equation}\label{proof: linear: 1}
	e^{-(\eta+\lambda)(s-t)}\bE_t\bigl[|\Xi_i(t,s)|^2\bigr]^{1/2}\leq\beta^{*i}_{\eta,\lambda}(s-t)
\end{equation}
for a.e.\ $(t,s)\in\Delta[0,\infty)$, a.s., where
\begin{equation*}
	\beta_{\eta,\lambda}(\tau):=\frac{1}{1-[K_B]_2(\lambda)}e^{-(\eta+\lambda)\tau}K_A(\tau),\ \tau\geq0,
\end{equation*}
and $\beta^{*i}_{\eta,\lambda}$ denotes the $i$-times convolution of $\beta_{\eta,\lambda}$. For $i=1$, by using the estimate \eqref{lemm: linear: 1}, we see that
\begin{align*}
	e^{-(\eta+\lambda)(s-t)}\bE_t\bigl[|\Xi_1(t,s)|^2\bigr]^{1/2}&\leq e^{-(\eta+\lambda)(s-t)}\bE_t\bigl[|\Phi(t,s)|^2_\mathrm{op}|A(t,s)|^2\bigr]^{1/2}\\
	&\leq\bE_t\bigl[|\Phi(t,s)|^2_\mathrm{op}\bigr]^{1/2}e^{-(\eta+\lambda)(s-t)}K_A(s-t)\\
	&\leq\frac{1}{1-[K_B]_2(\lambda)}e^{-(\eta+\lambda)(s-t)}K_A(s-t)=\beta_{\eta,\lambda}(s-t)
\end{align*}
for a.e.\ $(t,s)\in\Delta[0,\infty)$, a.s. For a fixed $j\in\bN$, assume that \eqref{proof: linear: 1} holds for $i=j$. By using the (conditional) Minkowski's inequality, we have
\begin{align*}
	e^{-(\eta+\lambda)(s-t)}\bE_t\bigl[|\Xi_{j+1}(t,s)|^2\bigr]^{1/2}&=e^{-(\eta+\lambda)(s-t)}\bE_t\Bigl[\Bigl|\int^s_t\Xi_j(t,r)\Xi_1(r,s)\rd r\Bigr|^2\Bigr]^{1/2}\\
	&\leq e^{-(\eta+\lambda)(s-t)}\int^s_t\bE_t\bigl[|\Xi_j(t,r)\Xi_1(r,s)|^2\bigr]^{1/2}\rd r\\
	&\leq\int^s_te^{-(\eta+\lambda)(r-t)}\bE_t\Bigl[|\Xi_j(t,r)|^2e^{-2(\eta+\lambda)(s-r)}\bE_r\bigl[|\Xi_1(r,s)|^2\bigr]\Bigr]^{1/2}\rd r\\
	&\leq\int^s_te^{-(\eta+\lambda)(r-t)}\bE_t\bigl[|\Xi_j(t,r)|^2\bigr]^{1/2}\beta_{\eta,\lambda}(s-r)\rd r\\
	&\leq\int^s_t\beta^{*j}_{\eta,\lambda}(r-t)\beta_{\eta,\lambda}(s-r)\rd r=\beta^{*(j+1)}_{\eta,\lambda}(s-t)
\end{align*}
for a.e.\ $(t,s)\in\Delta[0,\infty)$, a.s. By the induction, the estimate \eqref{proof: linear: 1} holds for any $i\in\bN$.

Let $\zeta(\cdot)\in L^{2,\eta}_\bF(0,\infty;\bR^m)$ be fixed. For each $i\in\bN$, we have
\begin{align*}
	&\bE\Bigl[\int^\infty_0e^{2\eta t}\bE_t\Bigl[\int^\infty_te^{-\lambda(s-t)}|\Xi_i(t,s)|\,|\zeta(s)|\rd s\Bigr]^2\rd t\Bigr]^{1/2}\\
	&\leq\bE\Bigl[\int^\infty_0e^{2\eta t}\Bigl(\int^\infty_te^{-\lambda(s-t)}\bE_t\bigl[|\Xi_i(t,s)|^2\bigr]^{1/2}\bE_t\bigl[|\zeta(s)|^2\bigr]^{1/2}\rd s\Bigr)^2\rd t\Bigr]^{1/2}\\
	&=\bE\Bigl[\int^\infty_0\Bigl(\int^\infty_te^{-(\eta+\lambda)(s-t)}\bE_t\bigl[|\Xi_i(t,s)|^2\bigr]^{1/2}\bE_t\bigl[e^{2\eta s}|\zeta(s)|^2\bigr]^{1/2}\rd s\Bigr)^2\rd t\Bigr]^{1/2}\\
	&\leq\bE\Bigl[\int^\infty_0\Bigl(\int^\infty_t\beta^{*i}_{\eta,\lambda}(s-t)\bE_t\bigl[e^{2\eta s}|\zeta(s)|^2\bigr]^{1/2}\rd s\Bigr)^2\rd t\Bigr]^{1/2}\\
	&\leq\bE\Bigl[\int^\infty_0\int^\infty_t\beta^{*i}_{\eta,\lambda}(s-t)\rd s\,\int^\infty_t\beta^{*i}_{\eta,\lambda}(s-t)\bE_t\bigl[e^{2\eta s}|\zeta(s)|^2\bigr]\rd s\rd t\Bigr]^{1/2}\\
	&=\Bigl(\int^\infty_0\beta^{*i}_{\eta,\lambda}(t)\rd t\Bigr)^{1/2}\bE\Bigl[\int^\infty_0\int^\infty_t\beta^{*i}_{\eta,\lambda}(s-t)e^{2\eta s}|\zeta(s)|^2\rd s\rd t\Bigr]^{1/2}\\
	&\leq\int^\infty_0\beta^{*i}_{\eta,\lambda}(t)\rd t\,\bE\Bigl[\int^\infty_0e^{2\eta t}|\zeta(t)|^2\rd t\Bigr]^{1/2},
\end{align*}
where, in the first inequality, we used the Cauchy--Schwarz inequality for the conditional expectation $\bE_t[\cdot]$; in the second inequality, we used the estimate \eqref{proof: linear: 1}; in the third inequality, we used the Cauchy--Schwarz inequality for the integral $\int^\infty_t\cdots\rd s$; in the last inequality, we used Young's convolution inequality. Observe that
\begin{equation*}
	\int^\infty_0\beta^{*i}_{\eta,\lambda}(t)\rd t\leq\Bigl(\int^\infty_0\beta_{\eta,\lambda}(t)\rd t\Bigr)^i=\Bigl(\frac{[K_A]_1(\eta+\lambda)}{1-[K_B]_2(\lambda)}\Bigr)^i
\end{equation*}
for each $i\in\bN$. Since $(\eta,\lambda)\in\cR_{A,B}$ with $\cR_{A,B}$ defined by \eqref{domain: linear}, we see that $[K_A]_1(\eta+\lambda)/(1-[K_B]_2(\lambda))<1$. Thus the estimate \eqref{lemm: linear: 2} holds. This completes the proof.
\end{proof}

Now we provide a variation of constant formula for the adapted M-solution to infinite horizon linear Type-I BSVIE~\eqref{linear}.


\begin{theo}\label{theo: explicit}
Suppose that Assumption~\ref{assum: linear} holds, and fix $(\eta,\lambda)\in\cR_{A,B}$. For a given $\psi(\cdot)\in L^{2,\eta}_{\cF_\infty}(0,\infty;\bR^m)$, let $(Y(\cdot),Z(\cdot,\cdot))\in\cM^{2,\eta}_\bF(0,\infty;\bR^m\times\bR^{m\times d})$ be the adapted M-solution of the infinite horizon linear Type-I BSVIE~\eqref{linear}. Then it holds that
\begin{equation*}
	Y(t)=\bE_t\Bigl[\Phi(t,\infty)\psi(t)+\int^\infty_te^{-\lambda(s-t)}R(t,s)\Phi(s,\infty)\psi(s)\rd s\Bigr]
\end{equation*}
for a.e.\ $t\geq0$, a.s.
\end{theo}


\begin{proof}
Define
\begin{equation*}
	Y(t,s):=\bE_s\Bigl[\psi(t)+\int^\infty_se^{-\lambda(r-t)}\Bigl\{A(t,r)Y(r)+\sum^d_{k=1}B_k(t,r)Z^k(t,r)\Bigr\}\rd r\Bigr]
\end{equation*}
for $(t,s)\in\Delta[0,\infty)$. For each $(t,s)\in\Delta[0,\infty)$, by taking the conditional expectation $\bE_s[\cdot]$ on both sides of \eqref{linear}, we have
\begin{align*}
	Y(t)&=\bE_s\Bigl[\psi(t)+\int^\infty_te^{-\lambda(r-t)}\Bigl\{A(t,r)Y(r)+\sum^d_{k=1}B_k(t,r)Z^k(t,r)\Bigr\}\rd r-\int^\infty_tZ(t,r)\rd W(r)\Bigr]\\
	&=Y(t,s)+\int^s_te^{-\lambda(r-t)}\Bigl\{A(t,r)Y(r)+\sum^d_{k=1}B_k(t,r)Z^k(t,r)\Bigr\}\rd r-\int^s_tZ(t,r)\rd W(r),
\end{align*}
and hence
\begin{equation*}
	Y(t,s)=Y(t)-\int^s_te^{-\lambda(r-t)}\Bigl\{A(t,r)Y(r)+\sum^d_{k=1}B_k(t,r)Z^k(t,r)\Bigr\}\rd r+\int^s_tZ(t,r)\rd W(r).
\end{equation*}
For a.e.\ $t\geq0$, by using It\^{o}'s formula for the product $\Phi(t,\cdot)Y(t,\cdot)$, we obtain
\begin{equation}\label{proof: explicit: 1}
\begin{split}
	&\Phi(t,s)Y(t,s)=Y(t)-\int^s_te^{-\lambda(r-t)}\Phi(t,r)A(t,r)Y(r)\rd r\\
	&\hspace{3cm}+\sum^d_{k=1}\int^s_t\bigl\{\Phi(t,r)Z^k(t,r)+e^{-\lambda(r-t)}\Phi(t,r)B_k(t,r)Y(t,r)\bigr\}\rd W_k(r)
\end{split}
\end{equation}
for any $s\geq t$, a.s. Note that the left-hand side tends to $\Phi(t,\infty)\psi(t)$ in $L^1_{\cF_\infty}(\Omega;\bR^m)$ as $s\to\infty$. Furthermore, noting that $\bE\bigl[\sup_{s\in[t,\infty)}|\Phi(t,s)|^2\bigr]<\infty$ and $\bE\bigl[\sup_{s\in[t,\infty)}|Y(t,s)|^2\bigr]<\infty$, we have
\begin{equation*}
	\bE\Bigl[\Bigl(\sum^d_{k=1}\int^\infty_t|\Phi(t,s)Z^k(t,s)|^2\rd s\Bigr)^{1/2}\Bigr]\leq\bE\Bigl[\sup_{s\in[t,\infty)}|\Phi(t,s)|^2\Bigr]^{1/2}\bE\Bigl[\int^\infty_t|Z(t,s)|^2\rd s\Bigr]^{1/2}<\infty
\end{equation*}
and
\begin{align*}
	&\bE\Bigl[\Bigl(\sum^d_{k=1}\int^\infty_t\bigl|e^{-\lambda(s-t)}\Phi(t,s)B_k(t,s)Y(t,s)\bigr|^2\rd s\Bigr)^{1/2}\Bigr]\\
	&\leq\bE\Bigl[\Bigl(\int^\infty_t|\Phi(t,s)|^2|Y(t,s)|^2e^{-2\lambda(s-t)}K_B(s-t)^2\rd s\Bigr)^{1/2}\Bigr]\\
	&\leq\bE\Bigl[\sup_{s\in[t,\infty)}|\Phi(t,s)|^2\Bigr]^{1/2}\bE\Bigl[\sup_{s\in[t,\infty)}|Y(t,s)|^2\Bigr]^{1/2}\Bigl(\int^\infty_te^{-2\lambda(s-t)}K_B(s-t)^2\rd s\Bigr)^{1/2}<\infty.
\end{align*}
Thus, the stochastic integral in the right-hand side of \eqref{proof: explicit: 1} is a uniformly integrable martingale (with respect to the time parameter $s\in[t,\infty)$). By letting $s\to\infty$ and taking the conditional expectation $\bE_t[\cdot]$ on both sides of \eqref{proof: explicit: 1}, we obtain
\begin{equation*}
	\bE_t\bigl[\Phi(t,\infty)\psi(t)\bigr]=Y(t)-\bE_t\Bigl[\int^\infty_te^{-\lambda(s-t)}\Phi(t,s)A(t,s)Y(s)\rd s\Bigr].
\end{equation*}
Hence, by using the notation $\Xi_1(t,s):=\Phi(t,s)A(t,s)$,
\begin{equation}\label{proof: explicit: 2}
	Y(t)=\bE_t\Bigl[\Phi(t,\infty)\psi(t)+\int^\infty_te^{-\lambda(s-t)}\Xi_1(t,s)Y(s)\rd s\Bigr]
\end{equation}
for a.e.\ $t\geq0$, a.s.

We shall show that, for each $N\in\bN$,
\begin{equation}\label{proof: explicit: 3}
	Y(t)=\bE_t\Bigl[\Phi(t,\infty)\psi(t)+\int^\infty_te^{-\lambda(s-t)}R_N(t,s)\bE_s\bigl[\Phi(s,\infty)\psi(s)\bigr]\rd s+\int^\infty_te^{-\lambda(s-t)}\Xi_{N+1}(t,s)Y(s)\rd s\Bigr]
\end{equation}
for a.e.\ $t\geq0$, a.s., where $R_N(t,s):=\sum^N_{i=1}\Xi_i(t,s)$ for $(t,s)\in\Delta[0,\infty)$. By using \eqref{proof: explicit: 2} and Fubini's theorem, we have
\begin{align*}
	Y(t)&=\bE_t\Bigl[\Phi(t,\infty)\psi(t)+\int^\infty_te^{-\lambda(s-t)}\Xi_1(t,s)\bE_s\bigl[\Phi(s,\infty)\psi(s)\bigr]\rd s\\
	&\hspace{2cm}+\int^\infty_te^{-\lambda(s-t)}\Xi_1(t,s)\bE_s\Bigl[\int^\infty_se^{-\lambda(r-s)}\Xi_1(s,r)Y(r)\rd r\Bigr]\rd s\Bigr]\\
	&=\bE_t\Bigl[\Phi(t,\infty)\psi(t)+\int^\infty_te^{-\lambda(s-t)}R_1(t,s)\bE_s\bigl[\Phi(s,\infty)\psi(s)\bigr]\rd s\\
	&\hspace{2cm}+\int^\infty_t\Xi_1(t,s)\int^\infty_se^{-\lambda(r-t)}\Xi_1(s,r)Y(r)\rd r\rd s\Bigr]\\
	&=\bE_t\Bigl[\Phi(t,\infty)\psi(t)+\int^\infty_te^{-\lambda(s-t)}R_1(t,s)\bE_s\bigl[\Phi(s,\infty)\psi(s)]\rd s+\int^\infty_te^{-\lambda(s-t)}\Xi_2(t,s)Y(s)\rd s\Bigr]
\end{align*}
for a.e.\ $t\geq0$, a.s. Thus the equality \eqref{proof: explicit: 3} holds for $N=1$. Let $M\in\bN$ be fixed, and assume that \eqref{proof: explicit: 3} holds for $N=M$. Again by using \eqref{proof: explicit: 2} and Fubini's theorem, we have
\begin{align*}
	Y(t)&=\bE_t\Bigl[\Phi(t,\infty)\psi(t)+\int^\infty_te^{-\lambda(s-t)}R_M(t,s)\bE_s\bigl[\Phi(s,\infty)\psi(s)\bigr]\rd s\\
	&\hspace{2cm}+\int^\infty_te^{-\lambda(s-t)}\Xi_{M+1}(t,s)\bE_s\bigl[\Phi(s,\infty)\psi(s)\bigr]\rd s\\
	&\hspace{2cm}+\int^\infty_te^{-\lambda(s-t)}\Xi_{M+1}(t,s)\bE_s\Bigl[\int^\infty_se^{-\lambda(r-s)}\Xi_1(s,r)Y(r)\rd r\Bigr]\rd s\Bigr]\\
	&=\bE_t\Bigl[\Phi(t,\infty)\psi(t)+\int^\infty_te^{-\lambda(s-t)}R_{M+1}(t,s)\bE_s\bigl[\Phi(s,\infty)\psi(s)\bigr]\rd s\\
	&\hspace{2cm}+\int^\infty_t\Xi_{M+1}(t,s)\int^\infty_se^{-\lambda(r-t)}\Xi_1(s,r)Y(r)\rd r\rd s\Bigr]\\
	&=\bE_t\Bigl[\Phi(t,\infty)\psi(t)+\int^\infty_te^{-\lambda(s-t)}R_{M+1}(t,s)\bE_s\bigl[\Phi(s,\infty)\psi(s)\bigr]\rd s+\int^\infty_te^{-\lambda(s-t)}\Xi_{M+2}(t,s)Y(s)\rd s\Bigr]
\end{align*}
for a.e.\ $t\geq0$, a.s. Thus, the equality \eqref{proof: explicit: 3} holds for $N=M+1$. By the induction, we see that \eqref{proof: explicit: 3} holds for any $N\in\bN$.

Noting that the process $Y(\cdot)$ is in $L^{2,\eta}_\bF(0,\infty;\bR^m)$, by Lemma~\ref{lemm: linear} (ii), the process
\begin{equation*}
	\Bigl(\bE_t\Bigl[\int^\infty_te^{-\lambda(s-t)}\Xi_{N+1}(t,s)Y(s)\rd s\Bigr]\Bigr)_{t\geq0}
\end{equation*}
converges to zero in $L^{2,\eta}_\bF(0,\infty;\bR^m)$ as $N\to\infty$. On the other hand, by the estimate \eqref{lemm: linear: 1},
\begin{align*}
	\bE\Bigl[\int^\infty_0e^{2\eta t}\bigl|\bE_t\bigl[\Phi(t,\infty)\psi(t)\bigr]\bigr|^2\rd t\Bigr]^{1/2}&\leq\bE\Bigl[\int^\infty_0e^{2\eta t}\bE_t\bigl[|\Phi(t,\infty)|^2_\mathrm{op}\bigr]|\psi(t)|^2\rd t\Bigr]^{1/2}\\
	&\leq \frac{1}{1-[K_B]_2(\lambda)}\bE\Bigl[\int^\infty_0e^{2\eta t}|\psi(t)|^2\rd t\Bigr]^{1/2}<\infty.
\end{align*}
This implies that the process $(\bE_t[\Phi(t,\infty)\psi(t)])_{t\geq0}$ is in $L^{2,\eta}_\bF(0,\infty;\bR^m)$, and thus the process
\begin{equation*}
	\Bigl(\bE_t\Bigl[\int^\infty_te^{-\lambda(s-t)}R_N(t,s)\bE_s\bigl[\Phi(s,\infty)\psi(s)\bigr]\rd s\Bigr]\Bigr)_{t\geq0}
\end{equation*}
converges to
\begin{equation*}
	\Bigl(\bE_t\Bigl[\int^\infty_te^{-\lambda(s-t)}R(t,s)\bE_s\bigl[\Phi(s,\infty)\psi(s)\bigr]\rd s\Bigr]\Bigr)_{t\geq0}
\end{equation*}
in $L^{2,\eta}_\bF(0,\infty;\bR^m)$ as $N\to\infty$. Consequently, it holds that
\begin{align*}
	Y(t)&=\bE_t\Bigl[\Phi(t,\infty)\psi(t)+\int^\infty_te^{-\lambda(s-t)}R(t,s)\bE_s\bigl[\Phi(s,\infty)\psi(s)\bigr]\rd s\Bigr]\\
	&=\bE_t\Bigl[\Phi(t,\infty)\psi(t)+\int^\infty_te^{-\lambda(s-t)}R(t,s)\Phi(s,\infty)\psi(s)\rd s\Bigr]
\end{align*}
for a.e.\ $t\geq0$, a.s. This completes the proof.
\end{proof}


\begin{rem}
The process $R(\cdot,\cdot)$ can be seen as a \emph{resolvent} of the kernel $\Xi_1(\cdot,\cdot)=\Phi(\cdot,\cdot)A(\cdot,\cdot)$. Indeed, the following holds:
\begin{equation*}
	R(t,s)=\Xi_1(t,s)+\int^s_tR(t,r)\Xi_1(r,s)\rd r
\end{equation*}
for a.e.\ $(t,s)\in\Delta[0,\infty)$, a.s. Theorem~\ref{theo: explicit} is an extension of the results of Hu and {\O}ksendal~\cite{HuOk19} and  Wang, Yong, and Zhang~\cite{WaYoZh20} to the infinite horizon and unbounded coefficients setting.
\end{rem}


\subsection{A duality principle}


In this subsection, we establish a \emph{duality principle} between a linear SVIE
\begin{equation}\label{duality SVIE}
	X(t)=\varphi(t)+\int^t_0C(t,s)X(s)\rd s+\sum^d_{k=1}\int^t_0D_k(t,s)X(s)\rd W_k(s),\ t\geq0,
\end{equation}
and an infinite horizon linear Type-II BSVIE 
\begin{equation}\label{duality BSVIE}
	Y(t)=\psi(t)+\int^\infty_te^{-\lambda(s-t)}\Bigl\{C(s,t)^\top Y(s)+\sum^d_{k=1}D_k(s,t)^\top Z^k(s,t)\Bigr\}\rd s-\int^\infty_tZ(t,s)\rd W(s),\ t\geq0,
\end{equation}
in a weighted $L^2$-space. We impose the following assumptions on the coefficients $C,D_k$, $k=1,\dots,d$:

\begin{assum}\label{assum: duality}
$C,D_k:\Omega\times\Delta^\comp[0,\infty)\to\bR^{n\times n}$, $k=1,\dots,d$, are measurable; $C(t,\cdot)$ and $D_k(t,\cdot)$, $k=1,\dots,d$, are adapted for each $t\geq0$; there exist $K_C\in L^{1,*}(0,\infty;\bR_+)$ and $K_D\in L^{2,*}(0,\infty;\bR_+)$ such that
\begin{equation*}
	|C(t,s)|\leq K_C(t-s)\ \text{and}\ \Bigl(\sum^d_{k=1}|D_k(t,s)|^2\Bigr)^{1/2}\leq K_D(t-s)
\end{equation*}
for a.e.\ $(t,s)\in\Delta^\comp[0,\infty)$, a.s.
\end{assum}

As before, we define
\begin{equation}\label{domain: duality}
	\rho_{C,D}:=\inf\{\rho\in\bR\,|\,[K_C]_1(\rho)+[K_D]_2(\rho)\leq1\}.
\end{equation}


\begin{rem}\label{rem: duality}
\begin{itemize}
\item[(i)]
By Proposition~\ref{prop: SVIE}, for any $\varphi(\cdot)\in L^{2,-\mu}_\bF(0,\infty;\bR^n)$ with $\mu>\rho_{C,D}$, the (forward) SVIE~\eqref{duality SVIE} admits a unique solution $X(\cdot)\in L^{2,-\mu}_\bF(0,\infty;\bR^n)$. By the uniqueness of the solution, we see that the map
\begin{equation*}
	L^{2,-\mu}_\bF(0,\infty;\bR^n)\ni\varphi(\cdot)\mapsto X(\cdot)\in L^{2,-\mu}_\bF(0,\infty;\bR^n)
\end{equation*}
is linear.
\item[(ii)]
By Theorem~\ref{theo: well-posedness BSVIE}, for any $\psi(\cdot)\in L^{2,\eta}_{\cF_\infty}(0,\infty;\bR^n)$ with $\eta+\lambda>\rho_{C,D}$, the infinite horizon Type-II BSVIE~\eqref{duality BSVIE} admits a unique adapted M-solution $(Y(\cdot),Z(\cdot,\cdot))\in\cM^{2,\eta}_\bF(0,\infty;\bR^n\times\bR^{n\times d})$. By the uniqueness of the solution, we see that the map
\begin{equation*}
	L^{2,\eta}_{\cF_\infty}(0,\infty;\bR^n)\ni\psi(\cdot)\mapsto(Y(\cdot),Z(\cdot,\cdot))\in\cM^{2,\eta}_\bF(0,\infty;\bR^n\times\bR^{n\times d})
\end{equation*}
is linear.
\item[(iii)]
For each $\mu,\eta,\lambda\in\bR$ with $\eta+\lambda\geq\mu$ and given processes $\theta_\mu(\cdot)\in L^{2,-\mu}_{\cF_\infty}(0,\infty;\bR^n)$ and $\theta_\eta(\cdot)\in L^{2,\eta}_{\cF_\infty}(0,\infty;\bR^n)$, it holds that
\begin{align*}
	\bE\Bigl[\int^\infty_0e^{-\lambda t}|\theta_\mu(t)|\,|\theta_\eta(t)|\rd t\Bigr]&\leq\bE\Bigl[\int^\infty_0e^{-\mu t}|\theta_\mu(t)|e^{\eta t}|\theta_\eta(t)|\rd t\Bigr]\\
	&\leq\bE\Bigl[\int^\infty_0e^{-2\mu t}|\theta_\mu(t)|^2\rd t\Bigr]^{1/2}\bE\Bigl[\int^\infty_0e^{2\eta t}|\theta_\eta(t)|^2\rd t\Bigr]^{1/2}<\infty.
\end{align*}
This implies that the process $(\langle \theta_\mu(t),\theta_\eta(t)\rangle)_{t\geq0}$ is in $L^{1,-\lambda}_{\cF_\infty}(0,\infty;\bR)$, where $\langle\cdot,\cdot\rangle$ denotes the usual inner product in $\bR^n$.
\end{itemize}
\end{rem}

Now we show the duality principle. This is an extension of Theorem~5.1 in \cite{Yo08} to the infinite horizon (and unbounded coefficients) case.


\begin{theo}\label{theo: duality}
Suppose that Assumption~\ref{assum: duality} holds, and let $\mu,\eta,\lambda\in\bR$ satisfy $\eta+\lambda\geq\mu>\rho_{C,D}$. For given $\varphi(\cdot)\in L^{2,-\mu}_\bF(0,\infty;\bR^n)$ and $\psi(\cdot)\in L^{2,\eta}_{\cF_\infty}(0,\infty;\bR^n)$, let $X(\cdot)\in L^{2,-\mu}_\bF(0,\infty;\bR^n)$ be the solution of SVIE~\eqref{duality SVIE}, and let $(Y(\cdot),Z(\cdot,\cdot))\in\cM^{2,\eta}_\bF(0,\infty;\bR^n\times\bR^{n\times d})$ be the adapted M-solution of infinite horizon BSVIE~\eqref{duality BSVIE}. Then it holds that
\begin{equation}\label{theo: duality: 1}
	\bE\Bigl[\int^\infty_0e^{-\lambda t}\langle\psi(t),X(t)\rangle\rd t\Bigr]=\bE\Bigl[\int^\infty_0e^{-\lambda t}\langle Y(t),\varphi(t)\rangle\rd t\Bigr].
\end{equation}
\end{theo}


\begin{proof}
By the (forward) SVIE~\eqref{duality SVIE}, the definition of adapted M-solutions, and Fubini's theorem, we have
\begin{align*}
	&\bE\Bigl[\int^\infty_0e^{-\lambda t}\langle Y(t),\varphi(t)\rangle\rd t\Bigr]\\
	&=\bE\Bigl[\int^\infty_0e^{-\lambda t}\langle Y(t),X(t)\rangle\rd t-\int^\infty_0e^{-\lambda t}\Bigl\langle Y(t),\int^t_0C(t,s)X(s)\rd s\Bigr\rangle\rd t\\
	&\hspace{2cm}-\int^\infty_0e^{-\lambda t}\Bigl\langle\bE\bigl[Y(t)\bigr]+\sum^d_{k=1}\int^t_0Z^k(t,s)\rd W_k(s),\sum^d_{k=1}\int^t_0D_k(t,s)X(s)\rd W_k(s)\Bigr\rangle\rd t\Bigr]\\
	&=\bE\Bigl[\int^\infty_0e^{-\lambda t}\langle Y(t),X(t)\rangle\rd t-\int^\infty_0e^{-\lambda t}\int^t_0\bigl\langle C(t,s)^\top Y(t), X(s)\bigr\rangle\rd s\rd t\\
	&\hspace{2cm}-\int^\infty_0e^{-\lambda t}\int^t_0\Bigl\langle\sum^d_{k=1}D_k(t,s)^\top Z^k(t,s),X(s)\Bigr\rangle\rd s\rd t\Bigr]\\
	&=\bE\Bigl[\int^\infty_0e^{-\lambda t}\langle Y(t),X(t)\rangle\rd t-\int^\infty_0\Bigl\langle\int^\infty_te^{-\lambda s}C(s,t)^\top Y(s)\rd s,X(t)\Bigr\rangle\rd t\\
	&\hspace{2cm}-\int^\infty_0\Bigl\langle\int^\infty_te^{-\lambda s}\sum^d_{k=1}D_k(s,t)^\top Z^k(s,t)\rd s,X(t)\Bigr\rangle\rd t\Bigr]\\
	&=\bE\Bigl[\int^\infty_0e^{-\lambda t}\Bigl\langle Y(t)-\int^\infty_te^{-\lambda(s-t)}\Bigl\{C(s,t)^\top Y(s)+\sum^d_{k=1}D_k(s,t)^\top Z^k(s,t)\Bigr\}\rd s,X(t)\Bigr\rangle\rd t\Bigr].
\end{align*}
On the other hand, by the infinite horizon BSVIE~\eqref{duality BSVIE},
\begin{equation*}
	Y(t)=\bE_t\Bigl[\psi(t)+\int^\infty_te^{-\lambda(s-t)}\Bigl\{C(s,t)^\top Y(s)+\sum^d_{k=1}D_k(s,t)^\top Z^k(s,t)\Bigr\}\rd s\Bigr]
\end{equation*}
for a.e.\ $t\geq0$, a.s. Noting that $X(\cdot)$ is adapted, we see that
\begin{align*}
	&\bE\Bigl[\int^\infty_0e^{-\lambda t}\Bigl\langle Y(t)-\int^\infty_te^{-\lambda(s-t)}\Bigl\{C(s,t)^\top Y(s)+\sum^d_{k=1}D_k(s,t)^\top Z^k(s,t)\Bigr\}\rd s,X(t)\Bigr\rangle\rd t\Bigr]\\
	&=\bE\Bigl[\int^\infty_0e^{-\lambda t}\langle\psi(t),X(t)\rangle\rd t\Bigr].
\end{align*}
Thus, the equality \eqref{theo: duality: 1} holds.
\end{proof}

In the end of this section, we show the following abstract lemma, which is useful to study some relationships between adapted M-solutions of infinite horizon BSVIEs and various types of infinite horizon BSDEs.


\begin{lemm}\label{lemm: BSVIE BSDE}
Suppose that $\mu,\lambda\in\bR$ satisfy $\lambda\geq2\mu>0$, and let $(Y(\cdot),Z(\cdot,\cdot))\in\cM^{2,-\mu}_\bF(0,\infty;\bR^m\times\bR^{m\times d})$ be fixed. Define a pair of processes $(\cY(\cdot),\cZ(\cdot))$ by
\begin{equation}\label{lemm: BSVIE BSDE: 1}
	\cY(t):=\bE_t\Bigl[\int^\infty_te^{-\lambda(s-t)}Y(s)\rd s\Bigr]\ \text{and}\ \cZ(t):=\int^\infty_te^{-\lambda(s-t)}Z(s,t)\rd s
\end{equation}
for $t\geq0$. Then $(\cY(\cdot),\cZ(\cdot))$ is in $L^{2,-\mu}_\bF(0,\infty;\bR^m)\times L^{2,-\mu}_\bF(0,\infty;\bR^{m\times d})$, $\cY(\cdot)$ is continuous a.s., and it holds that
\begin{equation}\label{lemm: BSVIE BSDE: 2}
	\cY(t)=\cY(T)+\int^T_t\{Y(s)-\lambda\cY(s)\}\rd s-\int^T_t\cZ(s)\rd W(s)
\end{equation}
for any $(t,T)\in\Delta[0,\infty)$, a.s. Conversely, if a pair $(\cY'(\cdot),\cZ'(\cdot))\in L^{2,-\mu}_\bF(0,\infty;\bR^m)\times L^{2,-\mu}_\bF(0,\infty;\bR^{m\times d})$ with $\cY'(\cdot)$ being continuous a.s. satisfies \eqref{lemm: BSVIE BSDE: 2} for any $(t,T)\in\Delta[0,\infty)$, a.s., then it holds that $\cY'(t)=\cY(t)$ for any $t\geq0$, and $\cZ'(t)=\cZ(t)$ for a.e.\ $t\geq0$, a.s.
\end{lemm}


\begin{proof}
Observe that, by Young's convolution inequality,
\begin{align*}
	\bE\Bigl[\int^\infty_0e^{-2\mu t}\Bigl(\int^\infty_te^{-\lambda(s-t)}|Y(s)|\rd s\Bigr)^2\rd t\Bigr]&\leq\bE\Bigl[\int^\infty_0e^{-2\mu t}\Bigl(\int^\infty_te^{-2\mu(s-t)}|Y(s)|\rd s\Bigr)^2\rd t\Bigr]\\
	&=\bE\Bigl[\int^\infty_0\Bigl(\int^\infty_te^{-\mu(s-t)}e^{-\mu s}|Y(s)|\rd s\Bigr)^2\rd t\Bigr]\\
	&\leq\Bigl(\int^\infty_0e^{-\mu t}\rd t\Bigr)^2\bE\Bigl[\int^\infty_0e^{-2\mu t}|Y(t)|^2\rd t\Bigr]<\infty.
\end{align*}
This shows that the process $\cY(\cdot)$ defined by \eqref{lemm: BSVIE BSDE: 1} is well-defined and in $L^{2,-\mu}_\bF(0,\infty;\bR^m)$. Clearly, $\cY(\cdot)$ has continuous paths a.s. Also, by H\"{o}lder's inequality, we have
\begin{align*}
	\bE\Bigl[\int^\infty_0e^{-2\mu t}\Bigl(\int^\infty_te^{-\lambda(s-t)}|Z(s,t)|\rd s\Bigr)^2\rd t\Bigr]&\leq\bE\Bigl[\int^\infty_0e^{-2\mu t}\Bigl(\int^\infty_te^{-2\mu(s-t)}|Z(s,t)|\rd s\Bigr)^2\rd t\Bigr]\\
	&\leq\bE\Bigl[\int^\infty_0e^{-2\mu t}\int^\infty_te^{-2\mu(s-t)}\rd s\int^\infty_te^{-2\mu(s-t)}|Z(s,t)|^2\rd s\rd t\Bigr]\\
	&=\int^\infty_0e^{-2\mu t}\rd t\,\bE\Bigl[\int^\infty_0e^{-2\mu t}\int^t_0|Z(t,s)|^2\rd s\rd t\Bigr]<\infty.
\end{align*}
This shows that $\cZ(\cdot)$ defined by \eqref{lemm: BSVIE BSDE: 1} is well-defined and in $L^{2,-\mu}_\bF(0,\infty;\bR^{m\times d})$.

We show that $(\cY(\cdot),\cZ(\cdot))$ defined by \eqref{lemm: BSVIE BSDE: 1} satisfies \eqref{lemm: BSVIE BSDE: 2} for any $(t,T)\in\Delta[0,\infty)$. We note that, by the definition of the space $\cM^{2,-\mu}_\bF(0,\infty;\bR^m\times\bR^{m\times d})$, it holds that $Y(t)=\bE[Y(t)]+\int^t_0Z(t,s)\rd W(s)$ for a.e.\ $t\geq0$, a.s. For each $t\geq0$, define
\begin{equation*}
	\tilde{\cY}(t):=e^{-\lambda t}\cY(t)=\bE_t\Bigl[\int^\infty_te^{-\lambda s}Y(s)\rd s\Bigr]
\end{equation*}
and
\begin{equation*}
	\tilde{\cZ}(t):=e^{-\lambda t}\cZ(t)=\int^\infty_te^{-\lambda s}Z(s,t)\rd s.
\end{equation*}
For any $(t,T)\in\Delta[0,\infty)$, by using stochastic Fubini's theorem  (cf.\ Theorem 4.A in \cite{BeMi80}), we have
\begin{align*}
	\tilde{\cY}(t)&=\int^\infty_Te^{-\lambda s}\bE_t\bigl[Y(s)\bigr]\rd s+\int^T_te^{-\lambda s}\bE_t\bigl[Y(s)\bigr]\rd s\\
	&=\int^\infty_Te^{-\lambda s}\Bigl\{\bE_T\bigl[Y(s)\bigr]-\int^T_tZ(s,r)\rd W(r)\Bigr\}\rd s+\int^T_te^{-\lambda s}\Bigl\{Y(s)-\int^s_tZ(s,r)\rd W(r)\Bigr\}\rd s\\
	&=\bE_T\Bigl[\int^\infty_Te^{-\lambda s}Y(s)\rd s\Bigr]+\int^T_te^{-\lambda s}Y(s)\rd s-\int^T_t\int^\infty_re^{-\lambda s}Z(s,r)\rd s\rd W(r)\\
	&=\tilde{\cY}(T)+\int^T_te^{-\lambda s}Y(s)\rd s-\int^T_t\tilde{\cZ}(s)\rd W(s).
\end{align*}
Thus, by using It\^{o}'s formula, we see that
\begin{align*}
	\cY(t)&=e^{\lambda t}\tilde{\cY}(t)=e^{\lambda T}\tilde{\cY}(T)+\int^T_t\{Y(s)-\lambda e^{\lambda s}\tilde{\cY}(s)\}\rd s-\int^T_te^{\lambda s}\tilde{\cZ}(s)\rd W(s)\\
	&=\cY(T)+\int^T_t\{Y(s)-\lambda\cY(s)\}\rd s-\int^T_t\cZ(s)\rd W(s),
\end{align*}
and hence \eqref{lemm: BSVIE BSDE: 2} holds.

Next, suppose that $(\cY'(\cdot),\cZ'(\cdot))\in L^{2,-\mu}_\bF(0,\infty;\bR^m)\times L^{2,-\mu}_\bF(0,\infty;\bR^{m\times d})$ with $\cY'(\cdot)$ being continuous also satisfies \eqref{lemm: BSVIE BSDE: 2} for any $(t,T)\in\Delta[0,\infty)$. Then It\^{o}'s formula yields that
\begin{equation*}
	e^{-\lambda t}\{\cY'(t)-\cY(t)\}=e^{-\lambda T}\{\cY'(T)-\cY(T)\}-\int^T_te^{-\lambda s}\{\cZ'(s)-\cZ(s)\}\rd W(s),
\end{equation*}
and hence $e^{-\lambda t}\{\cY'(t)-\cY(t)\}=\bE_t\bigl[e^{-\lambda T}\{\cY'(T)-\cY(T)\}\bigr]$ a.s.\ for any $(t,T)\in\Delta[0,\infty)$. Since $-\lambda<-\mu$, we have $L^{2,-\mu}_\bF(0,\infty;\bR^m)\subset L^{1,-\lambda}_\bF(0,\infty;\bR^m)$, and hence the process $(e^{-\lambda t}\{\cY'(t)-\cY(t)\})_{t\geq0}$ is in $L^1_\bF(0,\infty;\bR^m)$. This implies that there exists a sequence $\{T_n\}_{n\in\bN}\in(0,\infty)$ such that $\lim_{n\to\infty}T_n=\infty$ and $\lim_{n\to\infty}e^{-\lambda T_n}\{\cY'(T_n)-\cY(T_n)\}=0$ in $L^1_{\cF_\infty}(\Omega;\bR^m)$. Therefore, we get $\cY'(t)=\cY(t)$ a.s.\ for any $t\geq0$. This implies that $\cZ'(t)=\cZ(t)$ for a.e.\ $t\geq0$, a.s. We complete the proof.
\end{proof}


\begin{rem}
Equation~\eqref{lemm: BSVIE BSDE: 2} can be seen as an infinite horizon BSDE, which is written in the differential form:
\begin{equation}\label{infinite horizon BSDE given Y}
	\mathrm{d}\cY(t)=-\{Y(t)-\lambda\cY(t)\}\rd t+\cZ(t)\rd W(t),\ t\geq0.
\end{equation}
For a given constant $\lambda\in\bR$ and a given $\bR^m$-valued adapted process $Y(\cdot)$ such that $\bE[\int^T_0|Y(t)|^2\rd t]<\infty$ for any $T>0$, we say that a pair $(\cY(\cdot),\cZ(\cdot))$ of an $\bR^m$-valued continuous and adapted process $\cY(\cdot)$ and an $\bR^{m\times d}$-valued adapted process $\cZ(\cdot)$ is an adapted solution to the infinite horizon BSDE~\eqref{infinite horizon BSDE given Y} if $\bE[\int^T_0\{|\cY(t)|^2+|\cZ(t)|^2\}\rd t]<\infty$ for any $T>0$, and if \eqref{lemm: BSVIE BSDE: 2} holds for any $(t,T)\in\Delta[0,\infty)$, a.s. For detailed discussions on infinite horizon BSDEs, see, for example, \cite{FuTe04} and references cited therein. Lemma~\ref{lemm: BSVIE BSDE} shows that, for each given $(Y(\cdot),Z(\cdot,\cdot))\in\cM^{2,-\mu}_\bF(0,\infty;\bR^m\times\bR^{m\times d})$ with $\lambda\geq2\mu>0$, there exists a unique adapted solution $(\cY(\cdot),\cZ(\cdot))$ to the infinite horizon BSDE~\eqref{infinite horizon BSDE given Y} in $L^{2,-\mu}_\bF(0,\infty;\bR^m)\times L^{2,-\mu}_\bF(0,\infty;\bR^{m\times d})$ given by the explicit formula \eqref{lemm: BSVIE BSDE: 1}. Some interesting applications of this result are discussed in Section~\ref{subsection: example}.
\end{rem}


\section{Discounted control problems for SVIEs with singular coefficients}\label{section: control}


As an application of the results on infinite horizon BSVIEs, we consider an infinite horizon stochastic control problem where the dynamics of the state process is described as a (forward) SVIE with singular coefficients. The duality principle which we proved in the previous section plays a crucial role.

For each $\mu\in\bR$, define the set of \emph{control processes} by $\cU_{-\mu}:=L^{2,-\mu}_\bF(0,\infty;U)$, where $U$ is a convex body (i.e., $U$ is convex and has a nonempty interior) in $\bR^\ell$ with $\ell\in\bN$. For each control process $u(\cdot)\in\cU_{-\mu}$, define the corresponding \emph{state process} $X^u(\cdot)$ as the solution of the following controlled SVIE:
\begin{equation}\label{controlled SVIE}
	X^u(t)=\varphi(t)+\int^t_0b(t,s,X^u(s),u(s))\rd s+\int^t_0\sigma(t,s,X^u(s),u(s))\rd W(s),\ t\geq0,
\end{equation}
where $\varphi(\cdot)$ is a given process, and $b,\sigma$ are given deterministic maps. In order to measure the performance of $u(\cdot)$ and $X^u(\cdot)$, we consider the following \emph{discounted cost functional}:
\begin{equation}\label{cost functional}
	J_\lambda(u(\cdot)):=\bE\Bigl[\int^\infty_0e^{-\lambda t}h(t,X^u(t),u(t))\rd t\Bigr],
\end{equation}
where $h$ is a given $\bR$-valued deterministic function and $\lambda\in\bR$ is a \emph{discount rate}. The stochastic control problem is a problem to seek a control process $\hat{u}(\cdot)\in\cU_{-\mu}$ such that
\begin{equation*}
	J_\lambda(\hat{u}(\cdot))=\inf_{u(\cdot)\in\cU_{-\mu}}J_\lambda(u(\cdot)).
\end{equation*}
If it is the case, we call $\hat{u}(\cdot)$ an \emph{optimal control}.


\begin{assum}\label{assum: control}
\begin{itemize}
\item[(i)]
$\varphi(\cdot)\in L^{2,-\mu}_\bF(0,\infty;\bR^n)$.
\item[(ii)]
$b:\Delta^\comp[0,\infty)\times\bR^n\times\bR^\ell\to\bR^n$ and $\sigma:\Delta^\comp[0,\infty)\times\bR^n\times\bR^\ell\to\bR^{n\times d}$ are measurable; $b(t,s,\cdot,\cdot)$ and $\sigma(t,s,\cdot,\cdot)$ are continuously differentiable in $(x,u)\in\bR^n\times\bR^\ell$ for a.e.\ $(t,s)\in\Delta^\comp[0,\infty)$; there exist $K_{b,x},K_{b,u}\in L^{1,*}(0,\infty;\bR_+)$ and $K_{\sigma,x},K_{\sigma,u}\in L^{2,*}(0,\infty;\bR_+)$ such that
\begin{align*}
	&|b(t,s,x,u)-b(t,s,x',u')|\leq K_{b,x}(t-s)|x-x'|+K_{b,u}(t-s)|u-u'|,\\
	&|\sigma(t,s,x,u)-\sigma(t,s,x',u')|\leq K_{\sigma,x}(t-s)|x-x'|+K_{\sigma,u}(t-s)|u-u'|,
\end{align*}
for any $x,x'\in\bR^n$ and $u,u'\in\bR^\ell$, for a.e.\ $(t,s)\in\Delta^\comp[0,\infty)$; $b(t,s,0,0)=0$ and $\sigma(t,s,0,0)=0$ for a.e.\ $(t,s)\in\Delta^\comp[0,\infty)$.
\item[(iii)]
$h:[0,\infty)\times\bR^n\times\bR^\ell\to\bR$ is measurable; $h(t,\cdot,\cdot)$ is continuously differentiable in $(x,u)\in\bR^n\times\bR^\ell$ for a.e.\ $t\geq0$; there exists a constant $C>0$ such that
\begin{equation*}
	|h(t,x,u)|\leq C(1+|x|^2+|u|^2),\ |\partial_xh(t,x,u)|\leq C(1+|x|+|u|),\ |\partial_uh(t,x,u)|\leq C(1+|x|+|u|),
\end{equation*}
for any $(x,u)\in\bR^n\times\bR^\ell$, for a.e.\ $t\geq0$.
\end{itemize}
\end{assum}

Now we discuss the choice of the parameters $\mu$ and $\lambda$. First, in order to ensure a nonzero constant control $u(\cdot)\equiv u\in U$ and a nonzero initial state $\varphi(\cdot)\equiv x_0\in\bR^n$, the weight parameter $\mu$ has to be strictly positive. Second, by Lemma~\ref{lemm: coefficients SVIE}, if $[K_{b,u}]_1(\mu)+[K_{\sigma,u}]_2(\mu)<\infty$, then for any $u(\cdot)\in\cU_{-\mu}$,
\begin{equation*}
	\bE\Bigl[\int^\infty_0e^{-2\mu t}\Bigl(\int^t_0|b(t,s,0,u(s))|\rd s\Bigr)^2\rd t\Bigr]^{1/2}\leq[K_{b,u}]_1(\mu)\bE\Bigl[\int^\infty_0e^{-2\mu t}|u(t)|^2\rd t\Bigr]^{1/2}<\infty
\end{equation*}
and
\begin{equation*}
	\bE\Bigl[\int^\infty_0e^{-2\mu t}\int^t_0\bigl|\sigma(t,s,0,u(s))\bigr|^2\rd s\rd t\Bigr]^{1/2}\leq[K_{\sigma,u}]_2(\mu)\bE\Bigl[\int^\infty_0e^{-2\mu t}|u(t)|^2\rd t\Bigr]^{1/2}<\infty.
\end{equation*}
Third, if in addition $[K_{b,x}]_1(\mu)+[K_{\sigma,x}]_2(\mu)<1$, by Proposition~\ref{prop: SVIE} and Remark~\ref{rem: SVIE nonzero}, the controlled SVIE~\eqref{controlled SVIE} admits a unique solution $X^u(\cdot)\in L^{2,-\mu}_\bF(0,\infty;\bR^n)$ for any $u(\cdot)\in\cU_{-\mu}$. From these observations, we define
\begin{equation}\label{domain: control}
	\rho_{b,\sigma;x,u}:=\inf\{\rho\in\bR_+\,|\,[K_{b,u}]_1(\rho)+[K_{\sigma,u}]_2(\rho)<\infty,\ [K_{b,x}]_1(\rho)+[K_{\sigma,x}]_2(\rho)\leq1\}.
\end{equation}
If $\mu>\rho_{b,\sigma;x,u}$ and $\lambda\geq2\mu$, the state process $X^u(\cdot)\in L^{2,-\mu}_\bF(0,\infty;\bR^n)$ and the discounted cost functional $J_\lambda(u(\cdot))\in\bR$ is well-defined for any $u(\cdot)\in\cU_{-\mu}$. We note that the parameters $\mu$ and $\lambda$ satisfying these conditions are strictly positive.


\subsection{Pontryagin's maximum principle}


In this subsection, we provide a necessary condition for optimality by means of \emph{Pontryagin's maximum principle}. Let Assumption~\ref{assum: control} hold, and suppose that $\mu,\lambda\in\bR$ satisfy the conditions
\begin{equation}\label{mu lambda}
	\mu>\rho_{b,\sigma;x,u}\ \text{and}\ \lambda\geq2\mu.
\end{equation}

Fix a control process $\hat{u}(\cdot)\in\cU_{-\mu}$, and denote the corresponding state process by $\hat{X}(\cdot):=X^{\hat{u}}(\cdot)$. We use the following notations:
\begin{align*}
	&b_x(t,s):=\partial_xb(t,s,\hat{X}(s),\hat{u}(s)),\ b_u(t,s):=\partial_ub(t,s,\hat{X}(s),\hat{u}(s)),\\
	&\sigma^k_x(t,s):=\partial_x\sigma^k(t,s,\hat{X}(s),\hat{u}(s)),\ \sigma^k_u(t,s):=\partial_u\sigma^k(t,s,\hat{X}(s),\hat{u}(s)),\ k=1,\dots,d,
\end{align*}
for $(t,s)\in\Delta^\comp[0,\infty)$, and
\begin{equation*}
	h_x(t):=\partial_xh(t,\hat{X}(t),\hat{u}(t)),\ h_u(t):=\partial_uh(t,\hat{X}(t),\hat{u}(t)),
\end{equation*}
for $t\geq0$.


\begin{rem}
Here and elsewhere, for each scalar-valued differentiable function $f$ on $\bR^{d_1}$ with $d_1\in\bN$, the derivative $\partial_xf(x')\in\bR^{d_1}$ of $f$ at $x'\in\bR^{d_1}$ is understood as a column vector. If $f$ is $\bR^{d_2}$-valued with $d_2\in\bN$, the derivative $\partial_xf(x')$ is understood as a $(d_2\times d_1)$-matrix. 
\end{rem}

Note that $b_x(\cdot,\cdot),\sigma^k_x(\cdot,\cdot)$, $k=1,\dots,d$, are $\bR^{n\times n}$-valued and measurable, $b_u(\cdot,\cdot),\sigma^k_u(\cdot,\cdot)$, $k=1,\dots,d$, are $\bR^{n\times\ell}$-valued and measurable, $b_x(t,\cdot),b_u(t,\cdot),\sigma^k_x(t,\cdot),\sigma^k_u(t,\cdot)$, $k=1,\dots,d$, are adapted for each $t\geq0$, and
\begin{align*}
	&|b_x(t,s)|\leq K_{b,x}(t-s),\ |b_u(t,s)|\leq K_{b,u}(t-s),\\
	&\Bigl(\sum^d_{k=1}|\sigma^k_x(t,s)|^2\Bigr)^{1/2}\leq K_{\sigma,x}(t-s),\ \Bigl(\sum^d_{k=1}|\sigma^k_u(t,s)|^2\Bigr)^{1/2}\leq K_{\sigma,u}(t-s),
\end{align*}
for a.e.\ $(t,s)\in\Delta^\comp[0,\infty)$, a.s. Furthermore,
\begin{equation*}
	h_x(\cdot)\in L^{2,-\mu}_\bF(0,\infty;\bR^n),\ h_u(\cdot)\in L^{2,-\mu}_\bF(0,\infty;\bR^\ell).
\end{equation*}
We first show the following proposition.


\begin{prop}\label{prop: variational inequality}
Let Assumption~\ref{assum: control} hold, and suppose that $\mu,\lambda\in\bR$ satisfy the condition~\eqref{mu lambda}. If $\hat{u}(\cdot)$ is an optimal control, then for any $u(\cdot)\in\cU_{-\mu}$, the following \emph{variational inequality} holds:
\begin{equation}\label{prop: variational inequality: 1}
	\bE\Bigl[\int^\infty_0e^{-\lambda t}\bigl\{\langle h_x(t),X_1(t)\rangle+\langle h_u(t),u(t)-\hat{u}(t)\rangle\bigr\}\rd t\Bigr]\geq0,
\end{equation}
where $X_1(\cdot)\in L^{2,-\mu}_\bF(0,\infty;\bR^n)$ is the solution of the following \emph{variational SVIE}:
\begin{equation}\label{prop: variational inequality: 2}
\begin{split}
	&X_1(t)=\int^t_0b_u(t,s)(u(s)-\hat{u}(s))\rd s+\sum^d_{k=1}\int^t_0\sigma^k_u(t,s)(u(s)-\hat{u}(s))\rd W_k(s)\\
	&\hspace{2cm}+\int^t_0b_x(t,s)X_1(s)\rd s+\sum^d_{k=1}\int^t_0\sigma^k_x(t,s)X_1(s)\rd W_k(s),\ t\geq0.
\end{split}
\end{equation}
\end{prop}


\begin{proof}
By Proposition~\ref{prop: SVIE}, for any $u(\cdot)\in\cU_{-\mu}$, the variational SVIE~\eqref{prop: variational inequality: 2} admits a unique solution $X_1(\cdot)\in L^{2,-\mu}_\bF(0,\infty;\bR^n)$, and thus the expectation in the left-hand side of \eqref{prop: variational inequality: 1} is well-defined.

Let $u(\cdot)\in\cU_{-\mu}$ be fixed. For each $\gamma\in(0,1]$, define
\begin{equation*}
	u^\gamma(\cdot):=\gamma u(\cdot)+(1-\gamma)\hat{u}(\cdot)=\hat{u}(\cdot)+\gamma(u(\cdot)-\hat{u}(\cdot)).
\end{equation*}
Since $U\subset\bR^\ell$ is convex, the process $u^\gamma(\cdot)$ is in $\cU_{-\mu}$. Denote the corresponding state process by $X^\gamma(\cdot):=X^{u^\gamma}(\cdot)$, and define $\tilde{X}^\gamma_1(\cdot):=\frac{1}{\gamma}\bigl(X^\gamma(\cdot)-\hat{X}(\cdot)\bigr)$. By simple calculations, we see that $\tilde{X}^\gamma_1(\cdot)\in L^{2,-\mu}_\bF(0,\infty;\bR^n)$ is the solution of the following linear SVIE:
\begin{align*}
	&\tilde{X}^\gamma_1(t)=\int^t_0\tilde{b}^\gamma_u(t,s)(u(s)-\hat{u}(s))\rd s+\sum^d_{k=1}\int^t_0\tilde{\sigma}^{\gamma,k}_u(t,s)(u(s)-\hat{u}(s))\rd W_k(s)\\
	&\hspace{2cm}+\int^t_0\tilde{b}^\gamma_x(t,s)\tilde{X}^\gamma_1(s)\rd s+\sum^d_{k=1}\int^t_0\tilde{\sigma}^{\gamma,k}_x(t,s)\tilde{X}^\gamma_1(s)\rd W_k(s),\ t\geq0,
\end{align*}
where
\begin{equation*}
	\tilde{b}^\gamma_x(t,s):=\int^1_0\partial_xb\bigl(t,s,\hat{X}(s)+\theta(X^\gamma(s)-\hat{X}(s)),\hat{u}(s)+\theta(u^\gamma(s)-\hat{u}(s))\bigr)\rd\theta
\end{equation*}
for $(t,s)\in\Delta^\comp[0,\infty)$, and similar for $\tilde{b}^\gamma_u(\cdot,\cdot),\tilde{\sigma}^{\gamma,k}_x(\cdot,\cdot),\tilde{\sigma}^{\gamma,k}_u(\cdot,\cdot)$, $k=1,\dots,d$. By Proposition~\ref{prop: SVIE}, the following estimate holds:
\begin{align*}
	&\bE\Bigl[\int^\infty_0e^{-2\mu t}|\tilde{X}^\gamma_1(t)-X_1(t)|^2\rd t\Bigr]^{1/2}\\
	&\leq C_\mu\bE\Bigl[\int^\infty_0e^{-2\mu t}\Bigl|\int^t_0\bigl(\tilde{b}^\gamma_u(t,s)-b_u(t,s)\bigr)\bigl(u(s)-\hat{u}(s)\bigr)\rd s\\
	&\hspace{2.5cm}+\sum^d_{k=1}\int^t_0\bigl(\tilde{\sigma}^{\gamma,k}_u(t,s)-\sigma^k_u(t,s)\bigr)\bigl(u(s)-\hat{u}(s)\bigr)\rd W_k(s)\\
	&\hspace{2.5cm}+\int^t_0\bigl(\tilde{b}^\gamma_x(t,s)-b_x(t,s)\bigr)X_1(s)\rd s+\sum^d_{k=1}\int^t_0\bigl(\tilde{\sigma}^{\gamma,k}_x(t,s)-\sigma^k_x(t,s)\bigr)X_1(s)\rd W_k(s)\Bigr|^2\rd t\Bigr]^{1/2}\\
	&\leq C_\mu\Bigl\{\bE\Bigl[\int^\infty_0e^{-2\mu t}\Bigl(\int^t_0|\tilde{b}^\gamma_u(t,s)-b_u(t,s)|\,|u(s)-\hat{u}(s)|\rd s\Bigr)^2\rd t\Bigr]^{1/2}\\
	&\hspace{1cm}+\bE\Bigl[\int^\infty_0e^{-2\mu t}\int^t_0\sum^d_{k=1}|\tilde{\sigma}^{\gamma,k}_u(t,s)-\sigma^k_u(t,s)|^2|u(s)-\hat{u}(s)|^2\rd s\rd t\Bigr]^{1/2}\\
	&\hspace{1cm}+\bE\Bigl[\int^\infty_0e^{-2\mu t}\Bigl(\int^t_0|\tilde{b}^\gamma_x(t,s)-b_x(t,s)|\,|X_1(s)|\rd s\Bigr)^2\rd t\Bigr]^{1/2}\\
	&\hspace{1cm}+\bE\Bigl[\int^\infty_0e^{-2\mu t}\int^t_0\sum^d_{k=1}|\tilde{\sigma}^{\gamma,k}_x(t,s)-\sigma^k_x(t,s)|^2|X_1(s)|^2\rd s\rd t\Bigr]^{1/2}\Bigr\},
\end{align*}
where $C_\mu:=1/(1-[K_{b,x}]_1(\mu)-[K_{\sigma,x}]_2(\mu))\in(0,\infty)$. Now we show that
\begin{equation}\label{gamma convergence}
	\lim_{\gamma\downarrow0}\bE\Bigl[\int^\infty_0e^{-2\mu t}\Bigl(\int^t_0|\tilde{b}^\gamma_x(t,s)-b_x(t,s)|\,|X_1(s)|\rd s\Bigr)^2\rd t\Bigr]=0.
\end{equation}
Since $X_1(\cdot)\in L^{2,-\mu}_\bF(0,\infty;\bR^n)$ with $\mu>\rho_{b,\sigma;x,u}$, by using Young's convolution inequality, we see that
\begin{equation*}
	\bE\Bigl[\int^\infty_0e^{-2\mu t}\Bigl(\int^t_0K_{b,x}(t-s)|X_1(s)|\rd s\Bigr)^2\rd t\Bigr]<\infty.
\end{equation*}
Clearly, $u^\gamma(\cdot)$ converges to $\hat{u}(\cdot)$ in $L^{2,-\mu}_\bF(0,\infty;\bR^\ell)$ as $\gamma\downarrow0$. By the stability estimate for SVIEs (see Proposition~\ref{prop: SVIE: 2}), we can easily show that $X^\gamma(\cdot)\to\hat{X}(\cdot)$ in $L^{2,-\mu}_\bF(0,\infty;\bR^n)$ as $\gamma\downarrow0$. Take an arbitrary sequence $\{\gamma_n\}_{n\in\bN}\subset(0,1]$. Then there exists a subsequence $\{\gamma_{n_k}\}_{k\in\bN}$ of $\{\gamma_n\}_{n\in\bN}$ such that $u^{\gamma_{n_k}}(s)\to\hat{u}(s)$ and $X^{\gamma_{n_k}}(s)\to\hat{X}(s)$ as $k\to\infty$ for a.e.\ $s\geq0$, a.s. Since $\partial_xb(t,s,x,u)$ is continuous in $(x,u)\in\bR^n\times\bR^\ell$ and bounded by $K_{b,x}(t-s)$ for a.e.\ $(t,s)\in\Delta^\comp[0,\infty)$, by the dominated convergence theorem, we see that
\begin{equation*}
\lim_{k\to\infty}\bE\Bigl[\int^\infty_0e^{-2\mu t}\Bigl(\int^t_0|\tilde{b}^{\gamma_{n_k}}_x(t,s)-b_x(t,s)|\,|X_1(s)|\rd s\Bigr)^2\rd t\Bigr]=0.
\end{equation*}
Since the sequence $\{\gamma_n\}_{n\in\bN}\subset(0,1]$ is arbitrary, we get \eqref{gamma convergence}. By the same arguments, we can show that
\begin{align*}
	&\lim_{\gamma\downarrow0}\bE\Bigl[\int^\infty_0e^{-2\mu t}\Bigl(\int^t_0|\tilde{b}^\gamma_u(t,s)-b_u(t,s)|\,|u(s)-\hat{u}(s)|\rd s\Bigr)^2\rd t\Bigr]=0,\\
	&\lim_{\gamma\downarrow0}\bE\Bigl[\int^\infty_0e^{-2\mu t}\int^t_0\sum^d_{k=1}|\tilde{\sigma}^{\gamma,k}_x(t,s)-\sigma^k_x(t,s)|^2|X_1(s)|^2\rd s\rd t\Bigr]=0,\\
	&\lim_{\gamma\downarrow0}\bE\Bigl[\int^\infty_0e^{-2\mu t}\int^t_0\sum^d_{k=1}|\tilde{\sigma}^{\gamma,k}_u(t,s)-\sigma^k_u(t,s)|^2|u(s)-\hat{u}(s)|^2\rd s\rd t\Bigr]=0.
\end{align*}
Thus, we see that
\begin{equation}\label{proof: variational inequality: 1}
	\lim_{\gamma\downarrow0}\bE\Bigl[\int^\infty_0e^{-2\mu t}|\tilde{X}^\gamma_1(t)-X_1(t)|^2\rd t\Bigr]=0.
\end{equation}

Since $\hat{u}(\cdot)$ is optimal, the following inequality holds for any $\gamma\in(0,1]$:
\begin{equation}\label{proof: variational inequality: 2}
	0\leq\frac{J_\lambda(u^\gamma(\cdot))-J_\lambda(\hat{u}(\cdot))}{\gamma}=\bE\Bigl[\int^\infty_0e^{-\lambda t}\bigl\{\langle \tilde{h}^\gamma_x(t),\tilde{X}^\gamma_1(t)\rangle+\langle\tilde{h}^\gamma_u(t),u(t)-\hat{u}(t)\rangle\bigr\}\rd t\Bigr],
\end{equation}
where
\begin{equation*}
	\tilde{h}^\gamma_x(t):=\int^1_0\partial_xh\bigl(t,\hat{X}(t)+\theta(X^\gamma(t)-\hat{X}(t)),\hat{u}(t)+\theta(u^\gamma(t)-\hat{u}(t))\bigr)\rd\theta
\end{equation*}
for $t\geq0$, and similar for $\tilde{h}^\gamma_u(\cdot)$. Noting \eqref{proof: variational inequality: 1} and the assumption $\lambda\geq2\mu$, we can easily show that the right-hand side of \eqref{proof: variational inequality: 2} tends to
\begin{equation*}
	\bE\Bigl[\int^\infty_0e^{-\lambda t}\bigl\{\langle h_x(t),X_1(t)\rangle+\langle h_u(t),u(t)-\hat{u}(t)\rangle\bigr\}\rd t\Bigr]
\end{equation*}
as $\gamma\downarrow0$. Therefore, the variational inequality \eqref{prop: variational inequality: 1} holds, and we finish the proof.
\end{proof}

The above proposition provides a necessary condition for $\hat{u}(\cdot)$ to be optimal. However, since the process $X_1(\cdot)$ depends on another control process $u(\cdot)$, the variational inequality \eqref{prop: variational inequality} itself is not useful to determine $\hat{u}(\cdot)$. In order to obtain a more informative necessary condition, we have to get rid of $X_1(\cdot)$. To do so, we introduce the following \emph{adjoint equation} of the form of an infinite horizon Type-II BSVIE:
\begin{equation}\label{adjoint equation}
	\hat{Y}(t)=h_x(t)+\int^\infty_te^{-\lambda(s-t)}\Bigl\{b_x(s,t)^\top\hat{Y}(s)+\sum^d_{k=1}\sigma^k_x(s,t)^\top\hat{Z}^k(s,t)\Bigr\}\rd s-\int^\infty_t\hat{Z}(t,s)\rd W(s),\ t\geq0.
\end{equation}
Since $h_x(\cdot)\in L^{2,-\mu}_\bF(0,\infty;\bR^n)$, $|b_x(s,t)|\leq K_{b,x}(s-t)$ and $\bigl(\sum^d_{k=1}|\sigma^k_x(s,t)|^2\bigr)^{1/2}\leq K_{\sigma,x}(s-t)$ for a.e.\ $(t,s)\in\Delta[0,\infty)$, a.s., and $-\mu+\lambda\geq\mu>\rho_{b,\sigma;x,u}$, by Theorem~\ref{theo: well-posedness BSVIE}, there exists a unique adapted M-solution $(\hat{Y}(\cdot),\hat{Z}(\cdot,\cdot))\in\cM^{2,-\mu}_\bF(0,\infty;\bR^n\times\bR^{n\times d})$ of the adjoint equation \eqref{adjoint equation}. We note that $(\hat{Y}(\cdot),\hat{Z}(\cdot,\cdot))$ depends on the given control-state pair $(\hat{u}(\cdot),\hat{X}(\cdot))$, but it does not depend on another control process $u(\cdot)$. The following theorem is a type of \emph{Pontryagin's maximum principle} which provides a necessary condition for optimality.


\begin{theo}\label{theo: maximum principle}
Let Assumption~\ref{assum: control} hold, and suppose that $\mu,\lambda\in\bR$ satisfy the condition~\eqref{mu lambda}. If $\hat{u}(\cdot)\in\cU_{-\mu}$ is an optimal control, then it holds that
\begin{equation}\label{theo: maximum principle: 1}
	\Bigl\langle h_u(t)+\int^\infty_te^{-\lambda(s-t)}\Bigl\{b_u(s,t)^\top\bE_t\bigl[\hat{Y}(s)]+\sum^d_{k=1}\sigma^k_u(s,t)^\top\hat{Z}^k(s,t)\Bigr\}\rd s,u-\hat{u}(t)\Bigr\rangle\geq0,\ \forall\,u\in U,
\end{equation}
for a.e.\ $t\geq0$, a.s.
\end{theo}


\begin{proof}
Let an arbitrary $u(\cdot)\in\cU_{-\mu}$ be fixed. We apply the duality principle (Theorem~\ref{theo: duality}) to the variational SVIE~\eqref{prop: variational inequality: 2} and the adjoint equation~\eqref{adjoint equation}. By using the definition of adapted M-solutions and Fubini's theorem, we have
\begin{align*}
	&\bE\Bigl[\int^\infty_0e^{-\lambda t}\langle h_x(t),X_1(t)\rangle\rd t\Bigr]\\
	&=\bE\Bigl[\int^\infty_0e^{-\lambda t}\Bigl\langle \hat{Y}(t),\int^t_0b_u(t,s)\bigl(u(s)-\hat{u}(s)\bigr)\rd s+\sum^d_{k=1}\int^t_0\sigma^k_u(t,s)\bigl(u(s)-\hat{u}(s)\bigr)\rd W_k(s)\Bigr\rangle\rd t\Bigr]\\
	&=\bE\Bigl[\int^\infty_0e^{-\lambda t}\int^t_0\bigl\langle b_u(t,s)^\top\hat{Y}(t),u(s)-\hat{u}(s)\bigr\rangle\rd s\rd t\\
	&\hspace{2cm}+\int^\infty_0e^{-\lambda t}\int^t_0\Bigl\langle\sum^d_{k=1}\sigma^k_u(t,s)^\top\hat{Z}^k(t,s),u(s)-\hat{u}(s)\Bigr\rangle\rd s\rd t\Bigr]\\
	&=\bE\Bigl[\int^\infty_0\Bigl\langle\int^\infty_te^{-\lambda s}\Bigl\{b_u(s,t)^\top\hat{Y}(s)+\sum^d_{k=1}\sigma^k_u(s,t)^\top\hat{Z}^k(s,t)\Bigr\}\rd s,u(t)-\hat{u}(t)\Bigr\rangle\rd t\Bigr]\\
	&=\bE\Bigl[\int^\infty_0e^{-\lambda t}\Bigl\langle\int^\infty_te^{-\lambda(s-t)}\Bigl\{b_u(s,t)^\top\bE_t\bigl[\hat{Y}(s)\bigr]+\sum^d_{k=1}\sigma^k_u(s,t)^\top\hat{Z}^k(s,t)\Bigr\}\rd s,u(t)-\hat{u}(t)\Bigr\rangle\rd t\Bigr].
\end{align*}
Thus, by the variational inequality \eqref{prop: variational inequality: 1}, we get
\begin{equation*}
	\bE\Bigl[\int^\infty_0e^{-\lambda t}\Bigl\langle h_u(t)+\int^\infty_te^{-\lambda(s-t)}\Bigl\{b_u(s,t)^\top\bE_t\bigl[\hat{Y}(s)\bigr]+\sum^d_{k=1}\sigma^k_u(s,t)^\top\hat{Z}^k(s,t)\Bigr\}\rd s,u(t)-\hat{u}(t)\Bigr\rangle\rd t\Bigr]\geq0.
\end{equation*}
Since $u(\cdot)\in\cU_{-\mu}$ is arbitrary, we see that \eqref{theo: maximum principle: 1} holds for a.e.\ $t\geq0$, a.s. This completes the proof.
\end{proof}


\begin{rem}
The above result can be rewritten by using a \emph{Hamiltonian} defined by
\begin{equation}\label{Hamiltonian}
	H_\lambda(t,x,u,p(\cdot),q(\cdot)):=h(t,x,u)+\int^\infty_te^{-\lambda(s-t)}\Bigl\{\langle b(s,t,x,u),p(s)\rangle+\sum^d_{k=1}\langle\sigma^k(s,t,x,u),q^k(s)\rangle\Bigr\}\rd s
\end{equation}
for $(t,x,u,p(\cdot),q(\cdot))\in[0,\infty)\times\bR^n\times U\times L^{2,-\mu}(0,\infty;\bR^n)\times L^{2,-\mu}(0,\infty;\bR^{n\times d})$. An interesting feature is that the Hamiltonian $H_\lambda$ is a \emph{functional} of the adjoint components $p(\cdot)\in L^{2,-\mu}(0,\infty;\bR^n)$ and $q(\cdot)\in L^{2,-\mu}(0,\infty;\bR^{n\times d})$ which are (deterministic) functions. The adjoint equation \eqref{adjoint equation} can be written as
\begin{equation}\label{adjoint equation Hamiltonian}
\begin{cases}
	\hat{Y}(t)=\partial_xH_\lambda\bigl(t,\hat{X}(t),\hat{u}(t),\bE_t\bigl[\hat{Y}(\cdot)\bigr],\hat{Z}(\cdot,t)\bigr),\\
	\hat{Y}(t)=\bE[\hat{Y}(t)]+\int^t_0\hat{Z}(t,s)\rd W(s),
\end{cases}
	\text{for a.e.}\ t\geq0,\ \text{a.s.},
\end{equation}
which determines the adapted process $\hat{Y}(\cdot)$ and the values of $\hat{Z}(t,s)$ for $(t,s)\in\Delta^\comp[0,\infty)$. The values of $\hat{Z}(t,s)$ for $(t,s)\in\Delta[0,\infty)$ is determined via the martingale representation theorem:
\begin{equation*}
	\partial_xH_\lambda\bigl(t,\hat{X}(t),\hat{u}(t),\hat{Y}(\cdot),\hat{Z}(\cdot,t)\bigr)=\bE_t\bigl[\partial_xH_\lambda\bigl(t,\hat{X}(t),\hat{u}(t),\hat{Y}(\cdot),\hat{Z}(\cdot,t)\bigr)\bigr]+\int^\infty_t\hat{Z}(t,s)\rd W(s)
\end{equation*}
for each $t\geq0$. Furthermore, the optimality condition \eqref{theo: maximum principle: 1} can be written as
\begin{equation}\label{optimality condition Hamiltonian}
	\bigl\langle\partial_uH_\lambda\bigl(t,\hat{X}(t),\hat{u}(t),\bE_t\bigl[\hat{Y}(\cdot)\bigr],\hat{Z}(\cdot,t)\bigr),u-\hat{u}(t)\bigr\rangle\geq0,\ \forall\,u\in U.
\end{equation}
\end{rem}


\subsection{Sufficient conditions for optimality}


In this subsection, we provide sufficient conditions for optimality in terms of the Hamiltonian $H_\lambda$ defined by \eqref{Hamiltonian}.


\begin{theo}\label{theo: sufficient maximum principle}
Let Assumption~\ref{assum: control} hold, and suppose that $\mu,\lambda\in\bR$ satisfy the condition~\eqref{mu lambda}. Let $\hat{u}(\cdot)\in\cU_{-\mu}$ be a given control process, and denote the state process corresponding to $\hat{u}(\cdot)$ by $\hat{X}(\cdot):=X^{\hat{u}}(\cdot)$. Let $(\hat{Y}(\cdot),\hat{Z}(\cdot,\cdot))\in\cM^{2,-\mu}_\bF(0,\infty;\bR^n\times\bR^{n\times d})$ be the adapted M-solution of the infinite horizon BSVIE~\eqref{adjoint equation}. Assume that the function
\begin{equation*}
	\bR^n\times U\ni(x,u)\mapsto H_\lambda\bigl(t,x,u,\bE_t\bigl[\hat{Y}(\cdot)\bigr],\hat{Z}(\cdot,t)\bigr)\in\bR
\end{equation*}
is convex for a.e.\ $t\geq0$, a.s. Furthermore, assume that the condition \eqref{optimality condition Hamiltonian} holds for a.e.\ $t\geq0$, a.s. Then $\hat{u}(\cdot)$ is optimal.
\end{theo}


\begin{proof}
Take an arbitrary $u(\cdot)\in\cU_{-\mu}$, and denote the corresponding state process by $X(\cdot):=X^u(\cdot)$. Observe that
\begin{align*}
	&J_\lambda(u(\cdot))-J_\lambda(\hat{u}(\cdot))\\
	&=\bE\Bigl[\int^\infty_0e^{-\lambda t}\bigl\{h(t,X(t),u(t))-h(t,\hat{X}(t),\hat{u}(t))\bigr\}\rd t\Bigr]\\
	&=\bE\Bigl[\int^\infty_0e^{-\lambda t}\Bigl\{H_\lambda\bigl(t,X(t),u(t),\bE_t\bigl[\hat{Y}(\cdot)\bigr],\hat{Z}(\cdot,t)\bigr)-H_\lambda\bigl(t,\hat{X}(t),\hat{u}(t),\bE_t\bigl[\hat{Y}(\cdot)\bigr],\hat{Z}(\cdot,t)\bigr)\\
	&\hspace{2.5cm}-\int^\infty_te^{-\lambda(s-t)}\bigl\langle b(s,t,X(t),u(t))-b(s,t,\hat{X}(t),\hat{u}(t)),\bE_t\bigl[\hat{Y}(s)\bigr]\bigr\rangle\rd s\\
	&\hspace{2.5cm}-\int^\infty_te^{-\lambda(s-t)}\sum^d_{k=1}\bigl\langle \sigma^k(s,t,X(t),u(t))-\sigma^k(s,t,\hat{X}(t),\hat{u}(t)),\hat{Z}^k(s,t)\bigr\rangle\rd s\Bigr\}\rd t\Bigr].
\end{align*}
By using Fubini's theorem and the definition of adapted M-solutions, we see that
\begin{align*}
	&\bE\Bigl[\int^\infty_0e^{-\lambda t}\Bigl\{\int^\infty_te^{-\lambda(s-t)}\bigl\langle b(s,t,X(t),u(t))-b(s,t,\hat{X}(t),\hat{u}(t)),\bE_t\bigl[\hat{Y}(s)\bigr]\bigr\rangle\rd s\\
	&\hspace{2.5cm}+\int^\infty_te^{-\lambda(s-t)}\sum^d_{k=1}\bigl\langle \sigma^k(s,t,X(t),u(t))-\sigma^k(s,t,\hat{X}(t),\hat{u}(t)),\hat{Z}^k(s,t)\bigr\rangle\rd s\Bigr\}\rd t\Bigr]\\
	&=\bE\Bigl[\int^\infty_0e^{-\lambda t}\Bigl\langle\int^t_0\bigl(b(t,s,X(s),u(s))-b(t,s,\hat{X}(s),\hat{u}(s))\bigr)\rd s,\hat{Y}(t)\Bigr\rangle\rd t\\
	&\hspace{1.5cm}+\int^\infty_0e^{-\lambda t}\int^t_0\sum^d_{k=1}\bigl\langle \sigma^k(t,s,X(s),u(s))-\sigma^k(t,s,\hat{X}(s),\hat{u}(s)),\hat{Z}^k(t,s)\bigr\rangle\rd s\rd t\Bigr]\\
	&=\bE\Bigl[\int^\infty_0e^{-\lambda t}\Bigl\langle\int^t_0\bigl(b(t,s,X(s),u(s))-b(t,s,\hat{X}(s),\hat{u}(s))\bigr)\rd s\\
	&\hspace{3.5cm}+\int^t_0\bigl(\sigma(t,s,X(s),u(s))-\sigma(t,s,\hat{X}(s),\hat{u}(s))\bigr)\rd W(s),\hat{Y}(t)\Bigr\rangle\rd t\Bigr].
\end{align*}
Noting the controlled SVIE~\eqref{controlled SVIE} and the equation~\eqref{adjoint equation Hamiltonian}, we see that the last term in the above equalities is equal to
\begin{equation*}
	\bE\Bigl[\int^\infty_0e^{-\lambda t}\bigl\langle X(t)-\hat{X}(t),\partial_xH_\lambda\bigl(t,\hat{X}(t),\hat{u}(t),\bE_t\bigl[\hat{Y}(\cdot)\bigr],\hat{Z}(\cdot,t)\bigr)\bigr\rangle\rd t\Bigr].
\end{equation*} Consequently, we obtain
\begin{align*}
	J_\lambda(u(\cdot))-J_\lambda(\hat{u}(\cdot))&=\bE\Bigl[\int^\infty_0e^{-\lambda t}\Bigl\{H_\lambda\bigl(t,X(t),u(t),\bE_t\bigl[\hat{Y}(\cdot)\bigr],\hat{Z}(\cdot,t)\bigr)-H_\lambda\bigl(t,\hat{X}(t),\hat{u}(t),\bE_t\bigl[\hat{Y}(\cdot)\bigr],\hat{Z}(\cdot,t)\bigr)\\
	&\hspace{2.5cm}-\bigl\langle X(t)-\hat{X}(t),\partial_xH_\lambda\bigl(t,\hat{X}(t),\hat{u}(t),\bE_t\bigl[\hat{Y}(\cdot)\bigr],\hat{Z}(\cdot,t)\bigr)\bigr\rangle\Bigr\}\rd t\Bigr].
\end{align*}
Therefore, by the assumptions, it holds that
\begin{equation*}
	J_\lambda(u(\cdot))-J_\lambda(\hat{u}(\cdot))\geq\bE\Bigl[\int^\infty_0e^{-\lambda t}\bigl\langle u(t)-\hat{u}(t),\partial_uH_\lambda\bigl(t,\hat{X}(t),\hat{u}(t),\bE_t\bigl[\hat{Y}(\cdot)\bigr],\hat{Z}(\cdot,t)\bigr)\bigr\rangle\rd t\Bigr]\geq0.
\end{equation*}
Since $u(\cdot)\in\cU_{-\mu}$ is arbitrary, we see that
\begin{equation*}
	J_\lambda(\hat{u}(\cdot))=\inf_{u(\cdot)\in\cU_{-\mu}}J_\lambda(u(\cdot)),
\end{equation*}
and thus $\hat{u}(\cdot)$ is optimal. This completes the proof.
\end{proof}


\begin{rem}
The above arguments can be easily generalized to the case where $b(t,s,0,0)$ and $\sigma(t,s,0,0)$ are nonzero (see Remarks~\ref{rem: SVIE nonzero} and \ref{rem: BSVIE nonzero}). Also, the case of random coefficients $b$ and $\sigma$ (with suitable measurability and integrability conditions) can be treated by the same way as above.
\end{rem}


\subsection{Examples}\label{subsection: example}


In this subsection, we apply the above results to some concrete examples. First, we investigate an infinite horizon linear-quadratic (LQ for short) control problem for an SVIE. 


\begin{exam}[Infinite horizon LQ stochastic control problems]\label{exam: LQ}
Let $U=\bR^\ell$. For each control process $u(\cdot)\in\cU_{-\mu}$, consider the following controlled linear SVIE:
\begin{equation}\label{controlled linear SVIE}
	X^u(t)=\varphi(t)+\int^t_0\bigl\{A(t,s)X^u(s)+B(t,s)u(s)\bigr\}\rd s+\sum^d_{k=1}\int^t_0\bigl\{C_k(t,s)X^u(s)+D_k(t,s)u(s)\bigr\}\rd W_k(s),\ t\geq0,
\end{equation}
where $\varphi(\cdot)\in L^{2,-\mu}_\bF(0,\infty;\bR^n)$ is a given initial process, and $A,C_k:\Delta^\comp[0,\infty)\to\bR^{n\times n}$ and $B,D_k:\Delta^\comp[0,\infty)\to\bR^{n\times\ell}$, $k=1,\dots,d$, are given measurable functions. Define a discounted quadratic cost functional by
\begin{equation}\label{quadratic cost}
	J_\lambda(u(\cdot)):=\frac{1}{2}\bE\Bigl[\int^\infty_0e^{-\lambda t}\bigl\{\langle M_1(t)X^u(t),X^u(t)\rangle+\langle M_2(t)u(t),u(t)\rangle+2\langle M_3(t)X^u(t),u(t)\rangle\bigr\}\rd t\Bigr],
\end{equation}
where $M_1:[0,\infty)\to\bR^{n\times n}$, $M_2:[0,\infty)\to\bR^{\ell\times\ell}$, and $M_3:[0,\infty)\to\bR^{\ell\times n}$ are given measurable and bounded functions. We assume that there exist $K_{b,x},K_{b,u}\in L^{1,*}(0,\infty;\bR_+)$ and $K_{\sigma,x},K_{\sigma,u}\in L^{2,*}(0,\infty;\bR_+)$ such that
\begin{align*}
	&|A(t,s)|\leq K_{b,x}(t-s),\ |B(t,s)|\leq K_{b,u}(t-s),\\
	&\Bigl(\sum^d_{k=1}|C_k(t,s)|^2\Bigr)^{1/2}\leq K_{\sigma,x}(t-s),\ \Bigl(\sum^d_{k=1}|D_k(t,s)|^2\Bigr)^{1/2}\leq K_{\sigma,u}(t-s),
\end{align*}
for a.e.\ $(t,s)\in\Delta^\comp[0,\infty)$, and that $M_1$ and $M_2$ take values in the spaces of symmetric matrices with suitable dimensions. Define $\rho_{b,\sigma;x,u}\in\bR$ by \eqref{domain: control}. We see that if $\mu>\rho_{b,\sigma;x,u}$ and $\lambda\geq2\mu$, then for any $u(\cdot)\in\cU_{-\mu}$, there exists a unique solution $X^u(\cdot)\in L^{2,-\mu}_\bF(0,\infty;\bR^n)$ of the controlled linear SVIE~\eqref{controlled linear SVIE}, and the discounted quadratic cost functional \eqref{quadratic cost} is well-defined. An infinite horizon LQ stochastic control problem is a problem to seek a control process $\hat{u}(\cdot)\in\cU_{-\mu}=L^{2,-\mu}_\bF(0,\infty;\bR^\ell)$ which minimizes $J_\lambda(u(\cdot))$ over all control processes $u(\cdot)\in\cU_{-\mu}$. The corresponding Hamiltonian is defined by
\begin{align*}
	H_\lambda(t,x,u,p(\cdot),q(\cdot))&=\frac{1}{2}\bigl\{\langle M_1(t)x,x\rangle+\langle M_2(t)u,u\rangle+2\langle M_3(t)x,u\rangle\bigr\}\\
	&\hspace{0.5cm}+\int^\infty_te^{-\lambda(s-t)}\Bigl\{\langle A(s,t)x+B(s,t)u,p(s)\rangle+\sum^d_{k=1}\langle C_k(s,t)x+D_k(s,t)u,q^k(s)\rangle\Bigr\}\rd s
\end{align*}
for each $(t,x,u,p(\cdot),q(\cdot))\in[0,\infty)\times\bR^n\times\bR^\ell\times L^{2,-\mu}(0,\infty;\bR^n)\times L^{2,-\mu}(0,\infty;\bR^{n\times d})$. Let $\hat{u}(\cdot)$ be given, and define $\hat{X}(\cdot):=X^{\hat{u}}(\cdot)\in L^{2,-\mu}_\bF(0,\infty;\bR^n)$. The corresponding adjoint equation~\eqref{adjoint equation} becomes
\begin{align*}
	&\hat{Y}(t)=M_1(t)\hat{X}(t)+M_3(t)^\top\hat{u}(t)+\int^\infty_te^{-\lambda(s-t)}\Bigl\{A(s,t)^\top\hat{Y}(s)+\sum^d_{k=1}C_k(s,t)^\top\hat{Z}^k(s,t)\Bigr\}\rd s\\
	&\hspace{2cm}-\int^\infty_t\hat{Z}(t,s)\rd W(s),\ t\geq0,
\end{align*}
which admits a unique adapted M-solution $(\hat{Y}(\cdot),\hat{Z}(\cdot,\cdot))\in\cM^{2,-\mu}_\bF(0,\infty;\bR^n\times\bR^{n\times d})$. Also, noting that $U=\bR^\ell$, the optimality condition \eqref{optimality condition Hamiltonian} becomes
\begin{equation}\label{exam: LQ: 1}
	M_2(t)\hat{u}(t)+M_3(t)\hat{X}(t)+\int^\infty_te^{-\lambda(s-t)}\Bigl\{B(s,t)^\top\bE_t\bigl[\hat{Y}(s)]+\sum^d_{k=1}D_k(s,t)^\top\hat{Z}^k(s,t)\Bigr\}\rd s=0.
\end{equation}
We note that if the symmetric matrix $\left(\begin{array}{cc}M_1(t)&M_3(t)^\top\\M_3(t)&M_2(t)\end{array}\right)$ is nonnegative-definite for a.e.\ $t\geq0$, then the function $(x,u)\mapsto H_\lambda(t,x,u,p(\cdot),q(\cdot))$ is convex for a.e.\ $t\geq0$ and any $p(\cdot)$ and $q(\cdot)$. In this case, by Theorems~\ref{theo: maximum principle} and \ref{theo: sufficient maximum principle}, we see that $\hat{u}(\cdot)$ is optimal if and only if \eqref{exam: LQ: 1} holds for a.e.\ $t\geq0$, a.s.
\end{exam}


\begin{rem}
In finite horizon settings, LQ control problems for SVIEs with non-singular coefficients were studied by Chen and Yong~\cite{ChYo07} and Wang~\cite{WaT18}. Also, Abi Jaber, Miller, and Pham~\cite{AbMiPh19} investigated a finite horizon LQ control problem for an SVIE with singular coefficients of convolution type. Compared with these papers, the above example provides a characterization of the optimal control of an infinite horizon LQ control problem for an SVIE with singular coefficients of non-convolution type. However, the optimality condition~\eqref{exam: LQ: 1} alone dose not give a (causal) state-feedback representation of the optimal control, and further analysis would be needed. We hope to report some further results on this problem in the near future.
\end{rem}

Next, we observe a classical SDEs case.


\begin{exam}[Discounted control problems for SDEs]\label{exam: SDE}
If the coefficients $\varphi(t)=x_0\in\bR^n$, $b(t,s,x,u)=b(s,x,u)$, and $\sigma(t,s,x,u)=\sigma(s,x,u)$ are independent of $t$, then the stochastic control problem becomes a usual (infinite horizon) problem for a controlled SDE (see \cite{MaVe14,OrVe17}). Assume that $b(s,0,0)$ and $\sigma(s,0,0)$ are bounded, $b(s,x,u)$ and $\sigma(s,x,u)$ are continuously differentiable with respect to $(x,u)\in\bR^n\times\bR^\ell$ for a.e.\ $s\geq0$, and there exist four constants $L_{b,x},L_{b,u},L_{\sigma,x},L_{\sigma,u}>0$ such that
\begin{align*}
	&|b(s,x,u)-b(s,x',u')|\leq L_{b,x}|x-x'|+L_{b,u}|u-u'|,\\
	&|\sigma(s,x,u)-\sigma(s,x',u')|\leq L_{\sigma,x}|x-x'|+L_{\sigma,u}|u-u'|,
\end{align*}
for any $x,x'\in\bR^n$ and $u,u'\in\bR^\ell$, for a.e.\ $s\geq0$. In this case, the condition~\eqref{mu lambda} becomes $\mu>f^{-1}_1(1)$ and $\lambda\geq2\mu$, where $f_1:(0,\infty)\to(0,\infty)$ is defined by $f_1(\rho):=\frac{L_{b,x}}{\rho}+\frac{L_{\sigma,x}}{\sqrt{2\rho}}$ for $\rho\in(0,\infty)$. The Hamiltonian $H_\lambda$ is reduced to
\begin{align*}
	H_\lambda(t,x,u,p(\cdot),q(\cdot))&=h(t,x,u)+\Bigl\langle b(t,x,u),\int^\infty_te^{-\lambda(s-t)}p(s)\rd s\Bigr\rangle+\sum^d_{k=1}\Bigl\langle\sigma^k(t,x,u),\int^\infty_te^{-\lambda(s-t)}q^k(s)\rd s\Bigr\rangle\\
	&=\mathcal{H}\Bigl(t,x,u,\int^\infty_te^{-\lambda(s-t)}p(s)\rd s,\int^\infty_te^{-\lambda(s-t)}q(s)\rd s\Bigr),
\end{align*}
where $\mathcal{H}:[0,\infty)\times\bR^n\times U\times\bR^n\times\bR^{n\times d}\to\bR$ is the standard Hamiltonian defined by
\begin{equation*}
	\mathcal{H}(t,x,u,p,q):=h(t,x,u)+\langle b(t,x,u),p\rangle+\sum^d_{k=1}\langle\sigma^k(t,x,u),q^k\rangle
\end{equation*}
for $(t,x,u,p,q)\in[0,\infty)\times\bR^n\times U\times\bR^n\times\bR^{n\times d}$. Let $\hat{u}(\cdot)$ be given, and define $\hat{X}(\cdot):=X^{\hat{u}}(\cdot)\in L^{2,-\mu}_\bF(0,\infty;\bR^n)$. Then the adjoint equation \eqref{adjoint equation} can be written as
\begin{align*}
	&\hat{Y}(t)=\partial_xh(t,\hat{X}(t),\hat{u}(t))+\int^\infty_te^{-\lambda(s-t)}\Bigl\{\partial_xb(t,\hat{X}(t),\hat{u}(t))^\top\hat{Y}(s)+\sum^d_{k=1}\partial_x\sigma^k(t,\hat{X}(t),\hat{u}(t))^\top\hat{Z}^k(s,t)\Bigr\}\rd s\\
	&\hspace{3cm}-\int^\infty_t\hat{Z}(t,s)\rd W(s),\ t\geq0,
\end{align*}
which admits a unique adapted M-solution $(\hat{Y}(\cdot),\hat{Z}(\cdot,\cdot))\in\cM^{2,-\mu}_\bF(0,\infty;\bR^n\times\bR^{n\times d})$. Furthermore, the following holds:
\begin{equation}\label{adjoint equation SDE}
	\hat{Y}(t)=\partial_x\mathcal{H}\Bigl(t,\hat{X}(t),\hat{u}(t),\bE_t\Bigl[\int^\infty_te^{-\lambda(s-t)}\hat{Y}(s)\rd s\Bigr],\int^\infty_te^{-\lambda(s-t)}\hat{Z}(s,t)\rd s\Bigr)
\end{equation}
for a.e.\ $t\geq0$, a.s. Now we define $\hat{\cY}(t):=\bE_t[\int^\infty_te^{-\lambda(s-t)}\hat{Y}(s)\rd s]$ and $\hat{\cZ}(t):=\int^\infty_te^{-\lambda(s-t)}\hat{Z}(s,t)\rd s$ for $t\geq0$. By Lemma~\ref{lemm: BSVIE BSDE}, we have $(\hat{\cY}(\cdot),\hat{\cZ}(\cdot))\in L^{2,-\mu}_\bF(0,\infty;\bR^n)\times L^{2,-\mu}_\bF(0,\infty;\bR^{n\times d})$, and it is the unique adapted solution of the infinite horizon BSDE $\mathrm{d}\hat{\cY}(t)=-\{\hat{Y}(t)-\lambda\hat{\cY}(t)\}\rd t+\hat{\cZ}(t)\rd W(t)$, $t\geq0$. By inserting \eqref{adjoint equation SDE} into this formula, we see that $(\hat{\cY}(\cdot),\hat{\cZ}(\cdot))$ solves the following infinite horizon BSDE:
\begin{equation*}
	\mathrm{d}\hat{\cY}(t)=-\bigl\{\partial_x\mathcal{H}\bigl(t,\hat{X}(t),\hat{u}(t),\hat{\cY}(t),\hat{\cZ}(t)\bigr)-\lambda\hat{\cY}(t)\bigr\}\rd t+\hat{\cZ}(t)\rd W(t),\ t\geq0.
\end{equation*}
Conversely, if a pair $(\hat{\cY}(\cdot),\hat{\cZ}(\cdot))\in L^{2,-\mu}_\bF(0,\infty;\bR^n)\times L^{2,-\mu}_\bF(0,\infty;\bR^{n\times d})$ satisfies the above infinite horizon BSDE, then $(\hat{Y}(\cdot),\hat{Z}(\cdot,\cdot))\in\cM^{2,-\mu}_\bF(0,\infty;\bR^n\times\bR^{n\times d})$ defined by
\begin{equation*}
\begin{cases}
	\hat{Y}(t)=\partial_x\mathcal{H}\bigl(t,\hat{X}(t),\hat{u}(t),\hat{\cY}(t),\hat{\cZ}(t)\bigr),\\
	\int^\infty_0\hat{Z}(t,s)\rd W(s)=\partial_x\mathcal{H}\Bigl(t,\hat{X}(t),\hat{u}(t),\int^\infty_te^{-\lambda(s-t)}\hat{Y}(s)\rd s,\hat{\cZ}(t)\Bigr)\\
	\hspace{3.5cm}-\bE\Bigl[\partial_x\mathcal{H}\Bigl(t,\hat{X}(t),\hat{u}(t),\int^\infty_te^{-\lambda(s-t)}\hat{Y}(s)\rd s,\hat{\cZ}(t)\Bigr)\Bigr],
\end{cases}
\end{equation*}
satisfies \eqref{adjoint equation SDE}. Consequently, in this case, the optimality condition~\eqref{optimality condition Hamiltonian} is equivalent to the following condition:
\begin{equation*}
	\bigl\langle \partial_u\mathcal{H}\bigl(t,\hat{X}(t),\hat{u}(t),\hat{\cY}(t),\hat{\cZ}(t)\bigr),u-\hat{u}(t)\bigr\rangle\geq0,\ \forall\,u\in U.
\end{equation*}
Therefore, Theorems~\ref{theo: maximum principle} and \ref{theo: sufficient maximum principle} are reduced to the counterparts in SDEs cases~\cite{MaVe14,OrVe17}.
\end{exam}

The next example is concerned with a controlled dynamics with a fractional time derivative.


\begin{exam}[Discounted control problems for fractional stochastic differential equations]\label{exam: FSDE}
Consider a minimization problem of a discounted cost functional $J_\lambda(u(\cdot))$ defined by \eqref{cost functional}. Suppose that the state process $X^u(\cdot)$ solves the following controlled \emph{Caputo fractional stochastic differential equation} (Caputo-FSDE for short):
\begin{equation}\label{FSDE}
	\begin{cases}
	\CD X^u(t)=b(t,X^u(t),u(t))+\sigma(t,X^u(t),u(t))\frac{\rd W(t)}{\rd t},\ t\geq0,\\
	X^u(0)=x_0,
	\end{cases}
\end{equation}
where $x_0\in\bR^n$ is a given initial state, $b:[0,\infty)\times\bR^n\times\bR^\ell\to\bR^n$ and $\sigma:[0,\infty)\times\bR^n\times\bR^\ell\to\bR^{n\times d}$ are given measurable maps satisfying the same assumptions as in Example~\ref{exam: SDE}, and $\CD$ denotes the Caputo fractional derivative of order $\alpha\in(\frac{1}{2},1)$ defined by
\begin{equation*}
	\CD f(t):=\frac{1}{\Gamma(1-\alpha)}\frac{\mathrm{d}}{\mathrm{d}t}\int^t_0(t-s)^{-\alpha}\{f(s)-f(0)\}\rd s,\ t\geq0,
\end{equation*}
for each differentiable function $f:[0,\infty)\to\bR$ (see \cite{KiSrTr06}). Here and elsewhere, $\Gamma(\alpha)=\int^\infty_0\tau^{\alpha-1}e^{-\tau}\rd\tau$ denotes the Gamma function. We say that $X^u(\cdot)\in L^{2,*}_\bF(0,\infty;\bR^n)$ is a solution of the controlled Caputo-FSDE~\eqref{FSDE} if it holds that
\begin{equation}\label{FSDE mild}
	X^u(t)=x_0+\frac{1}{\Gamma(\alpha)}\Bigl\{\int^t_0(t-s)^{\alpha-1}b(s,X^u(s),u(s))\rd s+\int^t_0(t-s)^{\alpha-1}\sigma(s,X^u(s),u(s))\rd W(s)\Bigr\}
\end{equation}
for a.e.\ $t\geq0$, a.s. For further discussions on the concept of solutions to Caputo-FSDEs, see \cite{WaXuKl16}. We note that \eqref{FSDE mild} can be seen as a controlled SVIE~\eqref{controlled SVIE} with the coefficients
\begin{equation*}
	b(t,s,x,u)=\frac{1}{\Gamma(\alpha)}(t-s)^{\alpha-1}b(s,x,u),\ \sigma(t,s,x,u)=\frac{1}{\Gamma(\alpha)}(t-s)^{\alpha-1}\sigma(s,x,u).
\end{equation*}
Note that the coefficients are singular at the diagonal $t=s$. In this case, we can take the functions $K_{b,x},K_{b,u},K_{\sigma,x},K_{\sigma,u}$ in Assumption~\ref{assum: control} as follows:
\begin{equation*}
	K_{b,x}(\tau)=\frac{L_{b,x}}{\Gamma(\alpha)}\tau^{\alpha-1},\ K_{b,u}(\tau)=\frac{L_{b,u}}{\Gamma(\alpha)}\tau^{\alpha-1}, K_{\sigma,x}(\tau)=\frac{L_{\sigma,x}}{\Gamma(\alpha)}\tau^{\alpha-1}, K_{\sigma,u}(\tau)=\frac{L_{\sigma,u}}{\Gamma(\alpha)}\tau^{\alpha-1}.
\end{equation*}
By simple calculations, we see that the condition~\eqref{mu lambda} becomes $\mu>f^{-1}_\alpha(1)\ \text{and}\ \lambda\geq2\mu$, where $f_\alpha:(0,\infty)\to(0,\infty)$ is defined by
\begin{equation*}
	f_\alpha(\rho):=[K_{b,x}]_1(\rho)+[K_{\sigma,x}]_2(\rho)=L_{b,x}\rho^{-\alpha}+L_{\sigma,x}\frac{\sqrt{\Gamma(2\alpha-1)}}{\Gamma(\alpha)}(2\rho)^{-(\alpha-1/2)},\ \rho\in(0,\infty).
\end{equation*}
Under the above conditions, for any control process $u(\cdot)\in\cU_{-\mu}$, there exists a unique solution $X^u(\cdot)\in L^{2,-\mu}_\bF(0,\infty;\bR^n)$ to the controlled Caputo-FSDE~\eqref{FSDE}, and the discounted cost functional $J_\lambda(u(\cdot))\in\bR$ is well-defined. The corresponding Hamiltonian $H_\lambda$ is defined by
\begin{align*}
	&H_\lambda(t,x,u,p(\cdot),q(\cdot))=h(t,x,u)+\frac{1}{\Gamma(\alpha)}\Bigl\{\Bigl\langle b(t,x,u),\int^\infty_te^{-\lambda(s-t)}(s-t)^{\alpha-1}p(s)\rd s\Bigr\rangle\\
	&\hspace{6cm}+\sum^d_{k=1}\Bigl\langle\sigma^k(t,x,u),\int^\infty_te^{-\lambda(s-t)}(s-t)^{\alpha-1}q^k(s)\rd s\Bigr\rangle\Bigr\}
\end{align*}
for $(t,x,u,p(\cdot),q(\cdot))\in[0,\infty)\times\bR^n\times U\times L^{2,-\mu}(0,\infty;\bR^n)\times L^{2,-\mu}(0,\infty;\bR^{n\times d})$. Let $\hat{u}(\cdot)\in\cU_{-\mu}$ be given, and define $\hat{X}(\cdot):=X^{\hat{u}}(\cdot)\in L^{2,-\mu}_\bF(0,\infty;\bR^n)$. Then the corresponding adjoint equation~\eqref{adjoint equation} becomes the following infinite horizon BSVIE:
\begin{align*}
	&\hat{Y}(t)=\partial_xh(t,\hat{X}(t),\hat{u}(t))+\frac{1}{\Gamma(\alpha)}\Bigl\{\partial_xb(t,\hat{X}(t),\hat{u}(t))^\top\int^\infty_te^{-\lambda(s-t)}(s-t)^{\alpha-1}\hat{Y}(s)\rd s\\
	&\hspace{5.5cm}+\sum^d_{k=1}\partial_x\sigma^k(t,\hat{X}(t),\hat{u}(t))^\top\int^\infty_te^{-\lambda(s-t)}(s-t)^{\alpha-1}\hat{Z}^k(s,t)\rd s\Bigr\}\\
	&\hspace{2cm}-\int^\infty_t\hat{Z}(t,s)\rd W(s),\ t\geq0,
\end{align*}
which admits a unique adapted M-solution $(\hat{Y}(\cdot),\hat{Z}(\cdot,\cdot))\in\cM^{2,-\mu}_\bF(0,\infty;\bR^n\times\bR^{n\times d})$. Furthermore, the optimality condition~\eqref{optimality condition Hamiltonian} becomes
\begin{equation}\label{FSDE optimality condition}
\begin{split}
	&\Bigl\langle\partial_uh(t,\hat{X}(t),\hat{u}(t))+\frac{1}{\Gamma(\alpha)}\Bigl\{\partial_ub(t,\hat{X}(t),\hat{u}(t))^\top\int^\infty_te^{-\lambda(s-t)}(s-t)^{\alpha-1}\bE_t\bigl[\hat{Y}(s)\bigr]\rd s\\
	&\hspace{2cm}+\sum^d_{k=1}\partial_u\sigma^k(t,\hat{X}(t),\hat{u}(t))^\top\int^\infty_te^{-\lambda(s-t)}(s-t)^{\alpha-1}\hat{Z}^k(s,t)\rd s\Bigr\},u-\hat{u}(t)\Bigr\rangle\geq0,\ \forall\,u\in U.
\end{split}
\end{equation}
Theorem~\ref{theo: maximum principle} implies that if $\hat{u}(\cdot)$ is optimal, then \eqref{FSDE optimality condition} holds for a.e.\ $t\geq0$, a.s. Conversely, by Theorem~\ref{theo: sufficient maximum principle}, if the map
\begin{align*}
	&(x,u)\mapsto h(t,x,u)+\frac{1}{\Gamma(\alpha)}\Bigl\{\Bigl\langle b(t,x,u),\int^\infty_te^{-\lambda(s-t)}(s-t)^{\alpha-1}\bE_t\bigl[\hat{Y}(s)\bigr]\rd s\Bigr\rangle\\
	&\hspace{4cm}+\sum^d_{k=1}\Bigl\langle\sigma^k(t,x,u),\int^\infty_te^{-\lambda(s-t)}(s-t)^{\alpha-1}\hat{Z}^k(s,t)\rd s\Bigr\rangle\Bigr\}
\end{align*}
is convex for a.e.\ $t\geq0$, a.s., and if \eqref{FSDE optimality condition} holds for a.e.\ $t\geq0$, a.s., then the control process $\hat{u}(\cdot)$ is optimal.
\end{exam}


\begin{rem}
The above observations are also valid in the cases of other fractional derivatives such as the Riemann--Liouville fractional derivative (see \cite{KiSrTr06}) and the Hilfer fractional derivative (see \cite{Hi00}). We remark that Lin and Yong~\cite{LiYo20} established Pontryagin's (necessary) maximum principle for a finite horizon control problem of a deterministic singular Volterra equation and applied it to fractional order ordinary differential equations of the Riemann--Liouville and Caputo types. Our result shown in the above example is a stochastic and infinite horizon version of \cite{LiYo20}.
\end{rem}

Lastly, we consider a controlled stochastic integro-differential equation of an It\^{o}--Volterra type, which can be seen as an SDE with unbounded delay. Since many systems arising from realistic models can be described as differential equations with unbounded delay (see \cite{DoHeHe11} and references cited therein), optimal control problems for integro-differential systems are important problems.


\begin{exam}[Discounted control problems for stochastic integro-differential equations]\label{exam: integro-differential}
Consider a minimization problem of a discounted cost functional $J_\lambda(u(\cdot))$ defined by \eqref{cost functional}. Suppose that the state process $X^u(\cdot)$ solves the following controlled \emph{stochastic integro-differential equation}:
\begin{equation}\label{integro-differential}
	\begin{cases}
	\mathrm{d}X^u(t)=b\Bigl(t,X^u(t),u(t),\int^t_0A_1(t,s)X^u(s)\rd s,\int^t_0A_2(t,s)u(s)\rd s\Bigr)\rd t\\
	\hspace{2cm}+\sigma\Bigl(t,X^u(t),u(t),\int^t_0A_3(t,s)X^u(s)\rd s,\int^t_0A_4(t,s)u(s)\rd s\Bigr)\rd W(t),\ t\geq0,\\
	X^u(0)=x_0,
	\end{cases}
\end{equation}
where $x_0\in\bR^n$ is a given initial state, $A_1:\Delta^\comp[0,\infty)\to\bR^{m_1\times n}$, $A_2:\Delta^\comp[0,\infty)\to\bR^{m_2\times\ell}$, $A_3:\Delta^\comp[0,\infty)\to\bR^{m_3\times n}$, and $A_4:\Delta^\comp[0,\infty)\to\bR^{m_4\times\ell}$ are given measurable functions with $m_1,m_2,m_3,m_4\in\bN$, and $b:[0,\infty)\times\bR^n\times\bR^\ell\times\bR^{m_1}\times\bR^{m_2}\to\bR^n$ and $\sigma:[0,\infty)\times\bR^n\times\bR^\ell\times\bR^{m_3}\times\bR^{m_4}\to\bR^{n\times d}$ are given measurable maps. Similar equations without control were studied by, for example, Appleby~\cite{Ap03} and Mao and Riedle~\cite{MaRi06}. Also, Sakthivel, Nieto, and Mahmudov~\cite{SaNiMa10} investigated approximate controllability of both deterministic and stochastic integro-differential equations, which include delay of the state process alone (without delay of the control process). We note that the controlled stochastic integro-differential equation \eqref{integro-differential} includes not only delay of the state process $X^u(\cdot)$, but also delay of the control process $u(\cdot)$.

Assume that there exist $K_1,K_2,K_3,K_4\in L^{1,*}(0,\infty;\bR_+)$ such that $|A_i(t,s)|\leq K_i(t-s)$ for a.e.\ $(t,s)\in\Delta^\comp[0,\infty)$, $i=1,2,3,4$. Furthermore, assume that $b(t,0,0,0,0)$ and $\sigma(t,0,0,0,0)$ are bounded, $b(t,x,u,x_1,x_2)$ and $\sigma(t,x,u,x_3,x_4)$ are continuously differentiable with respect to $(x,u,x_1,x_2,x_3,x_4)\in\bR^n\times\bR^\ell\times\bR^{m_1}\times\bR^{m_2}\times\bR^{m_3}\times\bR^{m_4}$ for a.e.\ $t\geq0$, and there exist eight constants $L_{b,x},L_{b,u},L_{\sigma,x},L_{\sigma,u},L_i>0$, $i=1,2,3,4$, such that
\begin{align*}
	&|b(t,x,u,x_1,x_2)-b(t,x',u',x'_1,x'_2)|\leq L_{b,x}|x-x'|+L_{b,u}|u-u'|+L_1|x_1-x'_1|+L_2|x_2-x'_2|,\\
	&|\sigma(t,x,u,x_3,x_4)-\sigma(t,x',u',x'_3,x'_4)|\leq L_{\sigma,x}|x-x'|+L_{\sigma,u}|u-u'|+L_3|x_3-x'_3|+L_4|x_4-x'_4|,
\end{align*}
for any $(x,u,x_1,x_2,x_3,x_4),(x',u',x'_1,x'_2,x'_3,x'_4)\in\bR^n\times\bR^\ell\times\bR^{m_1}\times\bR^{m_2}\times\bR^{m_3}\times\bR^{m_4}$, for a.e.\ $t\geq0$. Let $N=n+m_1+m_2+m_3+m_4$, and define an $N$-dimensional auxiliary state process $\bX^u(\cdot)$ by
\begin{equation*}
	\bX^u(t):=\left(\begin{array}{c}X^u(t)\\\int^t_0A_1(t,s)X^u(s)\rd s\\\int^t_0A_2(t,s)u(s)\rd s\\\int^t_0A_3(t,s)X^u(s)\rd s\\\int^t_0A_4(t,s)u(s)\rd s\end{array}\right),\ t\geq0.
\end{equation*}
Then the stochastic integro-differential equation~\eqref{integro-differential} can be written as the following SVIE for $\bX^u(\cdot)$:
\begin{equation}\label{SVIE bb}
	\bX^u(t)=\bX_0+\int^t_0\tilde{b}(t,s,\bX^u(s),u(s))\rd s+\int^t_0\tilde{\sigma}(s,\bX^u(s),u(s))\rd W(s),\ t\geq0,
\end{equation}
where
\begin{align*}
	\bX_0:=\left(\begin{array}{c}x_0\\0\\0\\0\\0\end{array}\right)\in\bR^N,
\end{align*}
and $\tilde{b}:\Delta^\comp[0,\infty)\times\bR^N\times\bR^\ell\to\bR^N$ and $\tilde{\sigma}:[0,\infty)\times\bR^N\times\bR^\ell\to\bR^{N\times d}$ are defined by
\begin{equation*}
	\tilde{b}(t,s,\bX,u):=\left(\begin{array}{c}b(s,J_n\bX,u,J_{m_1}\bX,J_{m_2}\bX)\\A_1(t,s)J_n\bX\\A_2(t,s)u\\A_3(t,s)J_n\bX\\A_4(t,s)u\end{array}\right),\ \tilde{\sigma}(s,\bX,u):=\left(\begin{array}{c}\sigma(s,J_n\bX,u,J_{m_3}\bX,J_{m_4}\bX)\\0\\0\\0\\0\end{array}\right),
\end{equation*}
for $(t,s,\bX,u)\in\Delta^\comp[0,\infty)\times\bR^N\times\bR^\ell$. Here, $J_n\in\bR^{n\times N}$ and $J_{m_i}\in\bR^{m_i\times N}$, $i=1,2,3,4$, are defined by
\begin{equation*}
	J_n:=(I_n,0_{n\times m_1},0_{n\times m_2},0_{n\times m_3},0_{n\times m_4}),\ J_{m_1}:=(0_{m_1\times n},I_{m_1},0_{m_1\times m_2},0_{m_1\times m_3},0_{m_1\times m_4}),
\end{equation*}
and similar for $J_{m_2}$, $J_{m_3}$, and $J_{m_4}$. Also, the discounted cost functional can be written as
\begin{equation*}
	J_\lambda(u(\cdot))=\bE\Bigl[\int^\infty_0e^{-\lambda t}h(t,X^u(t),u(t))\rd t\Bigr]=\bE\Bigl[\int^\infty_0e^{-\lambda t}\tilde{h}(t,\bX^u(t),u(t))\rd t\Bigr]
\end{equation*}
with $\tilde{h}(t,\bX,u):=h(t,J_n\bX,u)$ for $(t,\bX,u)\in[0,\infty)\times\bR^N\times\bR^\ell$. By simple calculations show that the condition \eqref{mu lambda} becomes
\begin{equation*}
	\mu>0,\ [K_2]_1(\mu)+[K_4]_1(\mu)<\infty,\ \frac{L_{b,x}+L_1+L_2}{\mu}+\frac{L_{\sigma,x}+L_3+L_4}{\sqrt{2\mu}}+[K_1]_1(\mu)+[K_3]_1(\mu)<1,\ \lambda\geq2\mu.
\end{equation*}
In this case, for any $u(\cdot)\in\cU_{-\mu}$, there exists a unique solution $\bX^u(\cdot)\in L^{2,-\mu}_\bF(0,\infty;\bR^N)$ of the auxiliary controlled SVIE~\eqref{SVIE bb}, and the discounted cost functional $J_\lambda(u(\cdot))$ is well-defined. Note that $X^u(\cdot)=J_n\bX^u(\cdot)\in L^{2,-\mu}_\bF(0,\infty;\bR^n)$ is the unique solution of the original controlled stochastic integro-differential equation~\eqref{integro-differential}. The corresponding Hamiltonian for the auxiliary controlled SVIE~\eqref{SVIE bb} is defined by
\begin{align*}
	H_\lambda(t,\bX,u,p(\cdot),q(\cdot))&=\tilde{h}(t,\bX,u)+\int^\infty_te^{-\lambda(s-t)}\Bigl\{\langle\tilde{b}(s,t,\bX,u),p(s)\rangle+\sum^d_{k=1}\langle\tilde{\sigma}^k(t,\bX,u),q^k(s)\rangle\Bigr\}\rd s\\
	&=h(t,J_n\bX,u)+\Bigl\langle b(t,J_n\bX,u,J_{m_1}\bX,J_{m_2}\bX),\int^\infty_te^{-\lambda(s-t)}J_np(s)\rd s\Bigr\rangle\\
	&\hspace{0.7cm}+\int^\infty_te^{-\lambda(s-t)}\langle A_1(s,t)J_n\bX,J_{m_1}p(s)\rangle\rd s+\int^\infty_te^{-\lambda(s-t)}\langle A_2(s,t)u,J_{m_2}p(s)\rangle\rd s\\
	&\hspace{0.7cm}+\int^\infty_te^{-\lambda(s-t)}\langle A_3(s,t)J_n\bX,J_{m_3}p(s)\rangle\rd s+\int^\infty_te^{-\lambda(s-t)}\langle A_4(s,t)u,J_{m_4}p(s)\rangle\rd s\\
	&\hspace{0.7cm}+\sum^d_{k=1}\Bigl\langle\sigma^k(t,J_n\bX,u,J_{m_3}\bX,J_{m_4}\bX),\int^\infty_te^{-\lambda(s-t)}J_nq^k(s)\rd s\Bigr\rangle
\end{align*}
for $(t,\bX,u,p(\cdot),q(\cdot))\in[0,\infty)\times\bR^N\times U\times L^{2,-\mu}(0,\infty;\bR^N)\times L^{2,-\mu}(0,\infty;\bR^{N\times d})$. Let $\hat{u}(\cdot)\in\cU_{-\mu}$ be given, and define $\hat{X}(\cdot):=X^{\hat{u}}(\cdot)\in L^{2,-\mu}_\bF(0,\infty;\bR^n)$ and $\hat{\bX}(\cdot):=\bX^{\hat{u}}(\cdot)\in L^{2,-\mu}_\bF(0,\infty;\bR^N)$. The corresponding adjoint equation is the following infinite horizon BSVIE:
\begin{equation}\label{adjoint equation integro-differential}
\begin{split}
	&\hat{\bY}(t)=\partial_\bX \tilde{h}(t,\hat{\bX}(t),\hat{u}(t))+\int^\infty_te^{-\lambda(s-t)}\Bigl\{\partial_\bX\tilde{b}(s,t,\hat{\bX}(t),\hat{u}(t))^\top\hat{\bY}(s)+\sum^d_{k=1}\partial_\bX\tilde{\sigma}^k(t,\hat{\bX}(t),\hat{u}(t))^\top\hat{\bZ}^k(s,t)\Bigr\}\rd s\\
	&\hspace{3cm}-\int^\infty_t\hat{\bZ}(t,s)\rd W(s),\ t\geq0,
\end{split}
\end{equation}
which admits a unique adapted M-solution $(\hat{\bY}(\cdot),\hat{\bZ}(\cdot,\cdot))\in\cM^{2,-\mu}_\bF(0,\infty;\bR^N\times\bR^{N\times d})$. Also, the optimality condition is
\begin{equation}\label{optimal condition auxiliary}
\begin{split}
	&\Bigl\langle \partial_u\tilde{h}(t,\hat{\bX}(t),\hat{u}(t))+\int^\infty_te^{-\lambda(s-t)}\Bigl\{\partial_u\tilde{b}(s,t,\hat{\bX}(t),\hat{u}(t))^\top\bE_t\bigl[\hat{\bY}(s)\bigr]+\sum^d_{k=1}\partial_u\tilde{\sigma}^k(t,\hat{\bX}(t),\hat{u}(t))^\top\hat{\bZ}^k(s,t)\Bigr\}\rd s,\\
	&\hspace{5cm}u-\hat{u}(t)\Bigr\rangle\geq0,\ \forall\,u\in U.
\end{split}
\end{equation}

We further observe the adjoint equation~\eqref{adjoint equation integro-differential} and the optimality condition~\eqref{optimal condition auxiliary}. Define $(\hat{Y}(\cdot),\hat{Z}(\cdot,\cdot))\in\cM^{2,-\mu}_\bF(0,\infty;\bR^n\times\bR^{n\times d})$ by $\hat{Y}(\cdot):=J_n\hat{\bY}(\cdot)$ and $\hat{Z}(\cdot,\cdot):=J_n\hat{\bZ}(\cdot,\cdot)$. Also, for each $i=1,2,3,4$, define $(\hat{Y}_i(\cdot),\hat{Z}_i(\cdot,\cdot))\in\cM^{2,-\mu}_\bF(0,\infty;\bR^{m_i}\times\bR^{m_i\times d})$ by $\hat{Y}_i(\cdot):=J_{m_i}\hat{\bY}(\cdot)$ and $\hat{Z}_i(\cdot,\cdot):=J_{m_i}\hat{\bZ}(\cdot,\cdot)$. Then we have
\begin{equation}\label{Y equation}
\begin{split}
	&\hat{Y}(t)=\partial_xh(t,\hat{X}(t),\hat{u}(t))+\int^\infty_te^{-\lambda(s-t)}A_1(s,t)^\top\bE_t\bigl[\hat{Y}_1(s)\bigr]\rd s+\int^\infty_te^{-\lambda(s-t)}A_3(s,t)^\top\bE_t\bigl[\hat{Y}_3(s)\bigr]\rd s\\
	&\hspace{1cm}+\partial_xb\Bigl(t,\hat{X}(t),\hat{u}(t),\int^t_0A_1(t,s)\hat{X}(s)\rd s,\int^t_0A_2(t,s)\hat{u}(s)\rd s\Bigr)^\top\int^\infty_te^{-\lambda(s-t)}\bE_t\bigl[\hat{Y}(s)\bigr]\rd s\\
	&\hspace{1cm}+\sum^d_{k=1}\partial_x\sigma^k\Bigl(t,\hat{X}(t),\hat{u}(t),\int^t_0A_3(t,s)\hat{X}(s)\rd s,\int^t_0A_4(t,s)\hat{u}(s)\rd s\Bigr)^\top\int^\infty_te^{-\lambda(s-t)}\hat{Z}^k(s,t)\rd s,
\end{split}
\end{equation}
and
\begin{equation}\label{Yi equation}
\begin{split}
	&\hat{Y}_1(t)=\partial_{x_1}b\Bigl(t,\hat{X}(t),\hat{u}(t),\int^t_0A_1(t,s)\hat{X}(s)\rd s,\int^t_0A_2(t,s)\hat{u}(s)\rd s\Bigr)^\top\int^\infty_te^{-\lambda(s-t)}\bE_t\bigl[\hat{Y}(s)\bigr]\rd s,\\
	&\hat{Y}_2(t)=\partial_{x_2}b\Bigl(t,\hat{X}(t),\hat{u}(t),\int^t_0A_1(t,s)\hat{X}(s)\rd s,\int^t_0A_2(t,s)\hat{u}(s)\rd s\Bigr)^\top\int^\infty_te^{-\lambda(s-t)}\bE_t\bigl[\hat{Y}(s)\bigr]\rd s,\\
	&\hat{Y}_3(t)=\sum^d_{k=1}\partial_{x_3}\sigma^k\Bigl(t,\hat{X}(t),\hat{u}(t),\int^t_0A_3(t,s)\hat{X}(s)\rd s,\int^t_0A_4(t,s)\hat{u}(s)\rd s\Bigr)^\top\int^\infty_te^{-\lambda(s-t)}\hat{Z}^k(s,t)\rd s,\\
	&\hat{Y}_4(t)=\sum^d_{k=1}\partial_{x_4}\sigma^k\Bigl(t,\hat{X}(t),\hat{u}(t),\int^t_0A_3(t,s)\hat{X}(s)\rd s,\int^t_0A_4(t,s)\hat{u}(s)\rd s\Bigr)^\top\int^\infty_te^{-\lambda(s-t)}\hat{Z}^k(s,t)\rd s,
\end{split}
\end{equation}
for a.e.\ $t\geq0$, a.s. Also, the optimality condition~\eqref{optimal condition auxiliary} can be rewritten as
\begin{align*}
	&\Bigl\langle \partial_uh(t,\hat{X}(t),\hat{u}(t))+\int^\infty_te^{-\lambda(s-t)}\bigl\{A_2(s,t)^\top\bE_t\bigl[\hat{Y}_2(s)\bigr]+A_4(s,t)^\top\bE_t\bigl[\hat{Y}_4(s)\bigr]\bigr\}\rd s\\
	&\hspace{1cm}+\partial_ub\Bigl(t,\hat{X}(t),\hat{u}(t),\int^t_0A_1(t,s)\hat{X}(s)\rd s,\int^t_0A_2(t,s)\hat{u}(s)\rd s\Bigr)^\top\int^\infty_te^{-\lambda(s-t)}\bE_t\bigl[\hat{Y}(s)\bigr]\rd s\\
	&\hspace{1cm}+\sum^d_{k=1}\partial_u\sigma^k\Bigl(t,\hat{X}(t),\hat{u}(t),\int^t_0A_3(t,s)\hat{X}(s)\rd s,\int^t_0A_4(t,s)\hat{u}(s)\rd s\Bigr)^\top\int^\infty_te^{-\lambda(s-t)}\hat{Z}^k(s,t)\rd s,\\
	&\hspace{5cm}u-\hat{u}(t)\Bigr\rangle\geq0,\ \forall\,u\in U.
\end{align*}
Letting $\hat{\cY}(t):=\int^\infty_te^{-\lambda(s-t)}\bE_t[\hat{Y}(s)]\rd s$ and $\hat{\cZ}(t):=\int^\infty_te^{-\lambda(s-t)}\hat{Z}(s,t)\rd s$ for $t\geq0$, by Lemma~\ref{lemm: BSVIE BSDE}, we see that the pair $(\hat{\cY}(\cdot),\hat{\cZ}(\cdot))$ is in $L^{2,-\mu}_\bF(0,\infty;\bR^n)\times L^{2,-\mu}_\bF(0,\infty;\bR^{n\times d})$ and satisfies the infinite horizon BSDE $\mathrm{d}\hat{\cY}(t)=-\{\hat{Y}(t)-\lambda\hat{\cY}(t)\}\rd t+\hat{\cZ}(t)\rd W(t)$, $t\geq0$. By inserting \eqref{Y equation} and \eqref{Yi equation} into this formula, we see that $(\hat{\cY}(\cdot),\hat{\cZ}(\cdot))$ solves the following \emph{infinite horizon anticipated BSDE} of an It\^{o}--Volterra type:
\begin{equation}\label{anticipated BSDE}
\begin{split}
	\mathrm{d}\hat{\cY}(t)&=-\Bigl\{\partial_xh(t,\hat{X}(t),\hat{u}(t))\\
	&\hspace{1.5cm}+\int^\infty_te^{-\lambda(s-t)}A_1(s,t)^\top\bE_t\Bigl[\partial_{x_1}b\Bigl(s,\hat{X}(s),\hat{u}(s),\int^s_0A_1(s,r)\hat{X}(r)\rd r,\\
	&\hspace{7.5cm}\int^s_0A_2(s,r)\hat{u}(r)\rd r\Bigr)^\top\hat{\cY}(s)\Bigr]\rd s\\
	&\hspace{1.5cm}+\sum^d_{k=1}\int^\infty_te^{-\lambda(s-t)}A_3(s,t)^\top\bE_t\Bigl[\partial_{x_3}\sigma^k\Bigl(s,\hat{X}(s),\hat{u}(s),\int^s_0A_3(s,r)\hat{X}(r)\rd r,\\
	&\hspace{8cm}\int^s_0A_4(s,r)\hat{u}(r)\rd r\Bigr)^\top\hat{\cZ}^k(s)\Bigr]\rd s\\
	&\hspace{1.5cm}+\partial_xb\Bigl(t,\hat{X}(t),\hat{u}(t),\int^t_0A_1(t,s)\hat{X}(s)\rd s,\int^t_0A_2(t,s)\hat{u}(s)\rd s\Bigr)^\top\hat{\cY}(t)\\
	&\hspace{1.5cm}+\sum^d_{k=1}\partial_x\sigma^k\Bigl(t,\hat{X}(t),\hat{u}(t),\int^t_0A_3(t,s)\hat{X}(s)\rd s,\int^t_0A_4(t,s)\hat{u}(s)\rd s\Bigr)^\top\hat{\cZ}^k(t)\Bigr\}\rd t\\
	&\hspace{0.5cm}+\lambda\hat{\cY}(t)\rd t+\hat{\cZ}(t)\rd W(t),\ t\geq0.
\end{split}
\end{equation}
Conversely, suppose that $(\hat{\cY}(\cdot),\hat{\cZ}(\cdot))\in L^{2,-\mu}_\bF(0,\infty;\bR^n)\times L^{2,-\mu}_\bF(0,\infty;\bR^{n\times d})$ satisfies the infinite horizon anticipated BSDE \eqref{anticipated BSDE}. Define $(\hat{Y}(\cdot),\hat{Z}(\cdot,\cdot))\in\cM^{2,-\mu}_\bF(0,\infty;\bR^n\times\bR^{n\times d})$ by
\begin{equation*}
	\begin{cases}
	\hat{Y}(t)=\partial_xh(t,\hat{X}(t),\hat{u}(t))\\
	\hspace{1.5cm}+\int^\infty_te^{-\lambda(s-t)}A_1(s,t)^\top\bE_t\Bigl[\partial_{x_1}b\Bigl(s,\hat{X}(s),\hat{u}(s),\int^s_0A_1(s,r)\hat{X}(r)\rd r,\\
	\hspace{7.5cm}\int^s_0A_2(s,r)\hat{u}(r)\rd r\Bigr)^\top\hat{\cY}(s)\Bigr]\rd s\\
	\hspace{1.5cm}+\sum^d_{k=1}\int^\infty_te^{-\lambda(s-t)}A_3(s,t)^\top\bE_t\Bigl[\partial_{x_3}\sigma^k\Bigl(s,\hat{X}(s),\hat{u}(s),\int^s_0A_3(s,r)\hat{X}(r)\rd r,\\
	\hspace{8cm}\int^s_0A_4(s,r)\hat{u}(r)\rd r\Bigr)^\top\hat{\cZ}^k(s)\Bigr]\rd s\\
	\hspace{1.5cm}+\partial_xb\Bigl(t,\hat{X}(t),\hat{u}(t),\int^t_0A_1(t,s)\hat{X}(s)\rd s,\int^t_0A_2(t,s)\hat{u}(s)\rd s\Bigr)^\top\hat{\cY}(t)\\
	\hspace{1.5cm}+\sum^d_{k=1}\partial_x\sigma^k\Bigl(t,\hat{X}(t),\hat{u}(t),\int^t_0A_3(t,s)\hat{X}(s)\rd s,\int^t_0A_4(t,s)\hat{u}(s)\rd s\Bigr)^\top\hat{\cZ}^k(t),\\
	\int^\infty_0\hat{Z}(t,s)\rd W(s)=\hat{Y}(t)-\bE\bigl[\hat{Y}(t)\bigr],
	\end{cases}
\end{equation*}
for $t\geq0$ (which implies that $\hat{Z}(t,s)=0$ for a.e.\ $(t,s)\in\Delta[0,\infty)$, a.s.). Then \eqref{anticipated BSDE} is written by $\mathrm{d}\hat{\cY}(t)=-\{\hat{Y}(t)-\lambda\hat{\cY}(t)\}\rd t+\hat{\cZ}(t)\rd W(t)$, $t\geq0$. By the uniqueness part of Lemma~\ref{lemm: BSVIE BSDE}, we obtain $\hat{\cY}(t)=\bE_t[\int^\infty_te^{-\lambda(s-t)}\hat{Y}(s)\rd s]$ and $\hat{\cZ}(t)=\int^\infty_te^{-\lambda(s-t)}\hat{Z}(s,t)\rd s$ for a.e.\ $t\geq0$, a.s. Thus, defining $\hat{Y}_i(\cdot)\in L^{2,-\mu}_\bF(0,\infty;\bR^{m_i})$, $i=1,2,3,4$, by \eqref{Yi equation}, we see that $\hat{Y}(\cdot)$ satisfies \eqref{Y equation} for a.e.\ $t\geq0$, a.s. Now we define $(\hat{\bY}(\cdot),\hat{\bZ}(\cdot,\cdot))\in\cM^{2,-\mu}_\bF(0,\infty;\bR^N\times\bR^{N\times d})$ by
\begin{equation*}
	\begin{cases}
	\hat{\bY}(t)=\left(\begin{array}{c}\hat{Y}(t)\\\hat{Y}_1(t)\\\hat{Y}_2(t)\\\hat{Y}_3(t)\\\hat{Y}_4(t)\end{array}\right),\\
	\int^\infty_0\hat{\bZ}(t,s)\rd W(s)=\partial_\bX \tilde{h}(t,\hat{\bX}(t),\hat{u}(t))+\int^\infty_te^{-\lambda(s-t)}\partial_\bX\tilde{b}(s,t,\hat{\bX}(t),\hat{u}(t))^\top\hat{\bY}(s)\rd s\\
	\hspace{3.2cm}+\sum^d_{k=1}\partial_\bX\tilde{\sigma}^k(t,\hat{\bX}(t),\hat{u}(t))^\top J^\top_n\hat{\cZ}(t)\\
	\hspace{4cm}-\bE\Bigl[\partial_\bX \tilde{h}(t,\hat{\bX}(t),\hat{u}(t))+\int^\infty_te^{-\lambda(s-t)}\partial_\bX\tilde{b}(s,t,\hat{\bX}(t),\hat{u}(t))^\top\hat{\bY}(s)\rd s\\
	\hspace{4.8cm}+\sum^d_{k=1}\partial_\bX\tilde{\sigma}^k(t,\hat{\bX}(t),\hat{u}(t))^\top J^\top_n\hat{\cZ}(t)\Bigr],
	\end{cases}
\end{equation*}
for $t\geq0$. Then $(\hat{\bY}(\cdot),\hat{\bZ}(\cdot,\cdot))$ satisfies \eqref{adjoint equation integro-differential}, and it holds that
\begin{equation}\label{YZ anticipated BSDE}
	\hat{\cY}(t)=\bE_t\Bigl[\int^\infty_te^{-\lambda(s-t)}J_n\hat{\bY}(s)\rd s\Bigr]\ \text{and}\ \hat{\cZ}(t)=\int^\infty_te^{-\lambda(s-t)}J_n\hat{\bZ}(s,t)\rd s
\end{equation}
for a.e.\ $t\geq0$, a.s. By the uniqueness of the adapted M-solution to the infinite horizon BSVIE~\eqref{adjoint equation integro-differential}, we see that, for each control process $\hat{u}(\cdot)\in\cU_{-\mu}$ and the corresponding state process $\hat{X}(\cdot)=X^{\hat{u}}(\cdot)\in L^{2,-\mu}_\bF(0,\infty;\bR^n)$, the infinite horizon anticipated BSDE~\eqref{anticipated BSDE} admits a unique adapted solution $(\hat{\cY}(\cdot),\hat{\cZ}(\cdot))\in L^{2,-\mu}_\bF(0,\infty;\bR^n)\times L^{2,-\mu}_\bF(0,\infty;\bR^{n\times d})$ given by \eqref{YZ anticipated BSDE}. Furthermore, the optimality condition~\eqref{optimal condition auxiliary} is equivalent to the following condition:
\begin{equation}\label{integro-differential optimal}
\begin{split}
	&\Bigl\langle \partial_uh(t,\hat{X}(t),\hat{u}(t))+\int^\infty_te^{-\lambda(s-t)}A_2(s,t)^\top\bE_t\Bigl[\partial_{x_2}b\Bigl(s,\hat{X}(s),\hat{u}(s),\int^s_0A_1(s,r)\hat{X}(r)\rd r,\\
	&\hspace{9.5cm}\int^s_0A_2(s,r)\hat{u}(r)\rd r\Bigr)^\top\hat{\cY}(s)\Bigr]\rd s\\
	&\hspace{0.5cm}+\sum^d_{k=1}\int^\infty_te^{-\lambda(s-t)}A_4(s,t)^\top\bE_t\Bigl[\partial_{x_4}\sigma^k\Bigl(s,\hat{X}(s),\hat{u}(s),\int^s_0A_3(s,r)\hat{X}(r)\rd r,\\
	&\hspace{8cm}\int^s_0A_4(s,r)\hat{u}(r)\rd r\Bigr)^\top\hat{\cZ}^k(s)\Bigr]\rd s\\
	&\hspace{0.5cm}+\partial_ub\Bigl(t,\hat{X}(t),\hat{u}(t),\int^t_0A_1(t,s)\hat{X}(s)\rd s,\int^t_0A_2(t,s)\hat{u}(s)\rd s\Bigr)^\top\hat{\cY}(t)\\
	&\hspace{0.5cm}+\sum^d_{k=1}\partial_u\sigma^k\Bigl(t,\hat{X}(t),\hat{u}(t),\int^t_0A_3(t,s)\hat{X}(s)\rd s,\int^t_0A_4(t,s)\hat{u}(s)\rd s\Bigr)^\top\hat{\cZ}^k(t),u-\hat{u}(t)\Bigr\rangle\geq0,\ \forall\,u\in U.
\end{split}
\end{equation}
Consequently, by Theorem~\ref{theo: maximum principle}, if $\hat{u}(\cdot)$ is optimal, then \eqref{integro-differential optimal} holds for a.e.\ $t\geq0$, a.s. Conversely, by Theorem~\ref{theo: sufficient maximum principle}, if the map
\begin{align*}
	(x,x_1,x_2,x_3,x_4,u)\mapsto &h(t,x,u)+\bigl\langle b(t,x,u,x_1,x_2),\hat{\cY}(t)\bigr\rangle+\sum^d_{k=1}\bigl\langle\sigma^k(t,x,u,x_3,x_4),\hat{\cZ}^k(t)\bigr\rangle
\end{align*}
is convex in $(x,x_1,x_2,x_3,x_4,u)\in\bR^n\times\bR^{m_1}\times\bR^{m_2}\times\bR^{m_3}\times\bR^{m_4}\times U$ for a.e.\ $t\geq0$, a.s., and if \eqref{integro-differential optimal} holds for a.e.\ $t\geq0$, a.s., then $\hat{u}(\cdot)$ is optimal.
\end{exam}

\section*{Acknowledgments}
The author was supported by JSPS KAKENHI Grant Number 21J00460.


\begin{thebibliography}{99}

\bibitem{AbMiPh19}
{Abi Jaber, E.}, {Miller, E.}, and {Pham, H.}
{Linear--Quadratic control for a class of stochastic Volterra equations: solvability and approximation.}
{\it preprint.}
{arXiv:1911.01900.}

\bibitem{Ag19}
{Agram, N.}
(2019).
{Dynamic risk measure for BSVIE with jumps and semimartingale issues.}
{\it Stoch. Anal. Appl.}
{\bf 37}(3)
361--376.

\bibitem{AgOk15}
{Agram, N.} and {{\O}ksendal, B.}
(2015).
{Mallivain calculus and optimal control of stochastic Volterra equations.} 
{\it J. Optim. Theory Appl.}
{\bf 167}
1070--1094.

\bibitem{Ap03}
{Appleby, J.}
(2003).
{$p$th mean integrability and almost sure asymptotic stability of solutions of It\^{o}--Volterra equations.}
{\it J. Integral Equ. Appl.}
{\bf 15}(4)
321--341.


%

\bibitem{BeRG21}
{Beissner, P.} and {Rosazza Gianin, E.}
(2021).
{The term structure of Sharpe ratios and arbitrage-free asset pricing in continuous time.}
{\it Probab. Uncertain. Quantit. Risk}
{\bf 6}(1)
23-52.

%
%

\bibitem{BeMi80}
{Berger, M.~A.} and {Mizel, V.~J.}
(1980).
{Volterra equations with It\^o integrals---\Rnum{1}.}
{\it J. Integral Equations}
{\bf 2}(3)
187--245.

%

%
%

\bibitem{ChYo07}
{Chen, S.} and {Yong, J.}
(2007).
{A linear quadratic optimal control problem for stochastic Volterra integral equations.}
{\it Control theory and related topics – in memory of professor Xunjing Li, Fudan university, China.}
44--66.

\bibitem{Cr85}
{Cromer, T. L.}
(1985).
{Asymptotically periodic solutions to Volterra integral equations in epidemic models.}
{\it J. Math. Anal. Appl.}
{\bf 110}(2)
483--494.


%

\bibitem{DoHeHe11}
{Dos Santos, J. P. C.}, {Henr\'{i}quez, H.}, and {Hern\'{a}ndez, E.}
(2011).
{Existence results for neutral integro-differential equations with unbounded delay.}
{\it J. Integral Equ. Appl.}
{\bf 23}
289--330.

%

%

\bibitem{FuHu06}
{Fuhrman, M.} and {Hu, Y.}
(2006).
{Infinite horizon BSDEs in infinite dimensions with continuous driver and applications.}
{\it J. Evol. Equ.}
{\bf 6}
459--484.

\bibitem{FuTe04}
{Fuhrman, M.} and {Tessitore, G.}
(2004).
{Infinite horizon backward stochastic differential equations and elliptic equations in Hilbert spaces.}
{\it Ann. Probab.}
{\bf 32}
607--660.

\bibitem{GaJaRo18}
{Gatheral, J.}, {Jaisson, T.}, and {Rosenbaum, M.}
(2018).
{Volatility is rough.}
{\it Quant. Finance}
{\bf 18}(6)
933--949.

%
%
%
%
%


\bibitem{Ha21}
{Hamaguchi, Y.}
(2021).
{Extended backward stochastic Volterra integral equations and their applications to time-inconsistent stochastic recursive control problems.}
{\it Math. Control Relat. Fields}
{\bf 11}(2)
197--242.


\bibitem{HePo20}
{Hern\'{a}ndez, C.} and {Possama\"{i}, D.}
{Me, myself and I: a general theory of non-Markovian time-inconsistent stochastic control for sophisticated agents.}
{\it preprint.}
arXiv:2002.12572.



\bibitem{Hi00}
{Hilfer, H.}
(2000).
{\it Applications of Fractional Calculus in Physics.}
World Scientific, Singapore.


\bibitem{HuOk19}
{Hu, Y.} and {{\O}ksendal, B.}
(2019).
{Linear Volterra backward stochastic integral equations.}
{\it Stochastic Process. Appl.}
{\bf 129}(2)
626--633.

\bibitem{KiSrTr06}
{Kilbas, A. A.}, {Srivastava, H. M.}, and {Trujillo, J. J.}
(2006).
{\it Theory and Applications of Fractional Differential Equations.}
North-Holland, Amsterdam.

%

\bibitem{KrOv17}
{Kromer, E.} and {Overbeck, L.}
(2017).
{Differentiability of BSVIEs and dynamic capital allocations.}
{\it Int. J. Theor. Appl. Finance}
{\bf 20}(7)
1--26.
%

\bibitem{Li02}
{Lin, J.}
(2002).
{Adapted solution of a backward stochastic nonlinear Volterra integral equation.}
{\it Stoch. Anal. Appl.}
{\bf 20}(1)
165--183.

\bibitem{LiYo20}
{Lin, P.} and {Yong, J.}
(2020)
{Controlled singular Volterra integral equations and Pontryagin maximum principle.}
{\it SIAM J. Control Optim.}
{\bf 58}(1)
136--164.

%

\bibitem{MaYo99}
{Ma, J.} and {Yong, J.}
(1999).
{\it Forward-Backward Stochastic Differential Equations and Their Applications.}
Lecture Notes in Math. 1702, Springer-Verlag, New York.

\bibitem{MaRi06}
{Mao, X.} and {Riedle, M.}
(2006).
{Mean square stability of stochastic Volterra integro-differential equations.}
{\it Systems Control Letters}
{\bf 55}(6)
459--465.

\bibitem{MaVe14}
{Maslowski, B.} and {Veverka, P.}
(2014).
{Sufficient stochastic maximum principle for discounted control problem.}
{\it Appl. Math. Optim.}
{\bf 70}
225--252.

%
%

\bibitem{OrVe17}
{Orrieri, C.} and {Veverka, P.}
(2017).
{Necessary stochastic maximum principle for dissipative systems on infinite time horizon.}
{\it ESAIM Control Optim. Calc. Var.}
{\bf 23}
337--371.

\bibitem{PaPa20}
{Pang, G.} and {Pardoux, E.}
{Functional limit theorems for non-Markovian epidemic models.}
{\it preprint.}
arXiv:2003.03249.

\bibitem{Pa99}
{Pardoux, E.}
(1999).
{BSDEs weak convergence and homogenizations of semilinear PDEs.}
In: {\it Clark, F.H., Stern, R.J. (eds.) Nonlinear Analysis Differential Equations and Control}.
503--509.

%
%
%


\bibitem{PeSh00}
{Peng, S.} and {Shi, Y.}
(2000).
{Infinite horizon forward-backward stochastic differential equations.}
{\it Stochastic Process. Appl.}
{\bf 85}
75--92.

%

%


\bibitem{SaNiMa10}
{Sakthivel, R.}, {Nieto, J. J.}, and {Mahmudov, N. I.}
(2010).
{Approximate controllability of nonlinear deterministic and stochastic systems with unbounded delay.}
{\it Taiwan J. Math.}
{\bf 14}(5)
1777--1797.

%

\bibitem{ShWaYo13}
{Shi, Y.}, {Wang, T.}, and {Yong, J.}
(2013).
{Mean-field backward stochastic Volterra integral equations.}
{\it Discrete Contin. Dyn. Syst.}
{\bf 18}(7)
1929--1967.

\bibitem{ShWaYo15}
{Shi, Y.}, {Wang, T.}, and {Yong, J.}
(2015).
{Optimal control problems of forward-backward stochastic Volterra integral equations.}
{\it Math. Control Relat. Fields}
{\bf 5}(3)
613--649.

%


\bibitem{WaSuYo19}
{Wang, H.}, {Sun, J.}, and {Yong, J.}
(2019).
{Recursive utility processes, dynamic risk measures and quadratic backward stochastic Volterra integral equations.}
{\it Appl. Math. Optim.}
1--46.


\bibitem{WaYo21}
{Wang, H.} and {Yong, J.}
(2021).
{Time-inconsistent stochastic optimal control problems and backward stochastic Volterra integral equations.}
{\it ESAIM Control Optim. Calc. Var.}
{\bf 27}(22).

\bibitem{WaYoZh20}
{Wang, H.}, {Yong, J.}, and {Zhang, J.}
{Path dependent Feynman--Kac formula for forward backward stochastic Volterra integral equations.}
To appear in {\it Ann. Inst. Henri Poincar\'e Probab. Stat.}
arXiv:2004.05825.
%



\bibitem{WaT18}
{Wang, T.}
(2018).
{Linear quadratic control problems of stochastic Volterra integral equations.}
{\it ESAIM: Control Optim. Cal. Var.}
{\bf 24}(4)
1849--1879.


\bibitem{WaT20}
{Wang, T.}
(2020).
{Necessary conditions of Pontraygin's type for general controlled stochastic Volterra integral equations.}
{\it ESAIM: Control Optim. Cal. Var.}
{\bf 26}(16).

%

\bibitem{WaTYo15}
{Wang, T.} and {Yong, J.}
(2015).
{Comparison theorems for some backward stochastic Volterra integral equations.}
{\it Stochastic Process. Appl.}
{\bf 125}
1756--1798.

\bibitem{WaTZh17}
{Wang, T.} and {Zhang, H.}
(2017).
{Optimal control problems of forward-backward stochastic Volterra integral equations with closed control regions.}
{\it SIAM J. Control Optim.}
{\bf 55}(4)
2574--2602.

\bibitem{WaXuKl16}
{Wang, Y.}, {Xu, J.}, and {Kloeden, P. E.}
(2016).
{Asymptotic behavior of stochastic lattice systems with a Caputo fractional time derivative.}
{\it Nonlinear Anal.}
{\bf 135}
205--222.


%


\bibitem{WuZh12}
{Wu, Z.} and {Zhang, F.}
(2012).
{Maximum principle for stochastic recursive optimal control problems involving impulse controls.}
{\it Abstr. Appl. Anal.}
{\bf 32}
1--16.



\bibitem{Yi08}
{Yin, J.}
(2008).
{On solutions of a class of infinite horizon FBSDE's.}
{\it Stat. Probab. Lett.}
{\bf 78}
2412--2419.

\bibitem{Yo06}
{Yong, J.}
(2006).
{Backward stochastic Volterra integral equations and some related problems.}
{\it Stoch. Anal. Appl.}
{\bf 116}(5)
779--795.

\bibitem{Yo07}
{Yong, J.}
(2007).
{Continuous-time dynamic risk measures by backward stochastic Volterra integral equations.}
{\it Appl. Anal.}
{\bf 86}(11)
1429--1442.

\bibitem{Yo08}
{Yong, J.}
(2008).
{Well-posedness and regularity of backward stochastic Volterra integral equations.}
{\it Probab. Theory Related Fields}
{\bf 142}(1-2)
2--77.
%

%

\bibitem{Zh17}
{Zhang, J.}
(2017).
{\it Backward Stochastic Differential Equations; From Linear to Fully Nonlinear Theory.}
Springer.


\end{thebibliography}

\end{document}